\documentclass[12pt]{article}
\title{Wronskians, cyclic group actions, \\ and ribbon tableaux}
\author{Kevin Purbhoo%
\footnote{Department of Combinatorics and Optimization, University of Waterloo,
Waterloo, Ontario, Canada; {\tt kpurbhoo@math.uwaterloo.ca}.}
}

\usepackage{latexsym, amssymb, amsthm, amsmath}
\usepackage{graphicx, xspace, ifthen, rotating}
\usepackage[svgnames]{xcolor}

\usepackage{young}
\ysetshade{Red!50}
\ysetaltshade{Blue!28}
\YSetShade{Green!18}
\YSetAltShade{Yellow!36}
\YSETSHADE{Purple!31}
\YSETALTSHADE{Orange!37}

\newlength\circlesize
\allowdisplaybreaks[1]

\newcommand{\bigmid}{\ \big|\ }
\newcommand{\Bigmid}{\ \Big|\ }

\newcommand{\middlemid}{\ \middle|\ }

\newcommand{\CC}{\mathbb{C}}
\newcommand{\FF}{\mathbb{F}}
\newcommand{\ZZ}{\mathbb{Z}}
\newcommand{\QQ}{\mathbb{Q}}
\newcommand{\RR}{\mathbb{R}}
\newcommand{\CP}{\mathbb{CP}}

\newcommand{\RP}{\mathbb{RP}}
\newcommand{\PP}{\mathbb{P}}

\newcommand{\Hy}{\hat y}
\newcommand{\Tlambda}{\widetilde \lambda}
\newcommand{\Tmu}{\widetilde \mu}
\newcommand{\Tnu}{\widetilde \nu}
\newcommand{\Tsigma}{\widetilde \sigma}
\newcommand{\Tx}{\widetilde x}
\newcommand{\Tw}{\widetilde w}
\newcommand{\Tboldw}{\widetilde \boldw}
\newcommand{\TF}{\widetilde F}

\newcommand{\nth}{\ensuremath{^\text{th}}\xspace}

\newcommand{\puiseux}[1]{\CC\{\!\{#1\}\!\}}
\newcommand{\Fpuiseux}[1]{\FF\{\!\{#1\}\!\}}
\newcommand{\psK}{\mathbb{K}}
\newcommand{\mirrc}{\mathcal{R}}
\newcommand{\calU}{\mathcal{U}}
\newcommand{\orbit}{\mathcal{O}}
\newcommand{\Orbits}{\mathrm{Orb}}

\newcommand{\poln}{\CC_{n-1}[z]}
\newcommand{\Rpoln}{\RR_{n-1}[z]}
\newcommand{\Kpoln}{\psK_{n-1}[z]}

\newcommand{\pol}[1]{\CC_{#1}[z]}
\newcommand{\Rpol}[1]{\RR_{#1}[z]}
\newcommand{\Kpol}[1]{\psK_{#1}[z]}

\newcommand{\PpolN}{\PP(\pol{N})}
\newcommand{\KPpolN}{\PP(\psK_N[z])}

\newcommand{\SL}{\mathrm{SL}}
\newcommand{\PGL}{\mathrm{PGL}}
\newcommand{\Wr}{\mathrm{Wr}}
\newcommand{\Gr}{\mathrm{Gr}}
\newcommand{\Rect}{{\mathchoice%
{\raisebox{.05ex}{$\sqsubset\!\!\sqsupset$}}
{\raisebox{.05ex}{$\sqsubset\!\!\sqsupset$}}
{\sqsubset\!\!\sqsupset}
{\sqsubset\!\!\sqsupset}
}}
\newcommand{\smallrect}{{\mathchoice%
   {\raisebox{.05ex}{\footnotesize $\sqsubset\!\!\!\sqsupset$}}
   {\raisebox{.05ex}{\footnotesize $\sqsubset\!\!\!\sqsupset$}}
   {\raisebox{.1ex}{\tiny $\sqsubset\!\!\!\sqsupset$}}
   {\raisebox{.1ex}{\tiny $\sqsubset\!\!\!\sqsupset$}}
}}

\newcommand{\SYT}{\mathsf{SYT}}
\newcommand{\ribbon}{\mathsf{SRT}^r}
\newcommand{\ribbonplain}{\mathsf{SRT}}
\newcommand{\rotribbon}{\mathsf{RRT}^r}
\newcommand{\rotribbonplain}{\mathsf{RRT}}
\newcommand{\longribbon}{\widehat{\mathsf{RRT}}\,\!^r}
\newcommand{\longribbonplain}{\widehat{\mathsf{RRT}}\,\!}

\newcommand{\bolda}{{\bf a}}
\newcommand{\boldb}{{\bf b}}

\newcommand{\bolds}{{\bf s}}
\newcommand{\boldp}{{\bf p}}
\newcommand{\boldw}{{\bf w}}

\newcommand{\boldlambda}{{\mathchoice
{\mbox{\boldmath{$\lambda$}}}
{\mbox{\boldmath{$\lambda$}}}
{\mbox{\scriptsize\boldmath{$\lambda$}}}
{\mbox{\scriptsize\boldmath{$\lambda$}}}
}}
\newcommand{\boldalpha}{{\mathchoice
{\mbox{\boldmath{$\alpha$}}}
{\mbox{\boldmath{$\alpha$}}}
{\mbox{\scriptsize\boldmath{$\alpha$}}}
{\mbox{\scriptsize\boldmath{$\alpha$}}}
}}

\newcommand{\smallidmatrix}
{\left(\begin{smallmatrix} 1 & 0 \\ 0 & 1\end{smallmatrix}\right)}
\newcommand{\stdreflect}
{\left(\begin{smallmatrix} 0 & 1 \\ 1 & 0\end{smallmatrix}\right)}
\newcommand{\altreflect}
{\left(\begin{smallmatrix} -1 & 0 \\ 0 & 1\end{smallmatrix}\right)}
\newcommand{\halfrotate}
{\left(\begin{smallmatrix} 0 & -1 \\ 1 & 0\end{smallmatrix}\right)}
\newcommand{\calT}{\mathcal{T}}

\newcommand{\mindeg}{\mathop{\mathrm{mindeg}}}

\newcommand{\rcore}{\mathrm{core}^r}
\newcommand{\rspec}{\mathrm{spec}^r}

\newcommand{\switch}{\mathsf{switch}}

\newcommand{\promote}{\textbf{\textit{j}}}
\newcommand{\evac}{\textbf{\textit{e}}}
\newcommand{\val}{\mathrm{val}}
\newcommand{\leadterm}{\text{\rm \footnotesize LT}}
\newcommand{\leadcoeff}{\text{\rm \footnotesize LC}}
\newcommand{\initial}{\mathrm{In}}
\newcommand{\weight}{\mathrm{wt}}

\newcommand{\identity}{\mathrm{id}}
\newcommand{\sign}{\mathrm{sign}}

\newcommand{\richvar}[1]{\widetilde X_{\widetilde\lambda_{#1}/\widetilde\mu_{#1}}}
\newcommand{\richid}[1]{I_{\widetilde\lambda_{#1}/\widetilde\mu_{#1}}}

\newenvironment{packedenum}{
\begin{enumerate}
  \setlength{\itemsep}{0pt}
}{\end{enumerate}}

\newtheorem{lemma}{Lemma}[section]
\newtheorem{theorem}[lemma]{Theorem}

\newtheorem{proposition}[lemma]{Proposition}
\newtheorem{corollary}[lemma]{Corollary}

\theoremstyle{definition}
\newtheorem{example}[lemma]{Example}

\newtheorem{remark}[lemma]{Remark}
\newtheorem{condition}[lemma]{Condition}

\numberwithin{equation}{section}
\numberwithin{figure}{section}
\numberwithin{table}{section}

\definecolor{DarkBlue}{rgb}{0, 0.1, 0.55}
\definecolor{DarkRed}{rgb}{0.45, 0, 0}
\newcommand{\defn}[1]{\textbf{#1}}

\begin{document}
\maketitle

\begin{abstract}
The Wronski map is a finite, $\PGL_2(\CC)$-equivariant morphism from 
the Grassmannian $\Gr(d,n)$ to a projective space (the projectivization
of a vector space of polynomials).
We consider the following problem.
If $C_r \subset \PGL_2(\CC)$ is a cyclic subgroup 
of order $r$, how may $C_r$-fixed points are in the
in a fibre of the Wronski map over a $C_r$-fixed point in the base?

In this paper, we compute a general answer in terms of $r$-ribbon 
tableaux.
When $r=2$, this computation gives the number of \emph{real} points 
in the fibre of the Wronski map over a real polynomial
with purely imaginary roots.
More generally, we can compute the number of real points in
certain intersections of Schubert varieties.

When $r$ divides $d(n-d)$ our main result says that 
the generic number of $C_r$-fixed points in the fibre is
the number of standard $r$-ribbon tableaux rectangular 
shape $(n{-}d)^d$.
Computing by a different method, we show that the answer in this
case is also given by the number of of standard Young tableaux 
of shape $(n{-}d)^d$ that are invariant under $\frac{N}{r}$ 
iterations of jeu de taquin promotion.  Together, these two results
give a new proof of Rhoades' cyclic sieving theorem for promotion 
on rectangular tableaux.  

We prove analogous results for dihedral group actions.
\end{abstract}


\section{Introduction}


\subsection{Wronskians}
\label{sec:introwronskian}

The Wronski map is a finite morphism from a Grassmannian to 
a projective space of the same dimension, which is remarkably
well-behaved.
In this paper, we consider the problem of counting points 
in a fibre of the Wronski map that are fixed by the action 
of a cyclic group.  Our motivation comes from questions in 
real algebraic geometry and combinatorics, which shall be discussed
later in the introduction.  We begin by defining the Wronski map
and recalling some of the of properties that make it an interesting 
object of study in enumerative geometry.

For any non-negative integer $m$, let $\pol{m}$ denote the vector space of 
polynomials of degree at most $m$.
Fix integers $0 < d < n$, and let
$X := \Gr(d,\poln)$ denote the Grassmannian of $d$-planes in the
$n$-dimensional vector space $\poln$.
The dimension of $X$ is $N := d(n-d)$.  

If $x \in X$ is spanned by polynomials
$f_1(z), \dots, f_d(z) \in \poln$, we consider the Wronskian
\begin{equation}
\label{eqn:wronskian}
  \Wr(x;z) :=
  \begin{vmatrix}
  f_1(z) & \cdots & f_d(z)\\
  f_1'(z) & \cdots  & f_d'(z) \\
  \vdots &  \vdots & \vdots \\, 
  f_1^{(d-1)}(z) & \cdots & f_d^{(d-1)}(z)
  \end{vmatrix}\,.
\end{equation}
It is easy to show that $\Wr(x;z)$ is a non-zero polynomial
of degree at most $N$, and
different choices of basis for $x$ give the same Wronskian up to a scalar 
multiple.  Thus $\Wr(x;z)$ determines a well-defined element of 
projective space $\PpolN$.  We adopt the convention that 
equations involving $\Wr(x;z)$ hold only up to a scalar multiple,
and should be properly interpreted in projective space.
The morphism $\Wr : X \to \PpolN$, $x \mapsto \Wr(x;z)$ is called 
the \defn{Wronski map}.  
There are many equivalent formulations.
For example, the map $z \mapsto [f_1(z) : \dots : f_d(z)]$ defines
a rational curve $\gamma : \CP^1 \to \CP^{d-1}$; the roots $\Wr(x;z)$
are the \emph{ramification points} of $\gamma$.
We refer the reader to the survey article~\cite{Sot-F} for a
detailed background.

Eisenbud and Harris showed that the Wronski map is a finite morphism
\cite{EH}, of degree is equal to the number of standard Young tableaux
of rectangular shape $\Rect = (n{-}d)^d$.   
By studying the asymptotic behaviour of the fibres, 
Eremenko and Gabrielov refined this result \cite{EG-deg}. 
They described a bijection 
between the set $\SYT(\Rect)$ of all standard Young tableaux of 
shape $\Rect$,
and points in the fibre $X(h(z)) := \Wr^{-1}(h(z))$ for certain polynomials
$h(z) \in \PpolN$.

More generally, one can consider the restriction of the
Wronski map to a Richardson variety.
Let $\Lambda$ denote the set of all partitions 
$\lambda : \lambda^1 \geq \lambda^2 \geq \dots \geq \lambda^d \geq 0$,
with at most $d$ parts and $\lambda^1 \leq n-d$,
and consider the complete flags 
\[
\begin{gathered}
F_\bullet : F_0 \subset F_1 \subset \dots \subset F_n \\
F_i = z^{n-i}\pol{i-1}
\end{gathered}
\qquad\qquad \text{and} \qquad\qquad
\begin{gathered}
\TF_\bullet : \TF_0 \subset \TF_1 
\subset \dots \subset \TF_n \\
\TF_i = \pol{i-1}\,.
\end{gathered}
\]
For partitions $\lambda \supset \mu$ in $\Lambda$, the 
\defn{Richardson variety}
\[
   X_{\lambda/\mu} 
   := \left\{x \in X \middlemid
   \begin{aligned}
   &\dim(x \cap F_{n-d-\mu^i+i}) \geq i & & \quad\text{and}\\
   &\dim(x \cap \TF_{\lambda^{d-i}+i}) \geq i,
    & & \quad\text{for }i = 1, \dots, d
   \end{aligned}
   \right\}\,.
\]
is the intersection of a Schubert variety (corresponding to $\mu$)
and and opposite Schubert variety (corresponding to $\lambda$).
We will say that $X_{\lambda/\mu}$ is \defn{compatible}
with a polynomial $h(z)$ if $|\lambda| = \deg h(z)$ and 
$|\mu| = \mindeg h(z) :=  
\max\{m \geq 0 \mid z^m \text{ divides }h(z)\}$.
This definition is motivated by the following fact.

\begin{proposition}[See Section \ref{sec:plucker}]
\label{prop:richardson}
Every point in $X(h(z))$
lies in a unique $h(z)$-compatible Richardson variety.
\end{proposition}

If we restrict the Wronski map to any Richardson variety
$X_{\lambda/\mu}$, we obtain a finite morphism over the closure of
of the space of compatible polynomials.  These restrictions interact
nicely with the combinatorics of tableaux.  The fibres are generically
reduced, and the degree of this map is equal to
the number of standard Young tableaux of skew shape $\lambda/\mu$.
The correspondence of Eremenko and Gabrielov extends to this situation,
giving a bijection for specially chosen polynomials.

One of the most remarkable facts about the Wronski map is the reality
theorem of Mukhin, Tarasov and Varchenko (formerly the Shapiro-Shapiro
conjecture).
We say that a point $x \in X$ is \defn{real} if the subspace of
$\poln$ represented by $x$ has a basis
$f_1(z), \dots, f_d(z) \in \Rpoln$.  

\begin{theorem}[Mukhin-Tarasov-Varchenko \cite{MTV1, MTV2}]
\label{thm:MTV}
Let $h(z) \in \Rpol{N}$ be a polynomial of degree $N$ or $N-1$ whose
roots are distinct real numbers.
Then the fibre $X(h(z))$ is reduced and every point in 
the fibre is real.
\end{theorem}

Apart from being a very striking and surprising statement --- one does
not normally expect an algebraic system with real parameters
to have only real solutions ---
this theorem has interesting combinatorial consequences.  
For example, 
it can be used to show
that the correspondence of Eremenko and Gabrielov can be unambiguously
extended to polynomials with real roots.  This extension
is not continuous, however it is discontinuous in an interesting way.  
In~\cite{Pur-Gr} we showed that the discontinuities encode the 
jeu de taquin theory on tableaux.  Many of the results and methods 
from this previous work will be used in our present paper.

\subsection{Cyclic group actions}

The Mukhin-Tarasov-Varchenko theorem does not hold if
$h(z) \in \Rpol{N}$ is a real polynomial with complex roots.  
Computer experiments by 
Hauenstein, Hein, Mart\'in del Campo and Sottile
 show that in general, the number of real points in $X(h(z))$
depends in a complicated way on the polynomial $h(z)$, and not just
on the number of real roots~\cite{HHMS}.  However, there are other 
situations where the number of real solutions is predictable.
The work in this paper was originally motivated by the following problem.
Suppose $h(z) \in \Rpol{N}$ is a real polynomial of degree at most $N$
with only \emph{pure imaginary} roots.  How many real points are
in the fibre $X(h(z))$?  Special cases of this problem were 
studied by Eremenko and Gabrielov, and used to construct instances
of the pole placement problem in control theory with no real 
solutions~\cite{EG-pole}.

To answer this (and other related questions), it is helpful to observe 
that the Wronski map is equivariant with respect to a $\PGL_2(\CC)$-action.
The group $\SL_2(\CC)$ acts on each vector space 
$\pol{m}$ by M\"obius transformations:
If $\phi = 
\left(\begin{smallmatrix} 
\phi_{11} & \phi_{12} \\ \phi_{21} & \phi_{22}
\end{smallmatrix}\right) \in \SL_2(\CC)$,
we let
\[
\phi f(z) := (\phi_{21} z + \phi_{11})^m 
f\big(\frac{\phi_{22} z + \phi_{12}}{\phi_{21} z + \phi_{11}}\big)
\]
for $f(z) \in \pol{m}$.
The action $\SL_2(\CC)$ on $\poln$ and $\pol{N}$ induces an action 
of $\PGL_2(\CC)$ on $X$, and on $\PpolN$, and the Wronski map
intertwines these two $\PGL_2(\CC)$-actions.  

The involution $z \mapsto -z$ is given by
$\phi = \altreflect \in \PGL_2(\CC)$,
and (up to a sign) fixes every real polynomial $h(z)$ with 
only imaginary roots.  
Hence for these polynomials, $\phi$ gives an involution on the 
fibre $X(h(z))$.  One can then show, using the
Mukhin-Tarasov-Varchenko theorem, that the fixed points of this 
involution are exactly the real points in the fibre.  We are left
with the problem of counting the number of $\phi$-fixed points
in the fibre over a $\phi$-fixed polynomial.

More generally, let $C_r \subset \PGL_2(\CC)$ be a cyclic subgroup of 
order $r$.  We can consider the problem of counting $C_r$-fixed points
in the fibre of a $C_r$-fixed polynomial in $\PpolN$.
As any two are of these cyclic subgroups are conjugate, we will 
generally take 
\[
  C_r :=  \left\{
  \begin{pmatrix}
  e^{\pi i j/r} & 0 \\
  0 & e^{-\pi i j/r}
  \end{pmatrix}
  \middlemid j=0, \dots, r-1\right\}
\,.
\]
Then, a polynomial $h(z)$ is $C_r$-fixed as an element of $\PpolN$ if
and only if it is of the form
\begin{equation}
\label{eqn:fixedpoly}
h(z) = z^m(z^r+h_1)(z^r+h_2) \dotsb (z^r+h_\ell)\,.
\end{equation}
where $h_1, \dots, h_\ell \in \CC^\times$, $m \geq 0$ and $m + r\ell \leq N$.
When we speak of a polynomial being $C_r$-fixed, we will always
mean that it is fixed as an element of the appropriate projective
space; indeed this is the only reasonable interpretation, since 
$C_r$ is a subgroup of $\PGL_2(\CC)$, not of $\SL_2(\CC)$.
 
Let $X^r$ denote the $C_r$-fixed subscheme of $X$.  For any
$h(z)$-compatible Richardson variety $X_{\lambda/\mu}$, we consider
\[
  X^r_{\lambda/\mu}(h(z)) 
  := X_{\lambda/\mu} \cap X^r \cap X(h(z))\,,
\]
the $C_r$-fixed points of the fibre of the Wronski map, 
that lie in the Richardson variety $X_{\lambda/\mu}$.

Our main result is a combinatorial formula for the number of 
points in $X^r_{\lambda/\mu}(h(z))$: we show that there is
a bijection between these points and the set of
of \emph{standard $r$-ribbon tableaux} of shape $\lambda/\mu$.
When $r=1$, these are just standard Young tableaux, and the group is
trivial; hence we recover the fact that $\Wr|_{X_{\lambda/\mu}}$
has degree $|\SYT(\lambda/\mu)|$.
The problem of counting real solutions
when $h(z)$ is a real polynomial with pure imaginary roots 
corresponds to the special case of an involution, i.e. $r=2$.  
In the case where $r$ divides $N$, $\mindeg h(z)=0$ and $\deg h(z) = N$, 
we can determine the answer in a second way --- in terms of the
jeu de taquin \emph{promotion} operator on rectangular tableaux --- 
giving a non-trivial combinatorial identity.
These results can also be extend to dihedral group actions.

\subsection{Ribbon tableaux}

We now recall the definitions of the combinatorial objects that are
relevant to the statements of our results.  We assume some
basic familiarity here, and refer the reader 
to \cite{Ful, Sag-sym, Sta} for a more thorough introduction.

Let $\lambda, \mu \in \Lambda$ be partitions, with $\lambda \supset \mu$.
By a \defn{tableau} of shape $\lambda/\mu$, we will mean a filling
of the boxes of the diagram of $\lambda/\mu$ with positive integer
entries that are weakly increasing along rows and down columns.  A 
tableau of shape $\lambda/\mu$ is a \defn{standard Young tableau} if 
its entries are 
exactly the numbers $1, \dots, |\lambda/\mu|$.   We denote the
set of all standard Young tableau of shape $\lambda/\mu$ 
by $\SYT(\lambda/\mu)$.  

More generally, suppose $|\lambda/\mu| = r\ell$, and consider
a tableau of shape $\lambda/\mu$ satisfying the following conditions:
\begin{packedenum}
\item[(1)] The entries are $1, 2, \dots, \ell$.
\item[(2)] The shape determined by the entries labelled $i$ is a
(connected) $r$-\defn{ribbon}; that is, a connected skew shape with
$r$ boxes that does not contain a $2 \times 2$ square.
\end{packedenum}
Such a tableau is called a \defn{standard $r$-ribbon tableau} (also
called a \emph{standard $r$-rim hook tableau}).
When $r=1$, this definition coincides with the definition of
a standard Young tableaux.  When $r=2$, the entries labelled $i$
must cover two adjacent boxes, and hence standard $2$-ribbon tableaux 
are often called \defn{domino tableaux}.
An example for $r=5$ is given in Figure~\ref{fig:ribbontableau}.
We denote by $\ribbon(\lambda/\mu)$ the set of all standard $r$-ribbon 
tableaux of shape $\lambda/\mu$.

\begin{figure}[tb]
\centering
\begin{young}
 , & !1 & !1 & !1  & ??3 & ??3 & !!!5 \\
!1 & !1 & ?2 & ??3 & ??3 & !!!5 & !!!5 \\
?2 & ?2 & ?2 & ??3 & !!!5 & !!!5 \\
?2 & !!4 & !!4 & !!4 \\
!!4 & !!4
\end{young}
\caption{A standard $5$-ribbon tableau of skew shape $77542/1$.}
\label{fig:ribbontableau}
\end{figure}

Ribbon tableaux have been well studied.  There are several formulae 
for $|\ribbon(\lambda)|$, including hook-length formulae
and character formulae (see~\cite{FL} and the references therein).
Lascoux, Leclerc and Thibon showed that evaluations of 
Kostka-Foulkes polynomials at roots of unity can be interpreted
as enumerations of ribbon tableau.
In our case, the relevant Kostka-Foulkes polynomial is
$K_{\lambda/\mu,1^k}(q)$ (where $k=|\lambda/\mu|$), which is
the generating function for $\SYT(\lambda/\mu)$ with respect
to the \emph{charge} statistic.  
The following is a special case of their result, which may be found
in \cite[Section 4]{LLT}. 

\begin{theorem}[Lascoux-Leclerc-Thibon]
\label{thm:LLT}
Suppose $|\lambda/\mu| = r\ell$.  Then 
\[
  |\ribbon(\lambda/\mu)| 
  = \pm K_{\lambda/\mu,1^{r\ell}}\big(e^{2\pi i /r}\big)\,.
\]
\end{theorem}

In the case where $\lambda/\mu = \Rect$, the Kostka-Foulkes polynomial
is up to a power of $q$ given by the $q$-analogue of the 
hook-length formula
\[
   K_{\Rect,1^N}(q) = q^{N(d-1)/2} \frac{[N]!_q}
   {\displaystyle\prod_{\alpha \in \Rect} [h_\alpha]_q}\,,
\]
and the sign appearing in Theorem~\ref{thm:LLT} is $(-1)^{N(d-1)/2}$.

These same numbers appear in another combinatorial context.
For $\lambda \in \Lambda$, let
 $\promote: \SYT(\lambda) \to \SYT(\lambda)$ be the jeu de taquin 
\defn{promotion}
operator.  For $T \in \SYT(\lambda)$, $\promote(T)$ is defined by 
performing the 
following steps:
\begin{packedenum}
\item[(1)] Delete the entry $1$ from $T$, leaving an empty box in the
northwest corner.
\item[(2)] Slide the empty box through the remaining entries of $T$.
(The empty box repeatedly swaps places with either its neighbour 
to the
east or its neighbour to the south, whichever is smaller, until 
neither exists.)
\item[(3)] Decrement all entries by $1$ and place a new entry $|\lambda|$ 
in the empty box.  The result is $\promote(T)$.
\end{packedenum}
An example is given in Figure~\ref{fig:promotionex}.

\begin{figure}[tb]
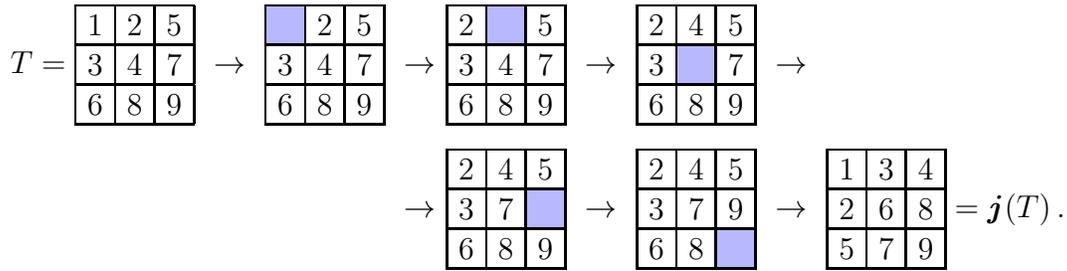

\centering
\begin{align*}
T = {\begin{young}[c]
1 & 2 & 5 \\
3 & 4 & 7 \\
6 & 8 & 9 
\end{young}} 
\ \to\ %
{\begin{young}[c]
? & 2 & 5 \\
3 & 4 & 7 \\
6 & 8 & 9 
\end{young}}
\ \to\ &%
{\begin{young}[c]
2 & ? & 5 \\
3 & 4 & 7 \\
6 & 8 & 9 
\end{young}}
\ \to\ %
{\begin{young}[c]
2 & 4 & 5 \\
3 & ? & 7 \\
6 & 8 & 9 
\end{young}} \ \to&%
\\[1ex]
\to\ &%
{\begin{young}[c]
2 & 4 & 5 \\
3 & 7 & ? \\
6 & 8 & 9 
\end{young}}
\ \to\ %
{\begin{young}[c]
2 & 4 & 5 \\
3 & 7 & 9 \\
6 & 8 & ? 
\end{young}} 
\ \to\ %
{\begin{young}[c]
1 & 3 & 4 \\
2 & 6 & 8 \\
5 & 7 & 9 
\end{young}} = \promote(T)\,.
\end{align*}
\caption{Promotion on a rectangular tableau $T \in \SYT(\protect\Rect)$,
with $n=6$, $d=3$.}
\label{fig:promotionex}
\end{figure}

It is not hard to see that $\promote: \SYT(\lambda) \to \SYT(\lambda)$ is 
invertible, and hence generates a cyclic group action on $\SYT(\lambda)$.
In the case where $\lambda = \Rect$, this operator has particularly
nice properties.
Haimain showed that $\promote^N(T) = T$ for all 
$T \in \SYT(\Rect)$ \cite{Hai}, 
hence we may think of this cyclic group as having order $N$.  
More generally, Rhoades proved the following
\emph{cyclic sieving theorem} for the promotion operator 
on $\SYT(\Rect)$~\cite{Rho} (see also~\cite{Sag-CSP,Wes}).

\begin{theorem}[Rhoades]
\label{thm:cyclicsieving}
Suppose $r$ divides $N$.  Then the number of $\promote^{N/r}$-fixed 
tableaux
in $\SYT(\Rect)$ is equal to $(-1)^{N(d-1)/r} K_{\Rect,1^N}(e^{2\pi i/r})$.
\end{theorem}

When $r=1$, the evaluation of the Kostka-Foulkes polynomial
$K_{\Rect,1^N}(1)$ is just the ordinary Kostka number 
$K_{\Rect,1^N} = |\SYT(\Rect)|$, and hence we recover Haiman's result.
More generally this theorem gives a
complete description of the orbit structure of 
promotion on $\SYT(\Rect)$.

The proof of Theorem~\ref{thm:cyclicsieving} is quite non-elementary.  
The idea is to show that the long cycle acts by promotion
on the Kazhdan-Lusztig basis of the representation
of the symmetric group $S_N$ associated to the partition $\Rect$;
hence the number of fixed points is given as a character evaluation.
In light of Theorem~\ref{thm:LLT}, Rhoades' result is equivalent to 
the following purely combinatorial theorem, for which there is
no known bijective proof.

\begin{theorem}
\label{thm:combinatorial}
Suppose $r$ divides $N$.  The number of $\promote^{N/r}$-fixed 
tableaux in
$\SYT(\Rect)$ is equal to the number of standard $r$-ribbon tableaux of
shape $\Rect$.
\end{theorem}

One consequence of our present work will be an alternate proof
of Theorem~\ref{thm:combinatorial}.  Our proof is direct in the 
sense that it does not require any knowledge of representation theory,
Kostka-Foulkes polynomials, their evaluations at roots of unity, or any
other algebraic formulae for the two quantities involved.
On the other hand, the proof does not give a direct bijection between
the two sets, nor does it suggest how to construct one (except in the
$r=2$ case, where a bijection is 
known \cite{vLee}).  
From our point of view, the two sets come from two different
(and seemingly incompatible)
ways of thinking about the problem of counting $C_r$-fixed points 
in the fibre of the Wronski map.

There is another operator on standard Young 
tableaux, closely related to promotion.  For $\lambda \in \Lambda$,
the \defn{evacuation} operator $\evac: \SYT(\lambda) \to \SYT(\lambda)$
is
\[
   \evac 
    := \promote_1 \circ \promote_2 \circ \dots \circ \promote_{|\lambda|}\,,
\]
where $\promote_i$ is the promotion operator restricted to the
entries $1, \dots, i$.  In other words, to compute $\evac(T)$, for
$T \in \SYT(\lambda)$, we first apply the promotion operator 
$\promote = \promote_{|\lambda|}$
to $T$.  We then freeze the largest entry $|\lambda|$ in $\promote(T)$, and
apply promotion to the subtableau of $\promote(T)$ consisting of entries
$1, \dots, |\lambda|-1$.  At the next stage entries $|\lambda|$ and
$|\lambda|-1$ are frozen, and we continue in this way until all
entries are frozen.  The result is $\evac(T)$.

The operator $\evac$ is an involution on $\SYT(\lambda)$.  In
the case of a rectangular shape $\lambda = \Rect$, there is a very
simple description of this involution: it coincides with the operation
of rotating the tableau by $180^\circ$ and reversing the order
of the entries (replacing each entry $k$ by $N{+}1{-}k$).  From here, 
it is easy to see that 
$\evac \circ \promote \circ \evac = \promote^{-1}$,
and hence $\evac$ and $\promote$ generate a dihedral group $D_{N}$ 
of order $2N$ acting on $\SYT(\Rect)$.
We may therefore consider the problem of counting tableaux 
in $\SYT(\Rect)$
that are fixed by a dihedral subgroup $D_r \subset N$ of order $2r$.
If $N$ is odd, every dihedral subgroup of order $2r$ is conjugate, so
we may take $D_r$ to be generated by $\evac$ and $\promote^{N/r}$.
If $N$ is even, there are up to conjugacy two dihedral subgroups 
of order $2r$: one is generated
by $\evac$ and $\promote^{N/r}$, the other by
$\evac \circ \promote$ and  $\promote^{N/r}$.

To answer this question in a manner analogous to 
Theorem~\ref{thm:combinatorial}, we consider two special types of
ribbon tableaux.
For each $\lambda \in \Lambda$, let $\lambda^\vee \in \Lambda$
denote the dual partition:
$(\lambda^\vee)^i := n-d- \lambda^{d+1-i}$ for $i=1, \dots, d$.
We will say a tableau is \defn{rotationally-invariant} if it is
fixed by the operation of rotating $180^\circ$ inside the rectangle
$\Rect$, and replacing each entry $k$ by $m{+}1{-}k$, where $m$ is 
the largest entry.
For a tableau to be rotationally-invariant, its shape must
be of the form $\lambda/\lambda^\vee$, where $\lambda \supset \lambda^\vee$.

Suppose $|\lambda/\mu| = r\ell$.
We define $\rotribbon(\lambda/\mu)$
to be the set of all rotationally-invariant standard $r$-ribbon 
tableaux (this is empty if $\mu \neq \lambda^\vee)$.  If $\ell$ is even, 
we also define 
$\longribbon(\lambda/\mu)$ be 
the set of rotationally-invariant tableaux with entries
$1, 2, \dots, \ell{-}1$, in which the entries not equal to $\frac{\ell}{2}$ 
form
$r$-ribbons, and the entries equal to $\frac{\ell}{2}$ form a $2r$-ribbon.
Some examples are given in Figure~\ref{fig:rotribbon}.  We note that
$\longribbon(\lambda/\mu)$ is always empty if $d$ and $n$ are both even.

\begin{figure}[tb]
\centering
\begin{young}
!1 & !1 & !1 & !1 & !!4 & ?5 \\
!1 & ?2 & ??3 & !!4 & !!4 & ?5 \\
?2 & ?2 & ??3 & !!4 & ?5 & ?5 \\
?2 & ??3 & ??3 & !!4 & ?5 & !6 \\
?2 & ??3 & !6 & !6 & !6 & !6
\end{young}
$\qquad\qquad$
\begin{young}
!1 & !1 & !1 & !1 & ??3 & ??3 \\
!1 & ?2 & ?2 & ??3 & ??3 & ?4 \\
?2 & ?2 & ??3 & ??3 & ?4 & ?4 \\
?2 & ??3 & ??3 & ?4 & ?4 & !5 \\
??3 & ??3 & !5 & !5 & !5 & !5
\end{young}
\caption{Examples of tableaux in $\rotribbonplain^5(\protect\Rect)$
and $\longribbonplain^5(\protect\Rect)$, with $n=11$, $d=5$.}
\label{fig:rotribbon}
\end{figure}

\begin{theorem}
\label{thm:D-combinatorial}
Suppose $r$ divides $N$.
\begin{packedenum}
\item[(i)] The number of tableaux in
$\SYT(\Rect)$ fixed by both $\evac$ and $\promote^{N/r}$  
is equal to $|\rotribbon(\Rect)|$.
\item[(ii)] The number of tableaux in
$\SYT(\Rect)$ fixed by both $\evac \circ \promote$ and $\promote^{N/r}$ 
is equal to $|\rotribbon(\Rect)|$ if $\frac{N}{r}$ is odd, and
$\big|\longribbon(\Rect)\big|$ if $\frac{N}{r}$ is even.
\end{packedenum}
\end{theorem}

In the $r=1$ case we are
computing the numbers of standard Young tableaux of
rectangular shape fixed by $\evac$, or $\evac \circ \promote$,
which were previously studied in~\cite{Rho,Ste}. 
Since these are involutions, the
quantities involved are related to the $r=2$ case of a cyclic group 
action.  The general case of Theorem~\ref{thm:D-combinatorial} 
appears to be new, and follows 
from dihedral group analogues of our main theorems.

\subsection{Statements of results}

Our main result relates the problem of counting $C_r$-fixed points in
the fibre of the Wronski map to enumeration of $r$-ribbon tableaux.

\begin{theorem}
\label{thm:ribbon}
Let $h(z)$ be a generic $C_r$-fixed polynomial of the 
form \eqref{eqn:fixedpoly}.
For any $h(z)$-compatible Richardson variety
$X_{\lambda/\mu}$,  $X^r_{\lambda/\mu}(h(z))$ is a reduced finite scheme.
The number of points in $X^r_{\lambda/\mu}(h(z))$ is equal to
the number of standard $r$-ribbon tableaux of
shape $\lambda/\mu$.
\end{theorem}

When $N = r\ell$, and 
\begin{equation}
\label{eqn:maxfixedpoly}
 h(z) = (z^r+h_1) (z^r+h_2) \dots (z^r + h_\ell)
\end{equation}
where $h_1, \dots, h_\ell \in \CC^\times$, there is only one 
compatible Richardson variety, $X = X_{\Rect}$ itself.  
In this case, we write $X^r(h(z))$ for the
$C_r$-fixed points in the fibre $X(h(z))$.  
Theorem~\ref{thm:ribbon} then states that $X^r(h(z))$ is generically
reduced, of size equal to $|\ribbon(\Rect)|$.
In this case, we will also show the following.

\begin{theorem}
\label{thm:promotion}
Let $h(z)$ be a generic polynomial of the form \eqref{eqn:maxfixedpoly}.
The number of points in $X^r(h(z))$ is equal to the 
number of $\promote^{N/r}$-fixed tableaux in $\SYT(\Rect)$.
\end{theorem}

Putting Theorems~\ref{thm:ribbon} and~\ref{thm:promotion} together, 
we deduce Theorem~\ref{thm:combinatorial}.

These two results are proved by considering the fibres over certain
special classes of polynomials.  To prove Theorem~\ref{thm:ribbon},
we change our field from the complex numbers to a field of puiseux
series $\Fpuiseux{u}$, where $\FF$ is an algebraically closed field
containing $\CC$.  Since both fields are algebraically closed
of characteristic zero, this does not change the number of points.  
However,
it allows us to make some non-algebraic assumptions about the roots:
specifically, we can consider polynomials $C_r$-fixed polynomials $h(z)$,
where $h_1, \dots, h_\ell$ have distinct valuations.  In this case,
we establish a bijection between $\ribbon(\lambda/\mu)$
and $X^r_{\lambda/\mu}(h(z))$.  These methods are related to the
asymptotic approach used by Eremenko and Gabrielov in \cite{EG-deg}, 
and build on the theory developed in \cite{Pur-Gr}.

To prove Theorem~\ref{thm:promotion}, we consider the case where
all $N$ roots of $h(z)$ lie on the unit circle.  As the unit circle is
$\PGL_2(\CC)$-equivalent to the real line, we can apply methods
based on Mukhin-Tarasov-Varchenko theorem.  The proof is essentially
a straightforward interpretation of results proved in \cite{Pur-Gr}.
From this point of view of this argument, it is clear why the tableau 
need to be of rectangular shape in Theorem~\ref{thm:cyclicsieving}: 
points corresponding to non-rectangular 
shaped tableaux lie in some proper Richardson variety of $X$; the
image under the Wronski map is a compatible polynomial, which
cannot have $N$ roots on the unit circle.

The case of an involution is rather special.  
There are a number of interpretations of the $C_2$-fixed points
in the fibre.  For example, if we assume that $h(z)$ is a real
polynomial with only pure imaginary roots, then the $C_2$-fixed 
points are precisely the real points in the fibre, and 
Theorem~\ref{thm:ribbon}
states that the number of these in any compatible Richardson
variety $X_{\lambda/\mu}$ is equal to the number of
domino tableaux of shape $\lambda/\mu$.
Alternatively may assume that the roots of $h(z)$ are 
all real, and thereby deduce Theorem~\ref{thm:ribbon} directly from 
results in \cite{Pur-Gr}.  We can also describe what happens
when the roots of $h(z)$ are not necessarily distinct.

Another consequence in real algebraic geometry concerns
maximally inflected real rational curves in projective space.  Consider
a real rational curve $\gamma : \RP^1 \to \RP^{d-1}$, 
$z \mapsto [f_1(z) : \dots : f_d(z)]$ where 
$f_1(z), \dots, f_d(z) \in \Rpoln$ are linearly independent. 
We say $\gamma$ is \defn{maximally inflected} if it has only 
simple, real ramification points.
Equivalently for the point $x \in X$ spanned by $f_1(z), \dots, f_d(z)$,
$\Wr(x;z)$ has $N$ or $N-1$ distinct roots, all of which are real.  
Consider the space $\mirrc$ of all maximally inflected real rational 
curves $\gamma : \RP^1 \to \RP^{d-1}$ of the above form.  
The results used to prove Theorem~\ref{thm:promotion} 
imply that the connected components of $\mirrc$ are in a one-or-two-to-one
correspondence with the $\promote$-orbits 
on $\SYT(\Rect)$.  

\begin{theorem}
\label{thm:mirrc}
Let $S$ be the group of order two
if $d$ is even, or the trivial group if $d$ is odd, and let $s \in S$ be 
the element of order $|S|$.
\begin{packedenum}
\item[(i)]
The connected components of $\mirrc$ are smooth, and are in bijection 
with the set $\Orbits(\promote,s)$ of orbits of $(\promote,s)$ acting 
on $\SYT(\Rect) \times S$.
\item[(ii)]
If $\theta \in \PGL_d(\RR)$ is a reflection, then the automorphism
$\gamma \mapsto \theta \circ \gamma$ of $\mirrc$ descends to the 
action of $(\identity, s)$ on $\Orbits(\promote,s)$.
\item[(iii)]
If $\psi \in \PGL_2(\RR)$ is a reflection, then the automorphism
$\gamma \mapsto \gamma \circ \psi$ of $\mirrc$ descends to the action of
$(\evac, s^{\lfloor d/2\rfloor})$ on $\Orbits(\promote,s)$.
\end{packedenum}
\end{theorem}

Both Theorems~\ref{thm:ribbon} and~\ref{thm:promotion} have extensions 
to actions of a dihedral group.  Let $D_r \subset \PGL_2(\CC)$ be the 
dihedral subgroup of order $2r$ generated by $C_r$ and the
reflection $\stdreflect$.
A polynomial $h(z)$ is $D_r$-fixed if it is of the form \eqref{eqn:fixedpoly}
where $2m+r\ell=N$, and $\{h_1, \dots, h_\ell\}$ is invariant under
$z \mapsto \frac{1}{z}$.  There are three distinct generic types:
\begin{packedenum}
\item[(1)] $\ell$ is odd, one of the $h_i$ is equal to $\pm 1$, and the
rest form generic inverse pairs;
\item[(2)] $\ell$ is even, all of the $h_i$ form generic inverse pairs
(none are equal to $\pm 1$);
\item[(3)] $\ell$ is even, one of the $h_i$ is equal to $1$, another
is equal to $-1$, and the rest form generic inverse pairs.
\end{packedenum}

\begin{theorem}
\label{thm:D-ribbon}
Let $h(z)$ be a $D_r$-fixed polynomial of one the three generic types 
above, and let 
$X_{\lambda/\mu}$, be a compatible Richardson variety.  
Then $X^r_{\lambda/\mu}(h(z))$ is a reduced finite scheme.  Moreover:
\begin{packedenum}
\item[(i)] If $h(z)$ is of type (1) or (2),
the number of $D_r$-fixed points in
$X_{\lambda/\mu}(h(z))$ is equal to $|\rotribbon(\lambda/\mu)|$.
\item[(ii)] If $h(z)$ is of type (3), 
the number of $D_r$-fixed points in
$X_{\lambda/\mu}(h(z))$ is equal to 
$\big|\longribbon(\lambda/\mu)\big|$.
\end{packedenum}
\end{theorem}

In the case where $m= 0$ (hence $h(z)$ is of the 
form \eqref{eqn:maxfixedpoly} and $\lambda/\mu = \Rect$), we
also prove:

\begin{theorem}
\label{thm:D-promotion}
Let $h(z)$ be a $D_r$-fixed polynomial of one of the three generic types,
with $m=0$.
\begin{packedenum}
\item[(i)]
If $h(z)$ is of type (1) or (2), 
the number of $D_r$-fixed points in $X(h(z))$ is equal to the 
number tableaux in $\SYT(\Rect)$ fixed by both $\evac$ and $\promote^{N/r}$.
\item[(ii)]
If $h(z)$ is of type (1) or (3), 
the number of $D_r$-fixed points in $X(h(z))$ is equal to the 
number tableaux in $\SYT(\Rect)$ fixed by both $\evac \circ \promote$ 
and $\promote^{N/r}$.
\end{packedenum}
\end{theorem}

Putting Theorems~\ref{thm:D-ribbon} and~\ref{thm:D-promotion} together,
we deduce Theorem~\ref{thm:D-combinatorial}.

The main idea behind the proof of Theorem~\ref{thm:D-ribbon} is
to show that the action of $\stdreflect$
on the fibre
models the action of $180^\circ$ rotation on tableaux.  This
follows fairly readily from the results we use to study $C_r$-fixed 
points.  In the type (3) case, a small amount of additional analysis
is needed, because we can't choose the $h_i$ to have distinct valuations.
Theorem~\ref{thm:D-promotion} uses essentially the
same ideas as Theorem~\ref{thm:promotion}.

\begin{remark}
It is worth mentioning that every virtually every result in this paper
should have an analogue for the Orthogonal Grassmannian 
$\mathrm{OG}(n,2n{+}1)$ that could be proved similarly
using the results of~\cite{Pur-Gr,Pur-OG}.
In particular, Theorems~\ref{thm:combinatorial} 
and~\ref{thm:D-combinatorial}
should have analogues where tableaux are replaced by shifted tableaux 
of staircase shape.  We have chosen not to check all of the details here.
An open question is whether it is possible to use our approach
to prove analogues of Theorems~\ref{thm:combinatorial}
and~\ref{thm:D-combinatorial}
for semistandard tableaux, such as \cite[Theorems 1.4 and 1.5]{Rho}.  
The geometry of this question is closely related the discussion 
in Section~\ref{sec:schubert}, where fibre of the Wronski map is
replaced by a Schubert intersection.  A more ambitious problem is
to generalize Thereom~\ref{thm:ribbon} to counting
$C_r$-fixed points of these Schubert intersections.
\end{remark}

\subsection{Outline}

The paper is effectively divided into two parts, which can be read
independently of each other.  The first part
comprises Sections~\ref{sec:circle} and~\ref{sec:involution}.
In these two sections, we will use arguments that apply in the case 
where $h(z)$ has $N$ roots that lie on a circle in $\CP^1$.
For such $h(z)$, one can define a correspondence between
$\SYT(\Rect)$ and points in $X(h(z))$.  We review the important
properties of this correspondence in Section~\ref{sec:circle},
and use it to prove Theorems~\ref{thm:promotion},~\ref{thm:mirrc} 
and~\ref{thm:D-promotion}.  
In Section~\ref{sec:involution}, we consider the special case of
counting points in $X(h(z))$ that are fixed by an involution 
in $\PGL_2(\CC)$.  We study this in a variety of ways, which 
include giving different interpretations of the problem,
and giving alternate proofs of some of our main results.  
More generally we explain how to count involution-fixed points in 
certain intersections of Schubert varieties.

Sections~\ref{sec:fixedpoints},~\ref{sec:ribbon} and~\ref{sec:dihedral}
form the second part of the paper.
The focus in these three sections is on understanding the restriction of 
the Wronski map to the fixed point scheme $X^r$, explicitly in terms
of the Pl\"ucker coordinates.
In Section~\ref{sec:fixedpoints}, we describe the components of $X^r$,
their coordinate rings in terms of Pl\"ucker coordinates, and their 
initial ideals with respect to certain weight functions. 
In Section~\ref{sec:ribbon}, we describe a very general way to 
associate a tableau to any point in $X$, working over a valuation ring.  
Using this association, we prove Theorem~\ref{thm:ribbon}.  
In Section~\ref{sec:dihedral}, 
we extend our results to dihedral group actions
and prove Theorem~\ref{thm:D-ribbon}.
 
\subsection*{Acknowledgements}
This work began with discussions at the October 2010 AIM workshop 
on algebraic systems with only real solutions.  I am grateful to all 
the organizers, staff and participants who made the workshop possible.
I particularly wish to thank Alexandre Eremenko, 
Frank Sottile and Alex Yong for the conversations that got everything 
started, and
Brendon Rhoades for valuable correspondence and corrections.
This research was partially supported by an NSERC discovery grant.


\section{Roots on a circle}
\label{sec:circle}


\subsection{Lifting paths and jeu de taquin}
\label{sec:lifting}

In the case where the roots of $h(z)$ all lie on a circle
in $\CP^1$, points in the fibre of the Wronski map $X(h(z))$ 
can be 
identified with 
standard Young tableaux of shape $\Rect$.  Moreover, any element 
of $\PGL_2(\CC)$ that fixes $h(z)$ can be identified with an
operator on $\SYT(\Rect)$.  We begin with a review these results,
wherein our circle is just $\RP^1 \subset \CP^1$.
The proofs may be found in~\cite{Pur-Gr}, but our presentation
here more closely follows~\cite{Pur-shifted}.

Let $\bolda = \{a_1, \dots, a_N\}$ be a multiset, where
$a_1, \dots, a_N \in \RP^1$.  Following the notation of~\cite{Pur-Gr},
we will think of a polynomial in terms of its roots, and 
write $X(\bolda) := X(h(z))$, where
\[
   h(z) = \prod_{a_i \neq \infty}(z+a_i)\,.
\]

Suppose $\bolda$ is a set (i.e. $a_i \neq a_j$ for $i \neq j$), and 
let $x \in X(\bolda)$.  We wish to associate to $x$, a tableau 
$T_x \in \SYT(\Rect)$.  
To accomplish this, we consider the total order $\preceq$ on $\RP^1$, 
defined by $a\preceq b$, if either $a=b$, $|a| < |b|$, or 
$0 < a = -b < \infty$.  Since $\bolda$ is a set, we may assume without
loss of generality, that 
\[
   a_1 \prec a_2 \prec \dots \prec a_N\,.
\]
Now, for $k = 0, \dots, N$ and $t \in [0,1]$, let
\begin{equation}
\label{eqn:tableaudefpath}
   \bolda_{k,t} = \{ta_1, ta_2, \dots, ta_k, a_{k+1}, a_{k+2}, \dots, a_N\}\,.
\end{equation}
For all $t \in (0,1]$, $\bolda_{k,t}$ is a set 
and hence
by Theorem~\ref{thm:MTV}, the fibres $X(\bolda_{k,t})$ are reduced.
It follows that there is unique (continuous) lifting of the path 
$\bolda_{k,t}$, $t \in [0,1]$ to a path $x_{k,t} \in X(\bolda_{k,t})$,
where $x_{k,1} = x$.  Now, since $0 \in \bolda_{k,0}$ 
with multiplicity $k$, and $x_{k,0} \in X(\bolda_{k,0})$ we have
$\mindeg \Wr(x_{k,0};z) = k$.
By Proposition~\ref{prop:richardson}, $x_{k,0}$ lies in a unique
Richardson variety $X_{\nu/\lambda_k}$, where $\lambda_k \vdash k$.

\begin{theorem}
The sequence $\lambda_0, \lambda_1, \dots, \lambda_N$ is a chain of 
partitions
\[
  \varnothing = \lambda_0 \subsetneq \lambda_1 \subsetneq \dots 
\subsetneq \lambda_N = \Rect\,.
\]
Moreover no two points of $X(\bolda)$ give rise to the same chain.
\end{theorem}

We define $T_x \in \SYT(\Rect)$ to be the tableau that encodes the
chain of partitions $\lambda_0 \subsetneq \lambda_1 \subsetneq 
\dots \subsetneq \lambda_N$: $T_x$ has entry $k$
in the box corresponding to $\lambda_k/\lambda_{k-1}$.

Informally, we are defining $T_x$ by moving the point $x$ so that $k$ 
smallest 
roots of $\Wr(x;z)$ are sent to zero, and seeing where $x$
ends up.  Alternatively we can obtain $T_x$ by sending the $N-k$ largest 
roots to $\infty$.  
For $k = 0, \dots, N$ and $t \in [0,1]$, let
\begin{equation}
\label{eqn:tableaualtdefpath}
   \bolda'_{k,t} = \{a_1, a_2, \dots, a_k, 
     t^{-1}a_{k+1}, t^{-1}a_{k+2}, \dots, t^{-1}a_N\}\,.
\end{equation}
and lift to $x'_{k,t} \in X(\bolda'_{k,t})$ where $x'_{k,1} = x$.
In this case, since $\infty \in \bolda_{k,0}$ with multiplicity $N-k$, we
have $\deg \Wr(x'_{k,0}; z) = k$, and hence
$x'_{k,0}$ lies in a unique Richardson variety $X_{\lambda'_k/\mu}$,
where $\lambda'_k \vdash k$.  It turns out that $\lambda'_k$ is 
equal to the partition $\lambda_k$ constructed above, and hence both
constructions yield the same tableau.  (This can be seen from
a third construction in which one considers
\[
   \bolda''_{k,t} = \{ta_1, ta_2, \dots, ta_k, 
     t^{-1}a_{k+1}, t^{-1}a_{k+2}, \dots, t^{-1}a_N\}\,.
\]
The lifted path $x''_{k,t} \in X(\bolda''_{k,t})$ (with $x''_{k,1} = x$)
is compatible
with the previous two, in that we must have $x''_{k,0} \in
X_{\lambda'_k/\lambda_k}$.  Hence this Richardson variety is 
non-empty, which implies $\lambda'_k = \lambda_k$. 
See \cite[Section 6]{Pur-Gr} for details.)

This definition of $T_x$ was based on the idea of lifting a 
path of subsets of $\RP^1$ to a path in $X$.
More generally, if
$(a_1)_t, \dots, (a_N)_t$, $t \in [0,1]$ are paths in
in $\RP^1$ and $\bolda_t = \{(a_1)_t, \dots, (a_N)_t\}$ has the
property that $X(\bolda_t)$ is reduced for all $t$,
we may consider lifting $\bolda_t$ to a path $x_t \in X(\bolda_t)$.  
We can understand such liftings in terms of the associated tableaux.

In the case where $(a_1)_t,\dots, (a_N)_t$ remain in the same 
$\preceq$-order, the tableau associated to $x_t$ remains constant.

\begin{theorem}
\label{thm:sameorder}
Suppose $(a_1)_t \prec (a_2)_t \prec \dots \prec (a_N)_t$
for all $t \in [0,1]$. If $x_t \in X(\bolda_t)$ is a lifting of
$\bolda_t$ to
$X$, then $T_{x_0} = T_{x_1}$.
\end{theorem}

The relationship between $T_{x_0}$ and $T_{x_1}$ is more interesting 
if $(a_1)_t, \dots, (a_N)_t$ 
do not remain in the same $\preceq$-order.  
First note that
by Theorem~\ref{thm:MTV}, the relation between $x_0$ and $x_1$ 
(and hence between $T_{x_0}$ and $T_{x_1}$) depends 
only on the homotopy class of the path $\bolda_t$.
It suffices, therefore, to describe this relationship in the simplest
possible case:
\begin{equation}
\label{eqn:simplepath}
\begin{gathered}
   (a_1)_t \prec  \dots \prec (a_k)_t, (a_{k+1})_t 
   \prec \dots \prec (a_N)_t \quad\text{for all $t \in [0,1]$}\,, \\
    (a_k)_0 \prec (a_{k+1})_0 \quad \text{and} \quad
    (a_{k+1})_1 \prec (a_k)_1\,.
\end{gathered}
\end{equation}
Any path is homotopic to a concatenation of paths of the 
form \eqref{eqn:simplepath}, so the answer in this simple case 
gives a complete description of the relationship between $T_{x_0}$
and $T_{x_1}$ in general.

\begin{theorem}
\label{thm:sliding}
Suppose that $\bolda_t$ is a path of the form \eqref{eqn:simplepath},
and $x_t \in X(\bolda_t)$ is a lifting to $X$.
\begin{packedenum}
\item[(i)] If $(a_k)_0$ and $(a_{k+1})_0$ have the same sign, then
$T_{x_0} = T_{x_1}$.
\item[(ii)]
If $k$ and $k{+}1$ are in the same row or column of $T_{x_0}$ then
$T_{x_0} = T_{x_1}$.  
\item[(iii)]
If $(a_k)_0$ and $(a_{k+1})_0$ have opposite signs, and $k$ and
$k{+}1$ are in different rows and columns of $T_{x_0}$, then $T_{x_1}$
is obtained from $T_{x_0}$ by switching the positions of $k$ and $k{+}1$.
\end{packedenum}
\end{theorem}

We will mainly be interested in the case where $(a_k)_0$ and 
$(a_{k+1})_0$ have opposite signs.
To see what is going on here, it is helpful to think of replacing the 
entries $1,  \dots , N$ of the tableau associated to a point
$x \in X(\bolda)$ with real entries from the
set $\bolda$.  Specifically, we will replace $k$ by $a_k$, where
$a_1 \prec a_2 \prec \dots \prec a_N$.  We refer to the resulting
object as a \defn{real valued tableau}, and write $T_x(\bolda)$
for the real valued tableau associated to $x$.

In case (ii), the tableau $T_x$ does not change; hence in the
real valued tableau $T_{x_t}(\bolda_t)$ the 
relative $\preceq$-order of 
entries does not change.  But
since the relative order of $(a_k)_t$ and $(a_{k+1})_t$ does
change, these must switch places in $T_{x_t}(\bolda_t)$.
In case (iii), 
the positions of $k$ and $k+1$ are changing along with the relative
order of $(a_k)_t$ and $(a_{k+1})_t$.  In the real valued tableau,
these two changes nullify each other, and the entries $(a_k)_t$
and $(a_{k+1})_t$ remain where they are in $T_{x_t}(\bolda_t)$.  
In both cases, this is 
a very small instance of a Sch\"utzenberger slide.  The box containing 
entry $(a_k)_t$ is sliding through the subtableau consisting of
the single box $(a_{k+1})_t$.  
As the next
example illustrates, the big picture emerges by considering what 
happens when a negative number passes by several positive numbers.

\begin{example}
\label{ex:promotion}
Consider a path
$\bolda_t = \{(a_1)_t,3,8,12,19,21,27,34,39\}$, where 
$(a_1)_0 = 0$, $(a_1)_1 = \infty$, and the path $(a_1)_t$ traverses
the negative real numbers.
Let $x_t \in X(\bolda_t)$ be a lifting such that $x_0$ is associated 
to the real valued tableau
\[
  T_{x_0}(\bolda_0) = 
  \begin{young}[c][auto=promotion]
  0 & 3 & 19  \\
  8 & 12 & 27 \\
  21 & 34 & 39
  \end{young}\,\,.
\]
As $(a_1)_t$ passes each of the other entries, the two entries switch 
places if
and only if they are in the same row or column.
\begin{align*}
  &{\begin{young}[c][auto=promotion]
  ?0 & 3 & 19  \\
  8 & 12 & 27 \\
  21 & 34 & 39
  \end{young}}
  \quad \to \quad
  {\begin{young}[c][auto=promotion]
  3 & ?-5 & 19  \\
  8 & 12 & 27 \\
  21 & 34 & 39
  \end{young}}
  \quad \to \quad
  {\begin{young}[c][auto=promotion]
  3 & ?-10 & 19 \\
  8 & 12 & 27 \\
  21 & 34 & 39
  \end{young}} \quad \to \\[2ex]
  \to \quad &
  {\begin{young}[c][auto=promotion]
  3 & 12 & 19  \\
  8 & ?-15 & 27 \\
  21 & 34 & 39
  \end{young}}
  \quad \to \quad
  {\begin{young}[c][auto=promotion]
  3 & 12 & 19  \\
  8 & ?-20 & 27 \\
  21 & 34 & 39
  \end{young}}
  \quad \to \quad
  {\begin{young}[c][auto=promotion]
  3 & 12 & 19  \\
  8 & ?-25 & 27 \\
  21 & 34 & 39
  \end{young}} \quad \to\\[2ex]
  \to \quad &
  {\begin{young}[c][auto=promotion]
  3 & 12 & 19  \\
  8 & 27 & ?-30 \\
  21 & 34 & 39
  \end{young}}
  \quad \to \quad
  {\begin{young}[c][auto=promotion]
  3 & 12 & 19  \\
  8 & 27 & ?-35 \\
  21 & 34 & 39
  \end{young}}
  \quad \to \quad
  {\begin{young}[c][auto=promotion]
  3 & 12 & 19  \\
  8 & 27 & 39 \\
  21 & 34 & ?\infty
  \end{young}}
\end{align*}
Comparing the steps here with the example in Figure~\ref{fig:promotionex},
we see that $T_{x_1} = \promote(T_{x_0})$, where $\promote$ is
the promotion operator on tableaux.
\end{example}

In general, we can model both the promotion and evacuation operators 
on $\SYT(\Rect)$, in terms of 
lifting a suitable path $\bolda_t$ that begins at a set of positive real 
numbers:
\[
   0 < (a_1)_0 < (a_2)_0 < \dots < (a_N)_0 < \infty\,.
\]

We will describe promotion as the lifting of a loop $\bolda_t$, 
$t \in [0,1]$.
Let $(a_1)_t$ be a path that first decreases from $(a_1)_0$ to $0$,
then moves along the \emph{negative} real axis to $\infty$, and
finally decreases from $\infty$ to the value $(a_1)_1 = (a_N)_0$.
Meanwhile for $k =2, \dots, N$,  let $(a_k)_t$ move monotonically
from $(a_k)_0$ to $(a_k)_1 = (a_{k-1})_0$.   Thus $\bolda_t$
is a loop that cyclically rotates the elements $(a_1)_0, \dots, (a_N)_0$.  
(See Figure~\ref{fig:modelpromote}.)
In terms of the $\preceq$-order, 
$(a_1)_t$ is changing places with each of $(a_k)_t$, $k \geq 2$,
and doing so with opposite signs.  Thus, in the real valued tableau 
associated to a lifting, the box containing $(a_1)_t$
begins in the northwest corner, and slides through all the other
entries, ending up in the southeast corner. 
Hence we see that for any lifting $x_t \in X(\bolda_t)$ of such 
a loop, we have $T_{x_1} = \promote(T_{x_0})$.

\begin{figure}[tb]
\centering
\includegraphics[height=\circlesize]{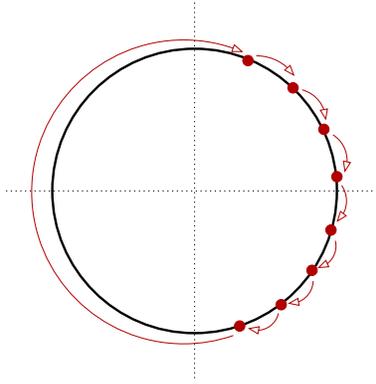}
\caption{A loop that models promotion.
The circle is $\RP^1$, drawn with $0$ south, $\infty$ north,
$1$ east, and $-1$ west.}
\label{fig:modelpromote}
\end{figure}

Evacuation, on the other hand, is described by lifting 
a path $\bolda_t$, $t \in [0,1]$, for which 
$\bolda_1 = - \bolda_0$.  For $k=1, \dots, N$, let 
$(a_k)_t$ be a path that begins at $(a_k)_0$, decreases to $0$, and 
continues moving monotonically until it reaches $(a_k)_1 = -(a_{N+1-k})_0$.  
(See Figure~\ref{fig:modelevac}.)
This describes a unique homotopy class of paths.
To show this homotopy class models evacuation, it is convenient
to consider a specific representative, which 
we will describe 
by dividing up the interval $[0,1]$ into $N$ subintervals
$[\frac{i-1}{N}, \frac{i}{N}]$, $i=1,\dots, N$.  On the first interval 
$[0, \frac{1}{N}]$, $(a_1)_t$ moves monotonically to 
$(a_1)_{1/N} = - (a_N)_0$.  Meanwhile for $k=2, \dots, N$, $(a_k)_t$
moves monotonically to $(a_k)_{1/N} = (a_{k-1})_0$.  Hence $(a_1)_t$
has changed places in the $\preceq$-order with each of
the other $(a_k)_t$, with opposite signs; this models the promotion operator 
$\promote = \promote_N$.  On the next interval 
$[\frac{1}{N},\frac{2}{N}]$, $(a_1)_t = -(a_N)_0$ remains constant,
while $(a_2)_t$ moves monotonically to $-(a_{N-1})_0$, and 
for $k=3, \dots, N$, $(a_k)_t$ moves monotonically to $(a_{k-2})_0$.
Hence $(a_2)_t$ is changing places in the $\preceq$-order, with each
$(a_k)_t$, $k \geq 3$, but does not pass $(a_1)_t$.  This models
the promotion operator restricted to the $N-1$ smallest entries of
the tableau,
which is $\promote_{N-1}$.  In general, on the $i$\nth interval 
$[\frac{i-1}{N}, \frac{i}{N}]$,
$(a_1)_t, \dots, (a_{i-1})_t$ have reached their final positions,
and $(a_i)_t$ moves to $-(a_{N+1-i})_t$, while $(a_k)_t$ moves
to $(a_{k-i})_0$; this models $\promote_i$.  Hence the complete
path models
$\evac = \promote_1 \circ \promote_2 \circ \dots \promote_N$;
that is, for any lifting $x_t \in X(\bolda_t)$, we have
$T_{x_1} = \evac(T_{x_0})$.

\begin{figure}[tb]
\centering
\includegraphics[height=\circlesize]{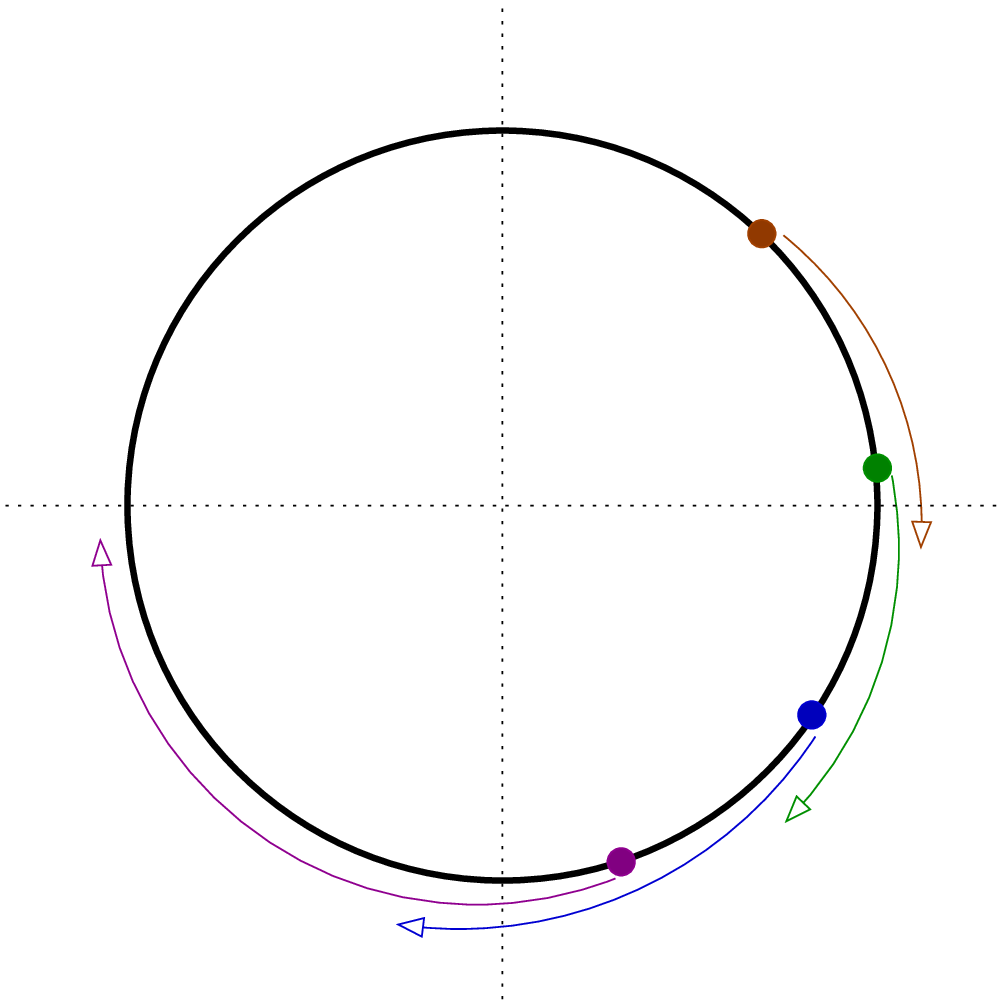}
\qquad\raisebox{0.5\circlesize}{\raisebox{-.4ex}{$\leadsto$}}\qquad
\includegraphics[height=\circlesize]{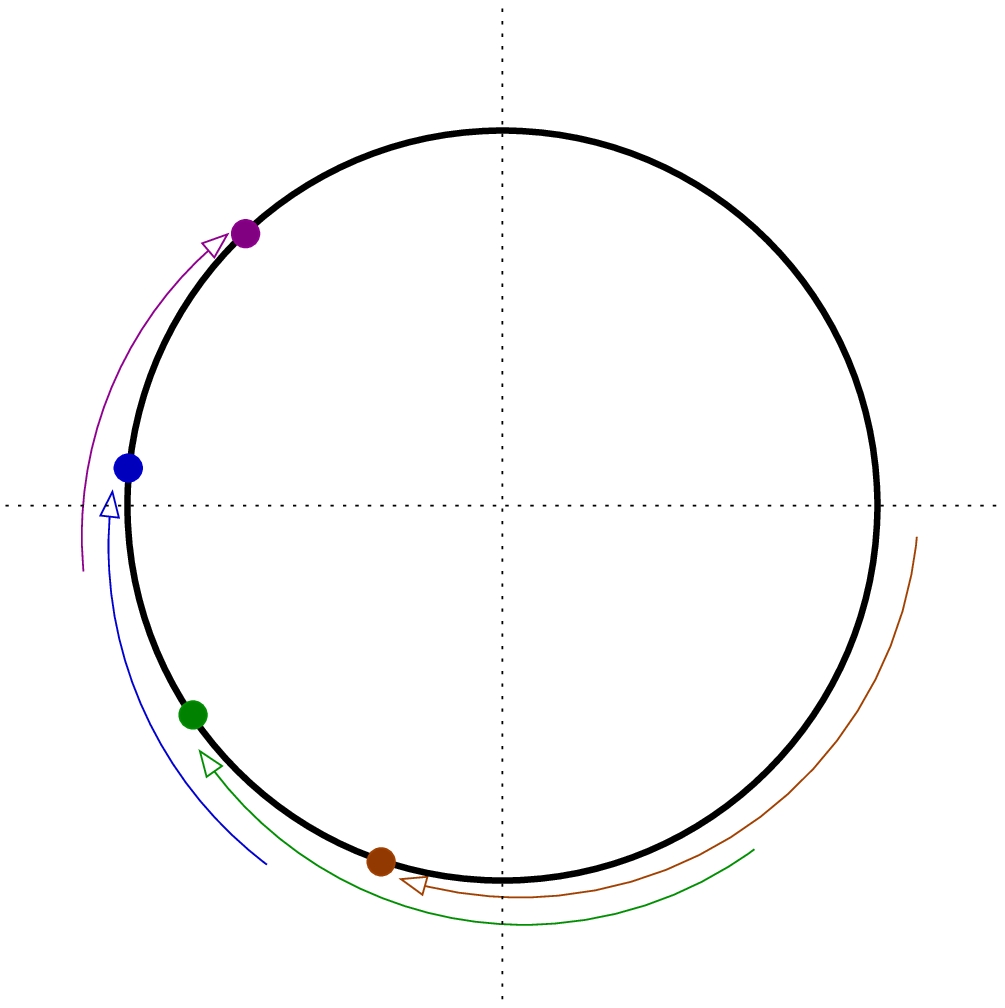}
\caption{A path that models evacuation.}
\label{fig:modelevac}
\end{figure}

From this perspective it is clear why $\evac$ is an involution:
if we reverse the roles of the positive and negative real numbers,
we see that the reverse of a path that models $\evac$ also models 
$\evac$, hence $\evac = \evac^{-1}$.

\subsection{The $\PGL_2(\RR)$-action on real valued tableaux}

The group $\PGL_2(\CC)$ acts on $\CP^1$ by M\"obius transformations.
If $a \in \CP^1$ and $\phi = 
\left(\begin{smallmatrix} 
\phi_{11} & \phi_{12} \\ \phi_{21} & \phi_{22}
\end{smallmatrix}\right)$, then
\[
  \phi(a) := \frac{\phi_{11} a + \phi_{12}}{\phi_{21} a + \phi_{22}}\,.
\]
This action restricts to an action of $\PGL_2(\RR)$ on $\RP^1$.

Let $\bolda = \{a_1, \dots, a_N\}$ be a subset of $\RP^1$, 
let $x \in X(\bolda)$ and let $\phi \in \PGL_2(\RR)$.  
Since the Wronski map is $\PGL_2(\CC)$-equivariant, we have
$\phi(x) \in X(\phi(\bolda))$, where 
$\phi(\bolda) = \{\phi(a_1), \dots, \phi(a_N)\}$ is also
a subset of $\RP^1$.  This defines an action on real valued
tableaux.  If $T = T_x(\bolda)$ is the real valued tableau 
associated to $x$, we define 
\[
  \phi(T) := T_{\phi(x)}(\phi(\bolda))\,.
\]
If $\phi$ is in the connected component of $\smallidmatrix$,
we can compute this action explicitly, using Theorem~\ref{thm:sliding}.
Let $\phi_t$, $t \in [0,1]$, be any path in $\PGL_2(\RR)$ 
from $\phi_0 = \smallidmatrix$ to $\phi_1 = \phi$.
Then $x_t = \phi_t(x)$ is a lifting of the path $\phi_t(\bolda)$.
Theorem~\ref{thm:sliding} therefore tells us exactly how the tableau
$T_{\phi_1(x)}(\phi_1(\bolda)) = \phi(T)$ is related to
$T_{\phi_0(x)}(\phi_0(\bolda)) = T$.

If $\phi$ is not in the connected component of the identity,
then we can write $\phi = \altreflect\phi' $, or 
$\phi = \stdreflect\phi''$, where $\phi', \phi''$ are in the
connected component of the identity matrix.  The next lemma
tells us how to compute the actions of $\altreflect$ and
$\stdreflect$.

\begin{lemma}
\label{lem:tworeflections}
Suppose that 
\begin{equation}
\label{eqn:absincreasing}
|a_1| < |a_2| < \dots < |a_N|\,,
\end{equation}
and let $T$ be a real valued tableau with entries from the set $\bolda$.  
\begin{packedenum}
\item[(i)] 
$\altreflect T = -T$, where $-T$ is obtained from $T$ by replacing
$a_k$ with $-a_k$, for $k=1, \dots, N$.
\item[(ii)]
$\stdreflect T = T^\vee$, where $T^\vee$ is obtained from $T$ by
rotating $180^\circ$ and 
replacing $a_k$ with $\frac{1}{a_k}$, for $k=1, \dots, N$.
\end{packedenum}
\end{lemma}

If we have a set $\bolda$ for which $a_i = -a_j$ for some $i \neq j$,
then Lemma~\ref{lem:tworeflections} cannot be used directly.  The
problem is that if we replace \eqref{eqn:absincreasing} by the
weaker condition $a_1 \prec a_2 \prec \dots \prec a_N$, then
$-T$ and $T^\vee$ may not be real valued tableaux:
the operation $a \mapsto -a$ is not $\preceq$-order preserving,
and $a \mapsto \frac{1}{a}$ is not $\preceq$-order reversing.
In this case, we can still compute $\altreflect T$ and $\stdreflect T$,
but we need to do so by perturbing the point $\bolda$, which introduces
additional sliding.  We consider examples of this 
in Section~\ref{sec:involution}.

\begin{proof}
Let $x \in X(\bolda)$, and
$\hat x = \altreflect x$.
Let $\hat \bolda = \altreflect\bolda = \{-a_1, \dots, -a_N\}$.
Recall from Section~\ref{sec:lifting} that the tableaux $T_x$ is computed 
by lifting a path $\bolda_{k,t}$ (see \eqref{eqn:tableaudefpath})
to a path $x_{k,t}$.
Similarly the tableau $T_{\hat x} = \altreflect T_x$ is computed by lifting 
a path $\hat \bolda_{k,t}$ to a path $\hat x_{k,t}$.  Now, since 
$a_1 \prec a_2 \prec \dots \prec a_N$ and 
$-a_1 \prec -a_2 \prec \dots \prec -a_N$, 
$\hat \bolda_{k,t} = \altreflect\bolda_{k,t}$
for all $k = 0, \dots, N$, $t \in [0,1]$, and therefore
$\hat x_{k,t} = x_{k,t}$.  Finally, since $\altreflect$ fixes every 
Richardson
variety,  $x_{k,0} \in X_{\nu, \lambda_k}$ if and only if
$\hat x_{k,0} \in X_{\nu, \lambda_k}$, and we see that 
$T_x = T_{\hat x}$, which is equivalent to (i).

For (ii), the argument is similar.  This time, let
$\hat x = \stdreflect x$, and
$\hat \bolda = \stdreflect\bolda = \{\frac{1}{a_1}, \dots, \frac{1}{a_N}\}$.
Since,
$\frac{1}{a_N} \prec \frac{1}{a_{N-1}} \prec \dots \prec \frac{1}{a_1}$,
the path $\hat \bolda_{N-k,t}$ is equal to $\stdreflect \bolda'_{k,t}$, 
where $\bolda'_{k,t}$ is the path \eqref{eqn:tableaudefpath} 
that appears in the \emph{alternate} 
definition of $T_x$.  Thus we 
obtain $\hat x_{N-k,t} = \stdreflect x'_{k,t}$, 
and deduce the result from the fact 
that $\stdreflect X_{\nu/\lambda_k} = X_{\lambda_k^\vee/\nu^\vee}$.
\end{proof}

Now suppose that $\phi(\bolda) = \bolda$ for some set $\bolda$.  
In this case, the real valued tableau $\phi(T_x(\bolda))$ has the same 
entries as $T_x(\bolda)$.  By identifying real valued tableaux
with entries from the set $\bolda$ with $\SYT(\Rect)$, 
we can regard $\phi$ as an operator on $\SYT(\Rect)$.  
We note that this identification is always defined with respect to 
a particular choice of $\phi$-fixed set $\bolda$, and that changing 
this set will change the identification.

We now show that if $a_1, \dots, a_N$ are positive real numbers,
then the operator defined by $\phi$ belongs to the dihedral 
subgroup $D_{2N}$ generated by the promotion and evacuation operators 
$\promote, \evac : \SYT(\Rect) \to \SYT(\Rect)$.
Moreover, for a suitable choice of $\bolda$, we will realize
$\promote$ and $\evac$ by elements of $\PGL_2(\RR)$.  From here 
it will be a simple matter to prove
Theorems~\ref{thm:promotion} and~\ref{thm:D-promotion}.

\begin{lemma}
\label{lem:fixeddihedral}
Let $\phi \in \PGL_2(\RR)$, and suppose that $\phi(\bolda) = \bolda$,
where $0 < a_1 < \dots < a_N < \infty$.  If $\phi$ is in the connected
component of $\smallidmatrix$ then the action of $\phi$ on $\SYT(\Rect)$
coincides with $\promote^{\,k}$ for some integer $k$.  If
$\phi$ is not in the connected component of $\smallidmatrix$, then
the action of $\phi$ on $\SYT(\Rect)$ coincides with 
$\evac \circ \promote^{\,k}$, for some integer $k$.
\end{lemma}

\begin{proof}
If $\phi$ is in the connected component of the identity, then
$\phi(T)$ is defined by lifting a path $\phi_t(\bolda)$, $t \in [0,1]$.
Such a path must cyclically rotate the elements of $\bolda$ by some
amount; i.e. there is some $k$ such that $\phi(a_i) = a_{i+k \!\!\pmod N}$.
Then $\phi_t(\bolda)$ is homotopic to the
path that models $\promote^{k}$, 
and hence $\phi(T) = \promote^{k}(T)$.

Now suppose $\phi$ is not in the connected component of the identity.
Write $\phi = \altreflect \phi'$, 
and consider the action of $\phi'$ on real valued tableau.
The path $\phi'_t(\bolda)$ used to define this action preserves 
cyclic order, and
$\phi'(\bolda) = \altreflect \bolda = \{-a_1, \dots, -a_N\}$.
Hence $\phi'_t(\bolda)$ must homotopic to a path that cyclically rotates
the elements of $\bolda$ by some amount, followed by a path that
models evacuation.  In other words, if $T_x$ is the tableau
corresponding to a point $x \in X(\bolda)$, then 
the tableau corresponding to $\phi'(x)$ is
$\evac \circ \promote^{\,k}(T_x)$, for some $k$.
By Lemma~\ref{lem:tworeflections}(i), this is also the tableau
corresponding to $\phi(x)$.
\end{proof}

In order to explicitly realize $\promote$ and $\evac$ by
elements of $\PGL_2(\RR)$ and relate these to the statements of
Theorems~\ref{thm:promotion} and~\ref{thm:D-promotion}, it will be 
convenient work with the unit circle $S^1 \subset \CP^1$, rather than $\RP^1$.
To relate the two, we will use a M\"obius transformation that maps 
``most'' of the unit circle to the positive real numbers.
Let $\eta := e^{i \varepsilon}$ where $0 < \varepsilon < \frac{\pi}{N}$,
and
\[
   \psi := 
  \begin{pmatrix}
   1+\overline{\eta} & -(1+\eta) \\
   1+\eta & -(1+\overline{\eta})  \\
   \end{pmatrix}\,.
\]
Then $\psi$ defines a M\"obius transformation which sends $S^1$
to $\RP^1$, with $\psi(\eta) = 0$, $\psi(\overline \eta) = \infty$
and $\psi(-1) = 1$.  We let $\PGL_2(S^1) := \psi \PGL_2(\RR) \psi^{-1}$
denote the subgroup of $\PGL_2(\CC)$ that fixes the unit circle.

For $k=1, \dots, N$, let $a_k = e^{(2k-1)\pi i/N}$.  
Since $a_1, \dots, a_N$ lie on the long
arc of the unit circle between $\eta$ and $\overline\eta$,
we have
\[
  0 < \psi(a_1) < \psi(a_2) < \dots < \psi(a_N) < \infty\,.
\]
For each point $x \in X(\bolda)$, we define the tableau 
$T_x \in \SYT(\Rect)$ associated to $x$ to be $T_{\psi(x)}$.  
Using this identification, any 
$\phi \in \PGL_2(S^1)$ that fixes $\bolda$ defines an operator 
on $\SYT(\Rect)$.

\begin{lemma}
\label{lem:evacpromote}
With $a_1, \dots, a_N$ as above,
the $\PGL_2(S^1)$ group elements
\[
\begin{pmatrix}e^{-\pi i/N} & 0 \\ 0 & e^{\pi i/N} \end{pmatrix}
\qquad\text{and}\qquad
\begin{pmatrix}0 & 1 \\ 1 & 0 \end{pmatrix}
\]
fix $\bolda$, and act on $\SYT(\Rect)$ as $\promote$ 
and $\evac$ respectively.
\end{lemma}

Together with Lemma~\ref{lem:tworeflections}(ii), 
Lemma~\ref{lem:evacpromote} shows that 
for $T \in \SYT(\Rect)$, $\evac(T)$ is 
the tableau obtained by rotating $T$ by $180^\circ$ and replacing each 
entry $k$ by $N{+}1{-}k$.  This is well known, but not entirely obvious.

\begin{proof}
We argue as in the proof of Lemma~\ref{lem:fixeddihedral}.
That both of these M\"obius transformations fix $\bolda$ is easy to
check.  In particular,
$\big(\begin{smallmatrix}e^{-\pi i/N} & 0 \\ 0 & e^{\pi i/N} 
\end{smallmatrix}\big) a_i = a_{i+1 \!\!\pmod N}$,  which implies
that $\phi$ acts as $\promote^1$.
For the action of $\stdreflect$, 
note that $\stdreflect$
commutes with $\psi$, so it is enough to show that 
$\stdreflect$ acts as $\evac$ on real valued tableaux with 
entries $\psi(\bolda)$.
But now, $\stdreflect = \altreflect\halfrotate$, and
the action of $\halfrotate$ on $\psi(\bolda)$ is computed by lifting
the path 
$\left(\begin{smallmatrix}
   \cos \pi t & - \sin \pi t \\
   \sin \pi t &  \cos \pi t 
\end{smallmatrix} \right)$, $t \in [0,1]$,
which models evacuation.
\end{proof}

We are now in a position to prove 
Theorems~\ref{thm:promotion} and~\ref{thm:D-promotion}.

\begin{proof}[Proof of Theorem~\ref{thm:promotion}]
Let $h(z) = (z+a_1)(z+a_2) \dotsb (z+a_N) = z^N + (-1)^N$.  
Then $h(z)$ is a $C_r$-fixed polynomial, and
the fibre $X(h(z)) = X(\bolda)$ is reduced.
Therefore the number of $C_r$-fixed points in $X(\bolda)$ 
is the generic answer.  A point $x \in X(\bolda)$ is $C_r$-fixed,
if and only if the corresponding tableau $T_x$ is $C_r$-fixed;
but by Lemma~\ref{lem:evacpromote}, the generator of $C_r$
acts as $\promote^{N/r}$ on $\SYT(\Rect)$.
Hence the $C_r$-fixed points in $X(\bolda)$ are
in bijection with $\promote^{N/r}$-fixed tableaux in $\SYT(\Rect)$.
\end{proof}

\begin{proof}[Proof of Theorem~\ref{thm:D-promotion}]
Case (i) is similar to the proof of Theorem~\ref{thm:promotion}.
The polynomial $h(z) = (z+a_1)(z+a_2) \dotsb (z+a_N) = z^N + (-1)^N$ 
is fixed
by $C_r$ and also by $\stdreflect$.  Moreover, $h(-1) \neq 0$,
so $h(z)$ is of type (1) or (2), depending on whether $\frac{N}{r}$ 
is odd or even.  By Lemma~\ref{lem:evacpromote}, the generators 
of $D_r$ act as $\promote^{N/r}$ and $\evac$, and hence the
$D_r$-fixed points of $X(\bolda)$ are in bijection with
tableaux in $\SYT(\Rect)$ fixed by $\promote^{N/r}$ and $\evac$.

For case (ii), we need to consider a different polynomial.
Let $h(z) = z^N - (-1)^N$.  Since $h(-1) = 0$, this is of type (1) or (3), 
depending on
whether $\frac{N}{r}$ is odd or even.  Let 
\[
   \phi = \begin{pmatrix}
          e^{-\pi i/2N} & 0 \\
          0 & e^{\pi i/2N} 
   \end{pmatrix}\,.
\]
Then $h(z) = (z+\phi(a_1))(z+\phi(a_2))\dotsb(z+\phi(a_N))$,
so $X(h(z)) = \phi (X(\bolda))$.  A point
$x \in X(h(z))$ is fixed by $D_r$ if and only if 
$\phi^{-1}(x) \in X(\bolda)$ is fixed by $\phi^{-1}D_r\phi$, i.e. by
\[
   \phi^{-1}
   \begin{pmatrix}e^{-\pi i/r} & 0 \\ 0 & e^{\pi i/r} \end{pmatrix}
   \phi =
   \begin{pmatrix}e^{-\pi i/r} & 0 \\ 0 & e^{\pi i/r} \end{pmatrix}
\]
and
\[
   \phi^{-1}
   \begin{pmatrix}0 & 1 \\ 1 & 0 \end{pmatrix}
   \phi =
   \begin{pmatrix}0 & 1 \\ 1 & 0 \end{pmatrix}
   \begin{pmatrix}e^{-\pi i/N} & 0 \\ 0 & e^{\pi i/N} \end{pmatrix}\,.
\]
By Lemma~\ref{lem:evacpromote}, the tableaux corresponding to 
points $\phi^{-1}(x)$ fixed by these two group elements
are those fixed by $\promote^{N/r}$ and $\evac \circ \promote$.
\end{proof}

\subsection{Maximally inflected real rational curves}
\label{sec:mirrc}

As an application of the results from this section, we will
briefly discuss the classification of maximally inflected real rational 
curves.
Let $\gamma : \RP^1 \to \RP^{d-1}$ be a real parameterized curve 
\[
    \gamma : z \mapsto [f_1(z) : \dotsb : f_d(z)]\,,
\]
with $f_1(z), \dots, f_d(z) \in \Rpoln$ linearly independent.  
A \defn{ramification point}
of $\gamma$ is a point $z \in \RP^1$ at which
$\gamma(z), \gamma'(z), \gamma''(z), \dots, \gamma^{(d-1)}(z)$ 
do not span $\RR^d$.
When $d=2$, a ramification point is
a critical point of the rational function $\frac{f_1(z)}{f_2(z)}$. 
When $d=3$, $\gamma$ is a planar curve and the simplest example
of a real ramification point of $\gamma$ is an inflection point.
In general, if $\langle \gamma \rangle$ denotes the vector space spanned
by $f_1(z), \dots, f_d(z)$,
it is not hard to see that a 
ramification point of $\gamma$ is a root of $\Wr(\langle\gamma\rangle;z)$; 
hence $\gamma$ can have at most $N$ real ramification points.  The space
$\mirrc$ of maximally inflected real rational curves is the algebraic 
set of all curves $\gamma$ of this form with exactly $N$ distinct, real
ramification points.  These curves and their relationship to 
Theorem~\ref{thm:MTV} were first studied in~\cite{KS}.

One interpretation of Theorem~\ref{thm:mirrc} is that the different ``types''
of curves in $\mirrc$ correspond to the $D_N$-orbits on $\SYT(\Rect)$,
where $D_N$ is the dihedral group of order $2N$ generated by $\promote$
and $\evac$.   Here, two curves are of the same type if 
one is obtained from the other by deforming continuously 
within $\mirrc$, and possibly composing with a reflection of 
$\RP^1$ or $\RP^{d-1}$ (or both).
The actual statement of Theorem~\ref{thm:mirrc} is about a finer 
invariant than this notion of type, in which one does not identify 
curves related by reflections.
On the other hand, curves of different types are fundamentally different 
because they
degenerate in different ways when the ramification points collide
(see~\cite{KS,Sot-F}); hence the type of a curve is the coarsest
invariant one wants for any classification.

The analysis of the components of $\mirrc$ is slightly different
depending on whether $d$ is odd or even.  This is because
$\RP^{d-1}$ is orientable if $d$ is even, and non-orientable if
$d$ is odd.  When $d$ is even, we put
$S := \{\pm 1\}$, and $s:=-1$.  Let $\Wr(\gamma;z)$ denote the
right hand side of~\eqref{eqn:wronskian}.
Although the coordiates of $\gamma(z)$ are
only defined up to a real scalar multiple, since $d$ is
even, the sign 
$\Wr(\gamma;z)$ is unchanged by any real rescaling 
of $[f_1(z) : \dots :f_d(z)]$.
We can therefore define the sign of $\gamma$ to be
\[
   \epsilon_\gamma 
:= \lim_{z\to 0^{-}} \sign(\Wr(\gamma, z))\,.
\]
If $\Wr(\gamma;0) \neq 0$, then $\epsilon_\gamma$ is just the
sign of $\Wr(\gamma; 0)$.  When $d$ is odd 
there is no well-defined notion of the sign of a curve:  
we put $S := \{1\}$, $s := 1$, and $\epsilon_\gamma := 1$ for
every $\gamma \in \mirrc$.

The group $\PGL_2(\RR)$ acts on $\mirrc$:  if
$\phi \in \PGL_2(\RR)$, then
$\phi(\gamma)(z) := [\phi(f_1)(z) : \dotsb : \phi(f_d)(z)]$ is
well-defined.  Let $\PGL_2(\RR)^\circ \subset \PGL_2(\RR)$ 
denote the connected component of the identity element.  It is
clear that if $\phi \in \PGL_2(\RR)^\circ$ then $\gamma$
and $\phi(\gamma)$ are in the same connected component of $\mirrc$.
We also have an action of $\PGL_d(\RR)$ on $\mirrc$, given 
by $\gamma \mapsto \theta \circ \gamma$ for $\theta \in \PGL_d(\RR)$.
If $\theta$ is in the connected component of the idenity element
then $\gamma$ and $\theta \circ \gamma$ are in the same component
of $\mirrc$.  Otherwise, we have
$\langle \gamma \rangle = \langle \theta \circ \gamma \rangle$
and $\epsilon_{\gamma} = s\,\epsilon_{\theta \circ \gamma}$.

We now associate to each $\gamma \in \mirrc$ a discrete invariant
$\orbit_\gamma$, which will be a subset of $\SYT(\Rect) \times S$.
Let $x = \langle \gamma \rangle$, and 
\[
   M_x := \{\phi \in \PGL_2(\RR)^\circ \mid \phi(x) \in X(\bolda)
   \text{ for }\bolda \subset \RR_{> 0}\}\,.
\]
We define $\orbit_\gamma$ to be the set of pairs 
\[
   \orbit_\gamma := 
   \big\{(T_{\phi(x)},\epsilon_{\phi(\gamma)}) 
   \bigmid \phi \in M_x\big\}\,.
\]
Our next goal is to show 
that $\orbit_\gamma$ has the properties needed to prove 
Theorem~\ref{thm:mirrc}.

\begin{lemma}
\label{lem:orbitinvariant1}
For every curve $\gamma \in \mirrc$, 
$\orbit_\gamma \in \Orbits(\promote,s)$;
i.e.  $\orbit_\gamma$ is an orbit of $(\promote,s)$ acting on
$\SYT(\Rect) \times S$.
\end{lemma}

\begin{proof}
Let $x = \langle \gamma \rangle$, and $\phi \in M_x$.
Then $\phi(x) \in X(\bolda)$, 
where $\bolda$ is a set of the form 
$\bolda =\{0 < a_1 < \dots, a_N < 0\}$.  We must show that
$\orbit_\gamma$ is equal to the 
$(\promote,s)$-orbit of 
$(T_{\phi(x)},\epsilon_{\phi(\gamma)})$.

For any other $\phi' \in M_x$, we can write
$\phi' = \psi \phi$ where $\psi \in \PGL_2(\RR)^\circ$.  
Since $\phi'(x) \in X(\psi(\bolda))$ we must have 
\[
    0 < \psi(a_k) < \psi(a_{k+1}) < \dots < \psi(a_N) < \psi(a_1)
       < \psi(a_{k-1}) < \infty
\]
for some integer $k$.
As in the proof of Lemma~\ref{lem:fixeddihedral}, 
this implies that 
$T_{\phi'(x)} = T_{\psi(\phi(x))} = \promote^{k}(T_{\phi(x)})$.
If $d$ is even, the sign of $\Wr(\gamma;z)$ changes each time we
pass a root; hence we also have 
$\epsilon_{\phi'(\gamma)}= s^k \epsilon_{\phi(\gamma)}$.
Thus $\orbit_\gamma$ is contained in the $(\promote,s)$-orbit 
of $(T_{\phi(x)},\epsilon_{\phi(\gamma)})$.

For the reverse inclusion, choose any $\psi \in \PGL_2(\RR)^\circ$ 
such that 
\[
    0 < \psi(a_2) < \psi(a_3) < \dots < \psi(a_N) < \psi(a_1) < \infty\,.
\]
If $\phi' = \psi \phi$, then $\phi' \in M_x$ and 
$(T_{\phi'(x)}, \epsilon_{\phi'(\gamma)})
= (\promote(T_{\phi(x)}), s\,  \epsilon_{\phi(\gamma)})$.  
Thus $\orbit_\gamma$ contains
the $(\promote,s)$-orbit of 
$(T_{\phi(x)},\epsilon_{\phi(\gamma)})$.
\end{proof}

\begin{lemma}
\label{lem:orbitinvariant2}
For every orbit $\orbit \in \Orbits(\promote,s)$, there
exists a curve $\gamma$ such that $\orbit_\gamma = \orbit$.
\end{lemma}

\begin{proof}
Choose any $(T,\epsilon) \in \orbit$, 
any $\bolda = \{0 < a_1 < \dots < a_N < \infty\}$,
any $x \in X(\bolda)$ such that $T_x = T$, and any $\gamma$ such that
$\langle \gamma \rangle = x$.  If $\epsilon_\gamma \neq \epsilon$, then
replace $\gamma$ by $\theta \circ \gamma$ for any reflection 
$\theta \in \PGL_d(\RR)$.
Then $(T,\epsilon) \in \orbit_\gamma$, 
and
thus by Lemma~\ref{lem:orbitinvariant1}, $\orbit_\gamma = \orbit$.
\end{proof}

\begin{lemma}
\label{lem:orbitinvariant3}
Two curves $\gamma$ and $\gamma'$ are in the same connected component 
of $\mirrc$ if and only if $\orbit_\gamma = \orbit_{\gamma'}$.
\end{lemma}

\begin{proof}
Suppose $\gamma$, $\gamma'$ are in the same connected component
of $\mirrc$.  Let $\gamma_t$, $t \in [0,1]$, be a path in $\mirrc$
from $\gamma_0 = \gamma$ to $\gamma_1 = \gamma'$.  Let
$x_t = \langle\gamma_t \rangle$.  There exists a continuous path
$\phi_t$ in $\PGL_2(\RR)^\circ$ such that $\phi_t \in M_{x_t}$.  By
Theorem~\ref{thm:sameorder}, $T_{\phi_t(x_t)}$ is constant for
all $t \in [0,1]$.  Since $\Wr(\gamma_t; 0) \neq 0$ for all $t \in [0,1]$,
$\epsilon_{\phi_t(\gamma_t)}$ is also constant. 
Thus we have 
$(T_{\phi_0(x_0)}, \epsilon_{\phi_0(\gamma_0)}) 
= (T_{\phi_1(x_1)}, \epsilon_{\phi_1(\gamma_1)}) 
\in \orbit_\gamma \cap \orbit_{\gamma'}$,
which implies $\orbit_\gamma = \orbit_{\gamma'}$.

Conversely, suppose that $\orbit_\gamma = \orbit_{\gamma'}$.
Let $x = \langle \gamma \rangle$, $x' = \langle \gamma' \rangle$.
Choose any $(T,\epsilon) \in \orbit_\gamma$, and
any $\phi \in M_x$, $\phi' \in M_{x'}$, so that
$(T,\epsilon) = (T_{\phi(x)}, \epsilon_{\phi(\gamma)})
= (T_{\phi'(x')}, \epsilon_{\phi'(\gamma')})$
Choose a
path of sets $\bolda_t \subset \RR_{>0}$ such that $\phi(x) \in X(\bolda_0)$
$\phi'(x') \in X(\bolda_1)$, and let $x_t \in X(\bolda_t)$ be a lifting
of $\bolda_t$ such that $x_0 = \phi(x)$.  By
Theorem~\ref{thm:sameorder}, $T_{x_0} = T_{x_1} = T$, which shows
that $x_1 = \phi'(x')$.  Lifting $x_t$ to a path in $\mirrc$,
it follows that $\phi(\gamma)$ is in the same component as
$\theta \circ \phi'(\gamma')$ for some $\theta \in \PGL_d(\RR)$.
But since $\epsilon_{\phi(\gamma)} = \epsilon_{\phi'(\gamma')}$,
$\theta$ must be in the connected component of the identity element.
It follows that $\gamma$ and $\gamma'$ are in the same connected
component of $\mirrc$.
\end{proof}

\begin{proof}[Proof of Theorem~\ref{thm:mirrc}]
For (i), Lemmas~\ref{lem:orbitinvariant1},~\ref{lem:orbitinvariant2} 
and~\ref{lem:orbitinvariant3} imply that $\gamma \mapsto \orbit_\gamma$
descends to a bijection from the connected components of $\mirrc$ to
the set $\Orbits(\promote, s)$.
To see $\mirrc$ is smooth it is enough to show that
$\{\langle \gamma \rangle \mid \gamma \in \mirrc\}$ is smooth.  But
this follows from the fact that
$X$ is smooth at $x = \langle \gamma \rangle$, and the fibre
$\Wr^{-1}(\Wr(x))$ is reduced.
For any reflection $\theta \in \PGL_2(\RR)$ 
we have $\epsilon_\gamma = s\,\epsilon_{\theta\circ \gamma}$
and $\langle \gamma \rangle = \langle \theta \circ \gamma \rangle$; 
statement (ii) follows.
It is enough to prove (iii) for the reflection 
$\stdreflect \in \PGL_2(\RR)$.  
Let $\hat \gamma = \stdreflect \gamma$, 
$\hat x = \langle \hat \gamma \rangle$.  
If $\phi \in M_x$, then 
$\hat \phi = \stdreflect\phi\stdreflect \in M_{\hat x}$.
By Lemma~\ref{lem:tworeflections}(i), 
$T_{\hat \phi(\hat x)} = \evac(T_{\phi(x)})$.  To compute the
sign $\epsilon_{\hat \phi(\hat \gamma)}$, note that there is
a path $\gamma_t$ from 
$\hat \phi(\hat \gamma)$ to $\altreflect \phi(\gamma)$ such that
$\Wr(\gamma_t; 0) \neq 0$ for all $t$.  Thus the sign of
$\hat \phi(\hat \gamma)$ is the same as that of 
$\altreflect \phi(\gamma)$, which is 
$s^{\lfloor d/2\rfloor} \epsilon_{\phi(\gamma)}$.
Thus 
$\orbit_{\hat \gamma} = (\evac, s^{\lfloor d/2\rfloor}) \cdot \orbit_\gamma$,
as required.
\end{proof}


\section{The case of an involution}
\label{sec:involution}


\subsection{The generic cases}

Suppose $\phi \in \PGL_2(\CC)$ is an involution (other than the
identity element) and let $h(z) \in \PpolN$ be a $\phi$-fixed polynomial.
We now
consider the problem of counting the $\phi$-fixed points 
in $X(h(z))$.  This is, of course, a special case of our more general 
problem; however, it is an interesting one because there are 
several interpretations of its various ingredients.  
For example, special choices of $h(z)$ and $\phi$ lead to 
different interpretations of the problem.  
Also,  $\phi$ generates a group of order $2$, which we can think of
as either the cyclic group $C_2$,
or the dihedral group $D_1$.  Hence all four of our theorems 
about fixed points
(Theorems~\ref{thm:ribbon}, \ref{thm:promotion}, \ref{thm:D-ribbon}
and~\ref{thm:D-promotion}) will have something to say about this case.
Unlike the general case, we can explicitly
see how the different combinatorial objects that
arise are related to each other.  Finally, our 
results for involutions extend to the case where the roots of $h(z)$
are non-distinct.

Before we begin our investigations, we should note that
there will be different cases to consider.
Suppose $h(z) = \prod_{a_i \neq \infty} (z+a_i)$.  For now
we will assume that $\bolda = \{a_1, \dots, a_N\}$ is subset of $\CP^1$ 
such that $\phi(\bolda) = \bolda$, and also that $X(\bolda)$ is reduced.
There are exactly two points in $\CP^1$ that are fixed by $\phi$.
The number of non-fixed points of $\phi$ in $\bolda$ must be even.
Thus we have three cases:
\begin{packedenum}
\item[(1)] $N$ is odd,  and the set $\bolda$ contains exactly one
$\phi$-fixed point of $\CP^1$.
\item[(2)] $N$ is even, and the set $\bolda$ contains no
$\phi$-fixed points of $\CP^1$.
\item[(3)] $N$ is even, and the set $\bolda$ contains both $\phi$-fixed 
points $\CP^1$.
\end{packedenum}
When $\phi = \stdreflect$, these three cases correspond to the three 
generic types for $h(z)$ discussed in the introduction.  They are 
fundamentally different, and we should expect to see different behaviours 
in each one.  On the other hand, the answer only depends on which of 
the three cases we are in.

\begin{lemma}
\label{lem:involutioncases}
For any  $\bolda$ and $\phi$ as above, the number of $\phi$-fixed 
points in $X(\bolda)$ depends only on the 
number of $\phi$-fixed elements of the set $\bolda$.
\end{lemma}

\begin{proof}
Let $\PpolN^\phi$ denote the $\phi$-fixed polynomials in $\PpolN$.
Let $w_1, w_2 \in \CP^1$ denote the two $\phi$-fixed points.
The $\phi$-fixed point scheme of $X$ is flat and finite over $\PpolN^\phi$.  
Hence the number of points
in the reduced fibre $X(\bolda)$ depends only on the component of
of $\PpolN^\phi$ that $h(z)$ lies in.  

It is not hard to see that
$\PpolN^\phi$ has two exactly components.  If $N$ is odd, one 
component is the closure of the set of $h(z)$ identified with 
$\{\bolda \mid w_1 \in \bolda\}$;
the other is the closure of the set of all $h(z)$
identified with $\{\bolda \mid w_2 \in \bolda\}$.  These two
components are isomorphic --- any involution $\psi$ commuting with $\phi$ 
such that $\psi(w_1) = w_2$ gives an isomorphism --- so we obtain the
same answer for both components.
If $N$ is even, the two components
are non-isomorphic and correspond to cases (2) and (3).
\end{proof}

We now describe a few of the different interpretations of the
problem.  Call $x \in X$ an \defn{even-odd} point if $x$ has a 
basis $f_1(z), \dots, f_d(z)$ where each $f_i(z)$ is either an even 
polynomial or an odd polynomial.

\begin{proposition}
\label{prop:evenodd}
Let $h(z)$ be an even or an odd polynomial $N$ with distinct roots.  
The even-odd points in $X(h(z))$ are exactly the
the $\altreflect$-fixed points.
\end{proposition}

\begin{proof} An even-odd point in $X$ is clearly $\altreflect$-fixed.
Conversely, any $\altreflect$-fixed subspace of $\poln$ has a basis of
$\altreflect$-eigenvectors, and hence is an even-odd point.
\end{proof}

\begin{proposition}
\label{prop:imaginary}
Let $h(z)$ be a real polynomial whose roots are all pure imaginary
numbers.  Then the real points in $X(h(z))$ are exactly the
$\altreflect$-fixed points.
\end{proposition}

\begin{proof}
Let $\psi = \left(\begin{smallmatrix} i & 0 \\ 0 & 1 \end{smallmatrix}\right) 
\in \PGL_2(\CC)$.  We claim that $x \in X(h(z))$ is real point if and 
only if $\psi(x)$ is an even-odd point.  The result then follows from
the fact that $\psi$ commutes with $\altreflect$.

The polynomial $\psi(h(z)) = h(iz)$ has real roots, and hence by 
Theorem~\ref{thm:MTV} every point $\psi(x) \in X(h(iz))$ is real.  
Suppose $\psi(x)$ has basis $f_1(z), \dots, f_d(z)$ where each
$f_j(z)$ is either an even polynomial or an odd polynomial.
Then $x$ has basis $f_1(iz), \dots, f_d(iz)$, each of which
is either real, or $i$ times a real polynomial.  Hence $x$ is a real
point.  

Conversely, suppose $f_1(iz), \dots, f_d(iz)$ are real polynomials
that form a basis for $x \in X(h(z))$.  Since $f_1(z), \dots, f_d(z)$
span a real vector space $\psi(x) \in X(h(iz))$, so do their complex 
conjugates.
Since $f_j(iz)$ is real, $\overline{f_j}(z) = f_j(-z)$, and so we deduce 
that $f_1(-z), \dots, f_d(-z)$ span $\psi(x)$.  Thus
$\psi(x)$ is spanned by $f_1(z) \pm f_1(-z), \dots, f_d(z) \pm f_d(-z)$,
which are odd or even polynomials, i.e. $\psi(x)$ is an even-odd point.
\end{proof}

\begin{proposition}
\label{prop:onsomecircle}
Let $h(z)$ be a real polynomial with $N$ roots lying 
on a circle $O$ that is symmetrical with respect to the real axis.  
There is an involution $\phi \in \PGL_2(\CC)$ that restricts to 
complex conjugation on the circle $O$.  
The real points in $X(h(z))$ are the $\phi$-fixed points.
\end{proposition}

Hence, by Lemma~\ref{lem:involutioncases} the number of real points 
in $X(h(z))$ depends only the number of real roots of $h(z)$.
If $O = S^1$ is the unit circle, then this involution is
$\phi = \stdreflect$.

\begin{proof}
There is a M\"obius transformation $\psi \in \PGL_2(\RR)$ that
sends the imaginary line to $O$.  The involution 
$\phi = \psi \altreflect \psi^{-1}$ restricts to complex conjugation
on $O$.  Since $x \in X(h(z))$ is real and $\phi$-fixed if and only 
if $\psi^{-1}(x)$ is real and $\altreflect$-fixed,
this result is equivalent to Proposition~\ref{prop:imaginary}.
\end{proof}

\begin{proposition}
\label{prop:evenreal}
Suppose $N$ is even.
Let $h(z)$ be a real, even polynomial with $N$ distinct non-real 
roots on the
unit circle.  Then the number of even-odd points in $X(h(z))$
is equal to the number of real points in $X(h(z))$.
\end{proposition}

\begin{proof}
We consider two involutions: $\altreflect$, whose fixed
points in $\CP^1$ are $0$, $\infty$, and $\stdreflect$, whose
fixed points in $\CP^1$ are $\pm 1$.  If 
$h(z) = \prod_{a_i \neq \infty} (z+a_i)$ then the set $\bolda$ is 
fixed by both $\altreflect$ and $\stdreflect$, and none of the
elements of $\bolda$ are fixed by either.  It follows from
Lemma~\ref{lem:involutioncases} that the number of $\altreflect$-fixed 
points
in $X(h(z))$ is equal to the number $\stdreflect$-fixed points
in $X(h(z))$.  But the former are even-odd points, and the
latter are real points.
\end{proof}

This last result has an interesting combinatorial interpretation.
By Lemma~\ref{lem:evacpromote} the real points 
in $X(h(z))$ correspond to $\evac$-fixed tableaux in $\SYT(\Rect)$. 
Note that $\evac$-fixed
tableaux are rotationally-invariant tableaux, so this statement
is in agreement with both Theorems~\ref{thm:D-ribbon} 
and~\ref{thm:D-promotion} (here we
are considering the case where $r=1$, and $h(z)$ is of type (2)).
On the other hand, by Theorem~\ref{thm:promotion}, the odd-even points 
correspond to $\promote^{N/2}$-fixed tableaux.  
Hence from Proposition~\ref{prop:evenreal}, we deduce:

\begin{corollary}
\label{cor:combinv1}
If $N$ is even, the number of $\evac$-fixed tableaux in $\SYT(\Rect)$ 
is equal to the number the number of $\promote^{N/2}$-fixed tableaux
in $\SYT(\Rect)$.
\end{corollary}

The $r=2$ case of Theorem~\ref{thm:ribbon} asserts that the fixed
points of an involution are also counted domino tableaux.
We can prove the $r=2$ case of Theorem~\ref{thm:ribbon} using 
Lemma~\ref{lem:evacpromote}.

\begin{theorem}[$r=2$ case of Theorem~\ref{thm:ribbon}]
\label{thm:specialribbon}
Let $h(z)$ be an even or an odd polynomial of degree $N$ or $N-1$,
with distinct real roots.  There is a unique $h(z)$-compatible Richardson
variety $X_{\lambda/\mu}$, and the number of even-odd points in 
$X_{\lambda/\mu}(h(z))$ is the number of domino tableaux of shape
$\lambda/\mu$.
\end{theorem}

Here the involution we use is $\altreflect$, which is the
generator for $C_2$.
The four possibilities for $\lambda/\mu$ are: $\Rect$ (if $h(z)$ is even
and $N$ is even), $\Rect/1$ (if $h(z)$ is odd and $N$ is odd),
$1^\vee$ (if $h(z)$ is even and $N$ is odd), $1^\vee/1$ (if $h(z)$ is 
odd and $N$ is even).  The two cases where $N$ is odd give the same
answer; here $\bolda$ contains either $0$ or $\infty$, but not both.
When $N$ is even and $\lambda/\mu = \Rect$, we are in the
case where $\bolda$ contains neither $0$ nor $\infty$.
When $N$ is even and $\lambda/\mu = 1^\vee/1$, $\bolda$ contains both 
$0$ and $\infty$.  Hence these cases correspond exactly to the
different possibilities for the number of $\altreflect$-fixed elements
in $\bolda$.

\begin{proof}
First, suppose that $\lambda/\mu = \Rect$.
Then we can write $\bolda = \{a_1, \dots, a_N\}$, where 
$0 < a_1 < a_3 < \dots < a_{N-1} < \infty$, 
and $a_{2k} = -a_{2k-1}$ for all $k$.
Let $T = T_x(\bolda)$ be the real valued tableau associated to a point 
$x \in X(\bolda)$.  To compute $\altreflect T_x$, we write
\[
  \begin{pmatrix} -1 & 0 \\ 0 & 1 \end{pmatrix} =
  \begin{pmatrix} 1 & \varepsilon \\ 0 & 1 \end{pmatrix}
  \begin{pmatrix} -1 & 0 \\ 0 & 1 \end{pmatrix}
  \begin{pmatrix} 1 & \varepsilon \\ 0 & 1 \end{pmatrix}
\]
where $\varepsilon$ is a small positive real number.
Then $\left(\begin{smallmatrix} 
1 & \varepsilon \\ 0 & 1 
\end{smallmatrix}\right) \bolda 
= \{a_1 + \varepsilon, \dots, a_N + \varepsilon\}$.  Note that
\begin{gather*}
 a_1 \prec a_2 \ \prec\  a_3 \prec a_4 \ \prec\  \dotsb\  \prec\  a_{N-1} \prec a_N \\
 a_2 + \varepsilon \prec a_1 + \varepsilon \ \prec \ a_4 + \varepsilon 
\prec a_3 + \varepsilon \ \prec\ \dotsb\ \prec\ a_N + \varepsilon \prec a_{N-1} + 
\varepsilon
\end{gather*}
By Theorem~\ref{thm:sliding}, 
$\left(\begin{smallmatrix} 
1 & \varepsilon \\ 0 & 1 
\end{smallmatrix}\right)T$ is obtained from $T$ 
by switching the positions of $a_{2k-1}$ and $a_{2k} = -a_{2k-1}$
if they are in the same row and column; then adding $\varepsilon$
to each of the entries.  The action of $\altreflect$ on the
result simply negates each of the entries by 
Lemma~\ref{lem:tworeflections}(i); finally acting by 
$\left(\begin{smallmatrix} 
1 & \varepsilon \\ 0 & 1 
\end{smallmatrix}\right)$ the second time does not change 
the $\preceq$-order
of the entries, so we simply add $\varepsilon$ to each entry.  
The net result is that $\altreflect T$ is obtained 
from $T$
by switching the positions of $a_{2k-1}$ and $a_{2k}$, whenever
these are \emph{not} in the same row or column.

It follows that $\phi(x) = x$ if and only if the corresponding 
tableau $T_x(\bolda)$
has the property that $a_{2k-1}$ and $a_{2k}$ are in the same
row or column for all $k = 1, \dots, \frac{N}{2}$.  These real
valued tableaux are in bijection with domino tableaux of shape $\Rect$.

In the other three cases, $\bolda$ is of the form
\newsavebox{\mysavebox}
\sbox{\mysavebox}{$\{0, a_1, \dots, a_{N-2}, \infty\}$}
\begin{alignat*}{3}
&&& \makebox[\wd\mysavebox][c]{$\{0, a_1, \dots, a_{N-1}\}$}  
&\qquad&\text{if }\lambda/\mu = \Rect/1\,, \\
&&& \makebox[\wd\mysavebox][c]{$\{a_1, \dots, a_{N-1}, \infty\}$} 
&\qquad&\text{if }\lambda/\mu = 1^\vee\,, \\
&\text{or}&\qquad& \usebox{\mysavebox}
&\qquad&\text{if }\lambda/\mu = 1^\vee/1\,,
\end{alignat*}
where $0 < a_{2k-1} = -a_{2k} < \infty$ for all $k$.   
A similar argument applies.
\end{proof}

\begin{corollary}
\label{cor:combinv2}
If $N$ is even, the number of rotationally-invariant tableaux in 
$\SYT(\Rect)$
is equal to the number of domino tableaux of shape $\Rect$.
\end{corollary}

A similar argument gives the $r=1$ case of Theorem~\ref{thm:D-ribbon}.

\begin{theorem}[$r=1$ case of Theorem~\ref{thm:D-ribbon}]
\label{thm:specialD-ribbon}
Let $h(z)$ be polynomial of degree $N$, whose roots are distinct real
numbers invariant under $z \mapsto \frac{1}{z}$.  If $h(1) \neq 0$
or $h(-1) \neq 0$, then the the number of
$\stdreflect$-fixed points in $X(h(z))$ is the number of
rotationally-tableaux in $\SYT(\Rect)$.  If $h(1) = h(-1) = 0$,
then the number of $\stdreflect$-fixed points in $X(h(z))$ is
the number of rotationally-invariant tableaux in $\SYT(\Rect)$
where $\frac{N}{2}$ and $\frac{N}{2}{+}1$ are in the same row
or column.
\end{theorem}

\begin{proof}
Write $\bolda = \{a_1, \dots, a_N\}$, where
$a_1 \prec a_2 \prec \dots \prec a_N$.  Let $x \in X(\bolda)$, with
corresponding real valued tableau $T = T_x(\bolda)$
It is enough to prove this when $h(z)$ is generic, so that
$a_i \neq -a_j$ unless $\{a_i, a_j\} = \{1, -1\}$.

In the first case, when $h(1) \neq 0$ or $h(-1) \neq 0$, this implies
that $|a_1| < |a_2| < \dots < |a_N|$.  By 
Lemma~\ref{lem:tworeflections}(ii), $\stdreflect T = T^\vee$.
and so $\stdreflect x = x$ if and only if the corresponding
tableau is rotationally-invariant.

In the second case, when $h(1) = h(-1) = 0$, we write
\[
  \begin{pmatrix} 0 & 1 \\ 1 & 0 \end{pmatrix} =
  \begin{pmatrix} 1 & 0 \\ -\varepsilon & 1 \end{pmatrix}
  \begin{pmatrix} 0 & 1 \\ 1 & 0 \end{pmatrix} 
  \begin{pmatrix} 1 & \varepsilon \\ 0 & 1 \end{pmatrix}\,,
\]
where $\varepsilon$ is a small positive real number.  Arguing as
in the proof of Theorem~\ref{thm:specialribbon}, we see that 
$\stdreflect T$ is obtained by rotating $T$ by $180^\circ$,
replacing each entry $a_k$ by $\frac{1}{a_k}$, and switching
the positions of $1$ and $-1$ if and only if they are in the
same row or column. For $T$ to be invariant under this operation,
we require rotation-invariance of $T_x$, and the
entries $a_{N/2} = 1$ and $a_{(N/2)+1} = -1$ of $T_x(\bolda)$ must 
be in the same row or column. 
\end{proof}

Corollaries~\ref{cor:combinv1} and~\ref{cor:combinv2} can also
be proved bijectively.  The bijection between $\evac$-fixed and
$\promote^{N/2}$-fixed tableaux in $\SYT(\Rect)$ is given by
the involution $\evac_{N/2}: \SYT(\Rect) \to \SYT(\Rect)$, which
is the evacuation
operator restricted to the entries $1, \dots, \frac{N}{2}$.

\begin{proposition}
If $N$ is even, we have
\[
\evac \circ \evac_{N/2} = 
\evac_{N/2} \circ \promote^{N/2}\,.
\]
Hence $\evac_{N/2}$ gives a bijection between $\evac$-fixed 
and $\promote^{N/2}$-fixed tableaux in $\SYT(\Rect)$.
\end{proposition}

\begin{proof}
Let
$\evac'_{N/2} := \evac \circ \evac_{N/2} \circ \evac$.
For $T \in \SYT(\Rect)$,
one can think of $\evac'_{N/2}(T)$ as performing ``reverse evacuation'' 
on the largest $\frac{N}{2}$ entries of $T$.
It is enough to show that
$\evac'_{N/2} \circ \evac = 
\evac_{N/2} \circ \promote^{N/2}$.
But as shown in Figure~\ref{fig:modelbijection1}
the two sides of this equation are modelled by lifts of homotopic paths.
\end{proof}

\begin{figure}[tb]
\centering
\setlength\unitlength{.05\circlesize}
\begin{picture}(60,50)(0,0)
\put(0,15){\includegraphics[height=\circlesize]{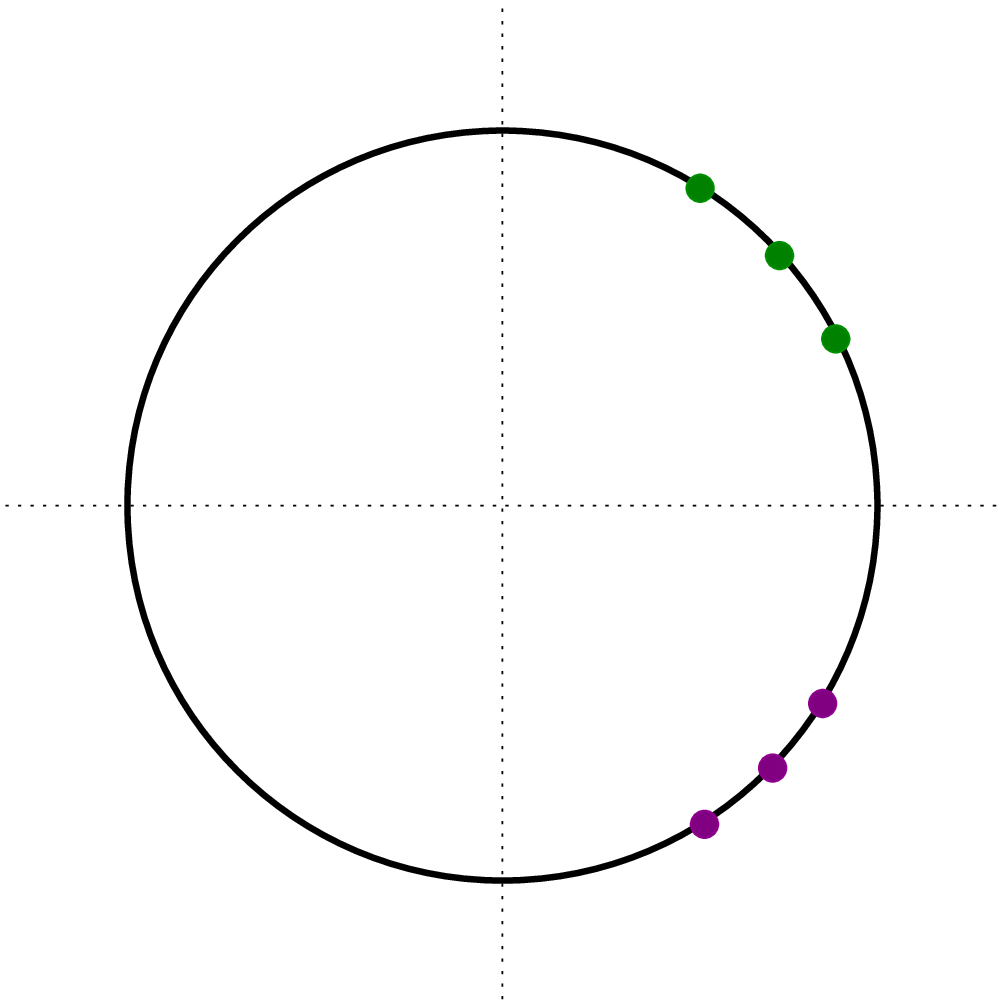}}
\put(16,30){\includegraphics[height=\circlesize]{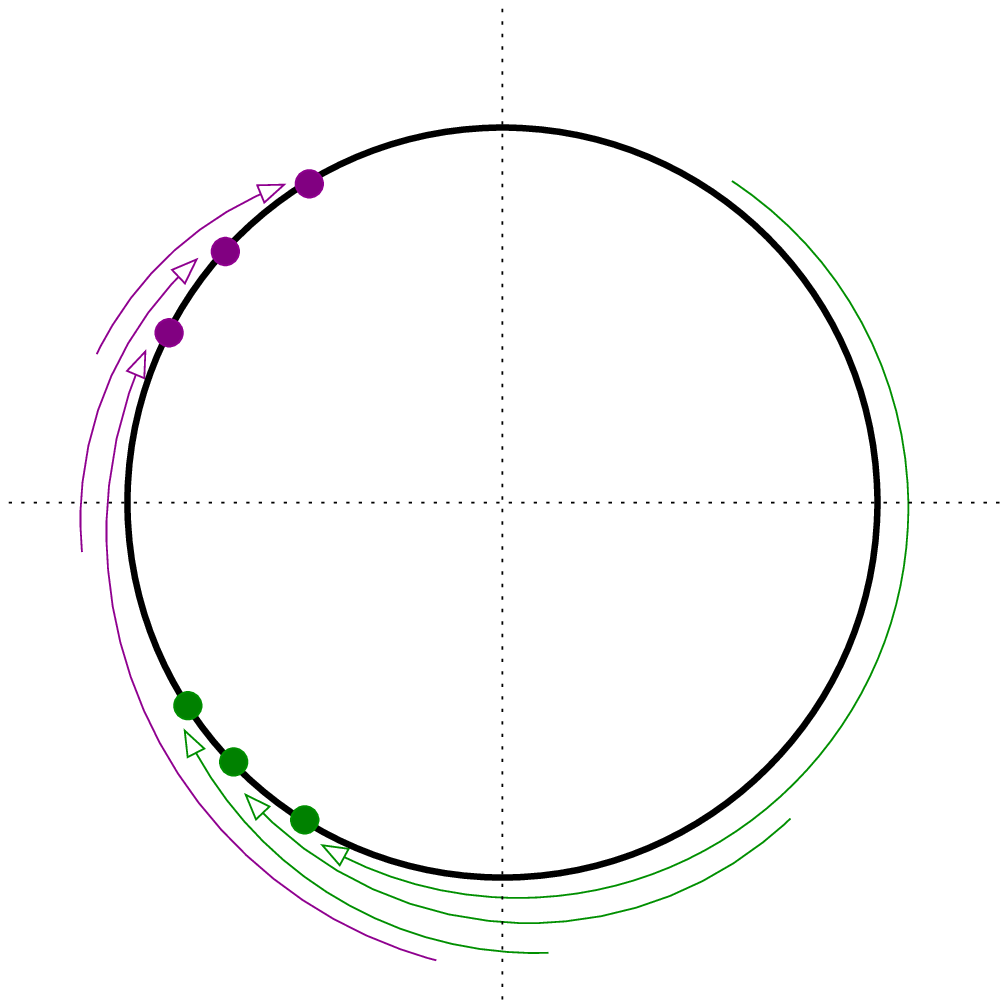}}
\put(40,27){\includegraphics[height=\circlesize]{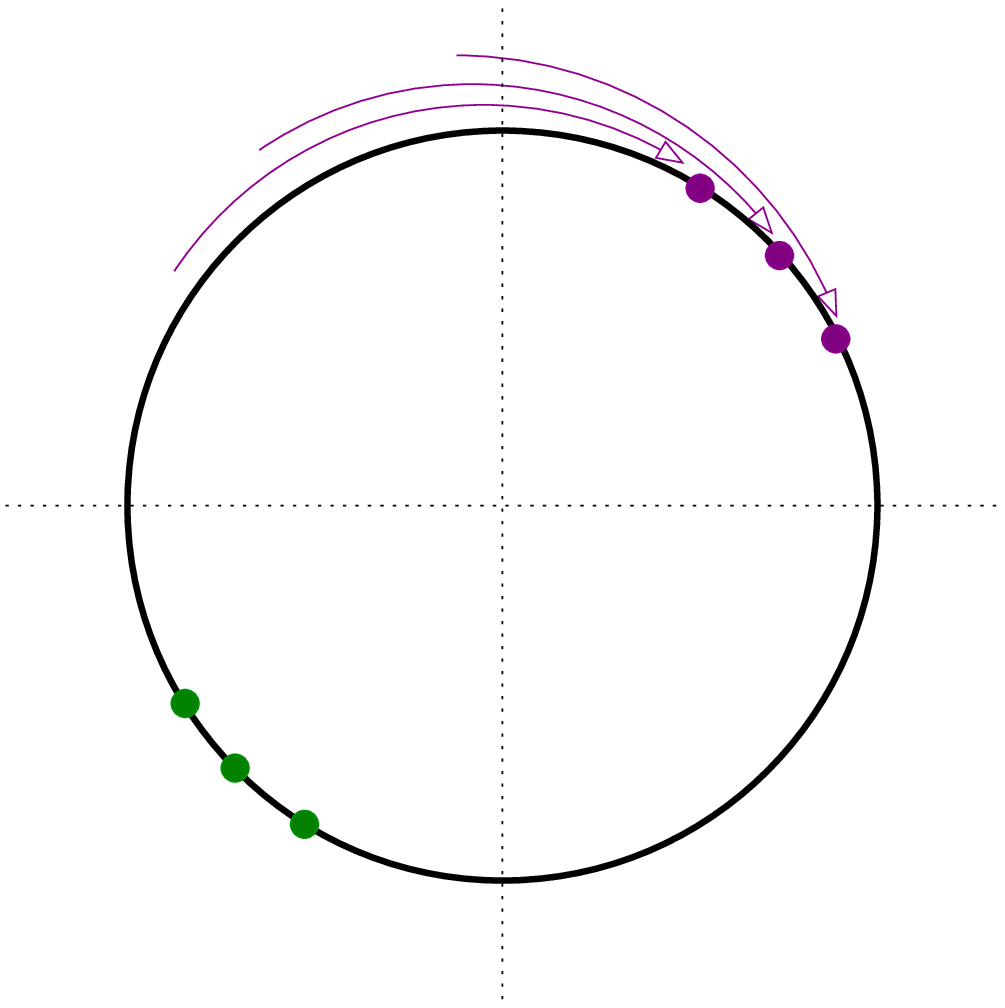}}
\put(16,0){\includegraphics[height=\circlesize]{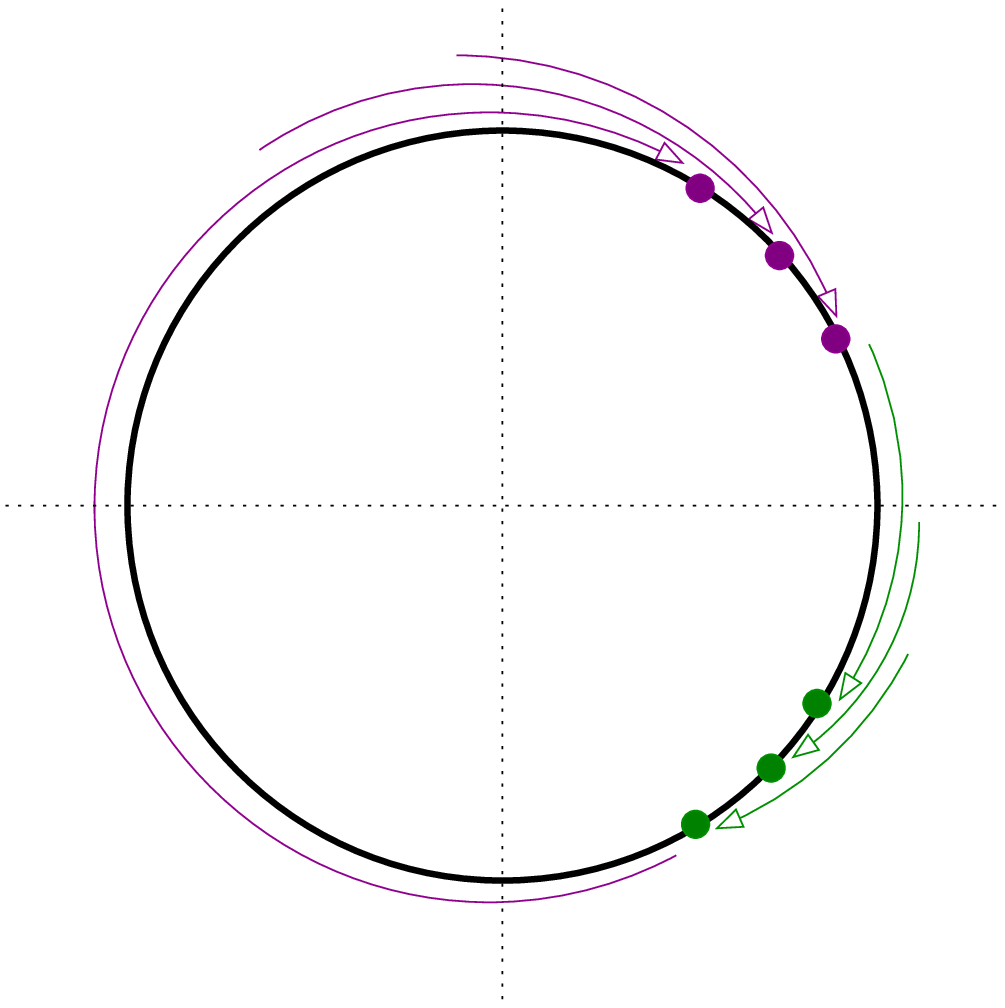}}
\put(40,3){\includegraphics[height=\circlesize]{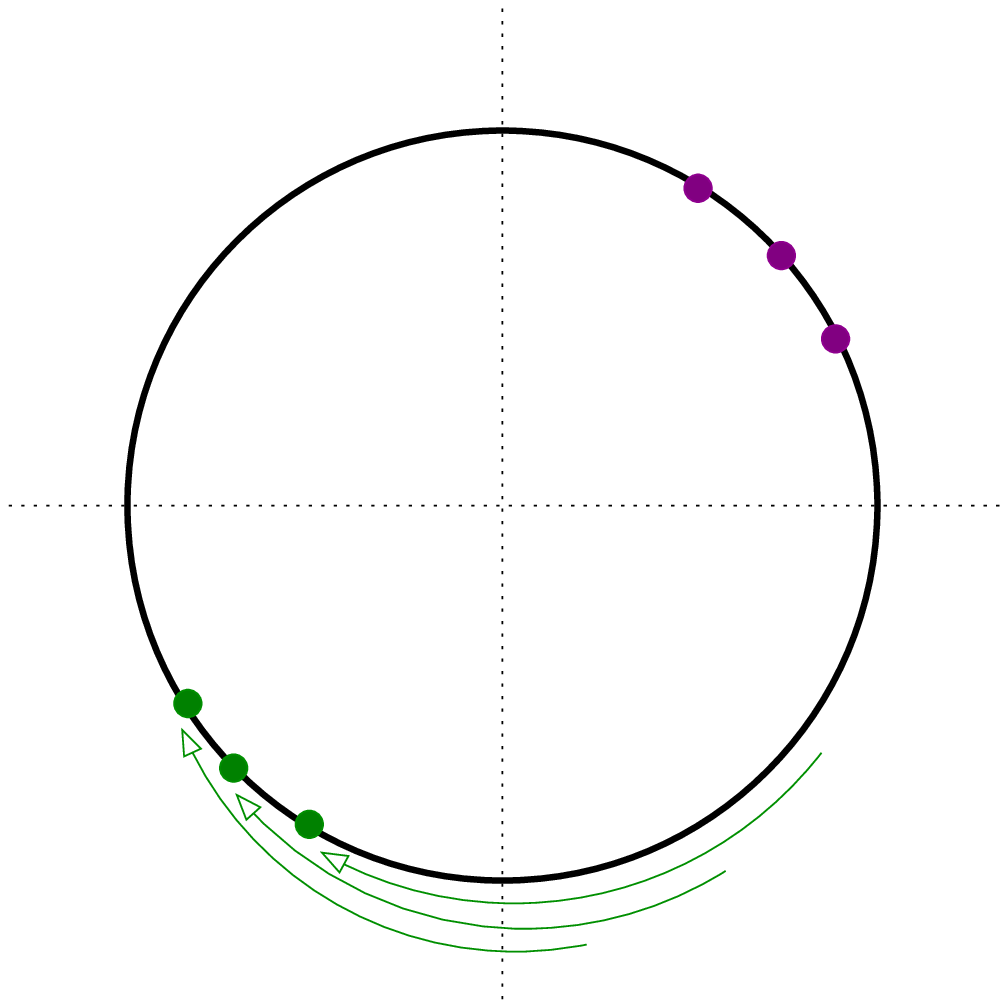}}
\put(13,27.5){\makebox(10,10){\begin{turn}{45}$\leadsto$\end{turn}}}
\put(16.5,33.5){$\evac$}
\put(13,12.5){\makebox(10,10){\begin{turn}{-45}$\leadsto$\end{turn}}}
\put(18.25,18.75){$\promote^{N/2}$}
\put(33,33.5){\makebox(10,10){\begin{turn}{-10}$\leadsto$\end{turn}}}
\put(37,40){$\evac'_{N/2}$}
\put(33,6.5){\makebox(10,10){\begin{turn}{10}$\leadsto$\end{turn}}}
\put(36,12.75){$\evac_{N/2}$}
\end{picture}
\caption{Homotopic paths that model
$\evac'_{N/2} \circ \evac$ (top) and 
$\evac_{N/2} \circ \promote^{N/2}$ (bottom).}
\label{fig:modelbijection1}
\end{figure}

The bijection for Corollary~\ref{cor:combinv2} is a bit more involved.
If $T$ is a domino tableau of shape $\Rect$, there are two entries labelled
$k$ for each $k=1, \dots, \frac{N}{2}$.  One of these is closer to the
northwest corner; colour this entry blue and the other one red.
With each $k$, starting at $k= \frac{N}{2}$ and descending to 
$k=1$, slide the red $k$ through the blue entries 
$k{+}1, \dots, \frac{N}{2}$, and then replace the red $k$ by
$k+\frac{N}{2}$.  The result will be a standard Young tableau of shape
$\Rect$.  Finally apply $\evac_{N/2}$ to this tableau, and denote
the result $\varphi(T)$.

\begin{proposition}
The map $T \mapsto \varphi(T)$ gives a bijection between domino tableaux 
of shape $\Rect$ and rotationally-invariant tableaux in $\SYT(\Rect)$.
\end{proposition}

\begin{proof}
Let $\bolda = \{a_1, \dots, a_N\}$, where 
$0 < a_1 < a_3 < \dots < a_{N-1} < 1$, and $a_{2k} = -a_{2k-1}$ for
$k=1, \dots, \frac{N}{2}$.  As we saw in the proof of 
Theorem~\ref{thm:specialribbon}, domino tableaux are in bijection with
real valued tableaux of shape $\Rect$ with entries $\bolda$, where
$a_{2k-1}$ and $a_{2k}$ are in the same row or column, and these
tableaux correspond to $\altreflect$-fixed points in $X(\bolda)$.
In this set-up, the map $\varphi$ is modelled by the path 
illustrated in Figure~\ref{fig:modelbijection2}.

Let
\[
  \phi =  \begin{pmatrix}
  \frac{1}{\sqrt{2}} & - \frac{1}{\sqrt{2}} \\
  \frac{1}{\sqrt{2}} & \frac{1}{\sqrt{2}} 
  \end{pmatrix}\,.
\]
When $\RP^1$ is drawn as a circle 
(as in Figure~\ref{fig:modelpromote}), $\phi \in \PGL_2(\RR)$ acts by 
rotating the circle $90^\circ$ clockwise.  
The path illustrated in Figure~\ref{fig:modelbijection2} is homotopic
to a path that computes the action of $\phi$; hence under the
identifications above, the action of $\phi$ coincides with the 
map $\varphi$.

Let $T = T_x(\bolda)$ be the real valued tableau corresponding to a
point $x \in X(\bolda)$. 
Note that $\phi \altreflect \phi^{-1} = \stdreflect$.
Thus $x$ is an $\altreflect$-fixed point if and only if
$\phi(x)$ is a $\stdreflect$-fixed point in $X(\phi(\bolda))$.
By Lemma~\ref{lem:tworeflections}, the tableau corresponding to 
$\stdreflect\phi(x)$ is $\phi(T)^\vee$; hence $\phi(T)$
corresponds to a $\stdreflect$-fixed point if and only if
$\phi(T) = \phi(T)^\vee$.
It follows that $T$ is identified with a domino tableau if and only if
$\phi(T)$ is identified with a rotationally invariant tableau, 
and the result follows.
\end{proof}

\begin{figure}[tb]
\centering
\setlength\unitlength{.05\circlesize}
\begin{picture}(60,30)(0,0)
\put(0,0){\includegraphics[height=\circlesize]{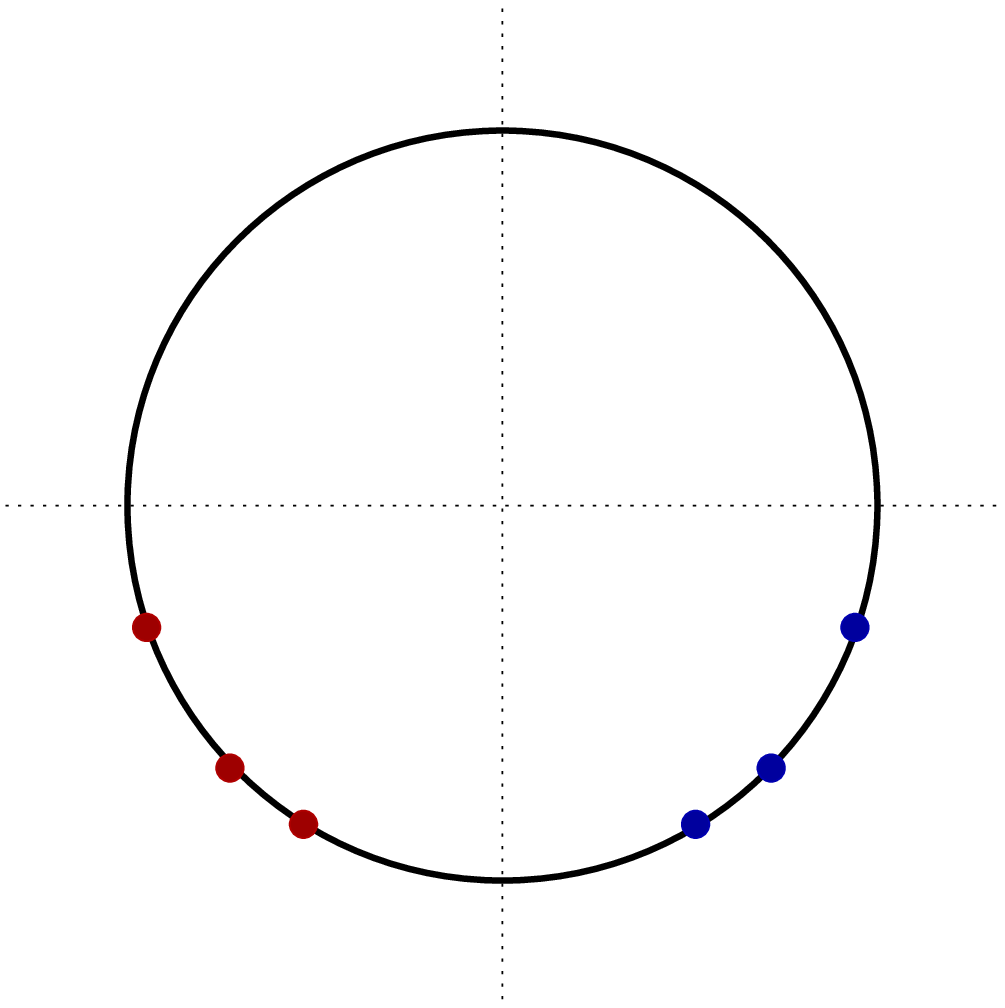}}
\put(20,10){\includegraphics[height=\circlesize]{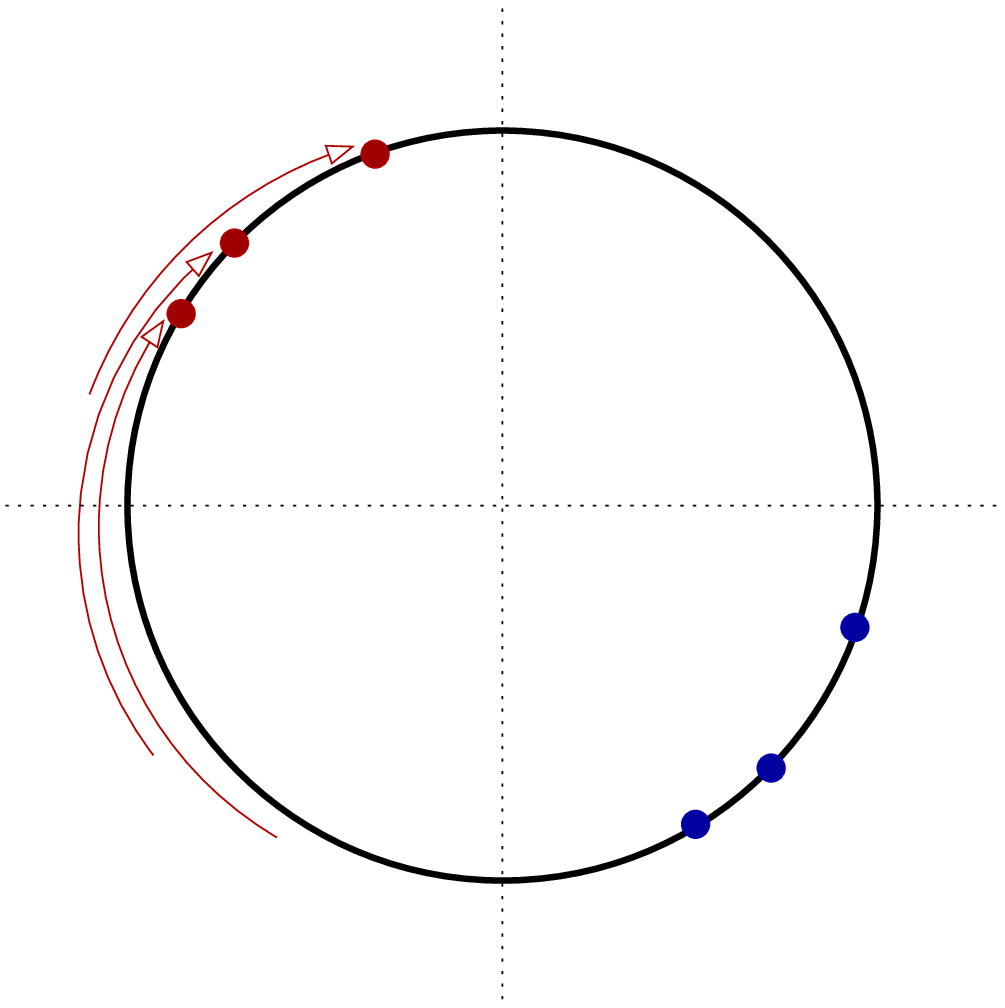}}
\put(40,0){\includegraphics[height=\circlesize]{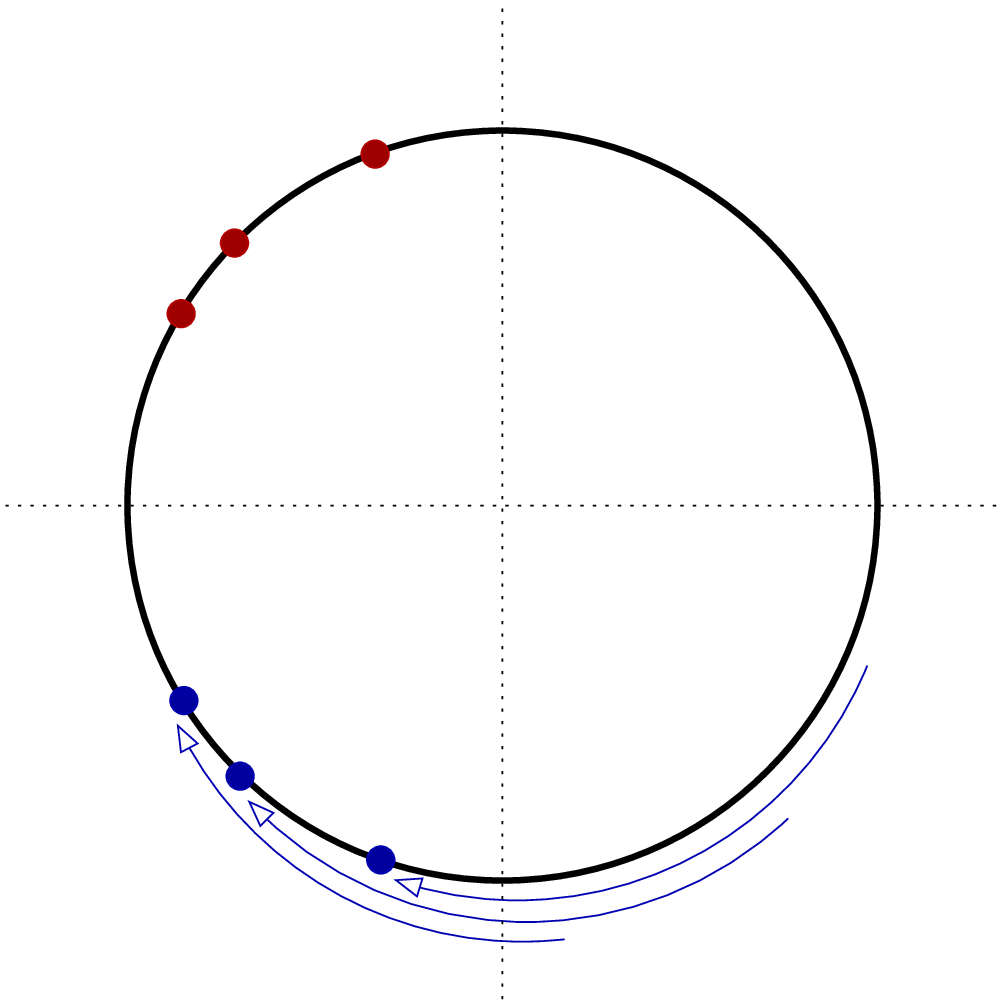}}
\put(15,10){\makebox(10,10){\begin{turn}{25}$\leadsto$\end{turn}}}
\put(35,10){\makebox(10,10){\begin{turn}{-25}$\leadsto$\end{turn}}}
\end{picture}
\caption{A path that models the bijection between domino tableaux
of shape $\protect\Rect$, and rotationally invariant tableaux of shape 
$\protect\Rect$, when $N$ is even.}
\label{fig:modelbijection2}
\end{figure}

Similar arguments can be used to establish explicit bijections
between domino tableaux and rotationally-invariant tableaux (of the
appropriate type) in cases (1) and (3).  
One can also show that the procedure used to define
$\varphi(T)$ gives a bijection between 
domino tableau of shape $\lambda$
and the $\evac$-fixed tableaux in $\SYT(\lambda)$ for any partition
$\lambda$ of even size.
We leave the details of these generalizations to the reader.

\subsection{Involutions and Schubert intersections}
\label{sec:schubert}

The correspondence between points in $X(\bolda)$ and tableaux can be
extended to the case where $\bolda$ is a multiset of points in $\RP^1$.
In this more general situation, 
the correspondence is no longer bijective.
Nevertheless, we can still use it to obtain generalizations of
Theorems~\ref{thm:specialribbon} and~\ref{thm:specialD-ribbon} to the 
case where $\bolda$ is a multiset fixed by an involutions 
$\phi \in \PGL_2(\CC)$.
Although the problem of counting $\phi$-fixed points in $X(\bolda)$ 
still makes some sense, the fibre $X(\bolda)$ is not necessarily 
reduced when $\bolda$ is a multiset.  
We resolve this by refining our problem, and considering
$\phi$-fixed points in an intersection of Schubert varieties.

For each $a \in \CP^1$, we define a flag
in $\poln$:
\[
  F_\bullet(a) 
  \ :\ \{0\} = F_0(a) \subset F_1(a) \subset \dots 
  \subset F_{n-1}(a)  = \poln\,.
\]
If $a \in \CC$, 
\[
   F_i(a) := (z+a)^{n-i}\CC[z] \cap \poln\,.
\]
In particular $F_\bullet(0)$ is the flag $F_\bullet$ defined in
the introduction.  We also define $F_\bullet(\infty)$ to be the flag 
$\TF_\bullet$ (defined in the Section~\ref{sec:introwronskian}), 
which is equal to $\lim_{a \to \infty} F_\bullet(a)$.
For every $\lambda \in \Lambda$, we have a \defn{Schubert variety} in $X$
relative to the flag $F_\bullet(a)$:
\[
  \Omega^\lambda(a) 
  := \{x \in X \mid \dim \big(x \cap F_{n-d-\lambda^i+i}(a) \big) \geq i\,,
  \text{ for $i=1, \dots, d$}\}\,.
\]
The codimension of $\Omega^\lambda(a)$ in $X$ is $|\lambda|$.

\begin{proposition}
\label{prop:wronskischubert}
The Wronskian $\Wr(x;z)$ is divisible by
$(z+a)^k$ if and only if $x \in \Omega^\lambda(a)$ for some partition
$\lambda \vdash k$.  Also, $x \in \Omega^\mu(\infty)$ for
some $\mu \vdash \big(N - \deg \Wr(x;z)\big)$.
\end{proposition}

Hence every point in $X(\bolda)$ lies in some intersection of Schubert
varieties relative to the flags $F_\bullet(a_i)$.  
(We will not discuss the multiplicities of such points here, but
more information can be found in \cite{Pur-Gr}.)
We will are left to consider the problem of counting $\phi$-fixed points 
in an intersection
\begin{equation}
\label{eqn:schubertintersection}
\Omega^{\lambda_1}(a_1) \cap \dots \cap \Omega^{\lambda_s}(a_s)\,,
\end{equation}
where $a_1, \dots, a_s$ are distinct, and 
$|\lambda_1| + \dots + |\lambda_s| = N$.
Mukhin, Tarasov
and Varchenko proved that if $\bolda$ is real, such intersections
of Schubert varieties are always transverse \cite{MTV2}.

\begin{theorem}
\label{thm:MTV2}[Mukhin-Tarasov-Varchenko]
If $a_1, \dots a_s \in \RP^1$ are distinct real points, and
$\lambda_1, \dots \lambda_s \in \Lambda$ are partitions with
$|\lambda_1| + \dots + |\lambda_s| = N$, then the
intersection \eqref{eqn:schubertintersection}
is finite, transverse, and real.
\end{theorem}

Let us assume that $a_1, \dots, a_s$ are real, and
$a_1 \prec a_2 \prec \dots \prec a_s$.
Let $x$ be a point in the intersection \eqref{eqn:schubertintersection}.
Instead of associating a unique tableau to the point $x$, we will
associate a set of tableaux $\calT_x \subset \SYT(\Rect)$.  We do this 
by the 
reverse of the construction described in~\cite{Pur-shifted}.
Let $\bolda$ denote the multiset containing $a_1, \dots, a_s$
with multiplicities $|\lambda_1|, \dots, |\lambda_s|$ respectively,
and let $x \in X(\bolda)$.
Consider a path
$\bolda_t = \{(a_1)_t \preceq \dots \preceq (a_N)_t\}$, $t \in [0,1]$, 
starting at $\bolda_0 = \bolda$, and ending at a set $\bolda_1$.
(Essentially we want to perturb $\bolda$ without changing
the relative $\preceq$-order of the entries.)
For any lifting to a path
$x_t \in X(\bolda_t)$ with $x_0 = x$, it is reasonable to think
of $T_{x_1}$ as being a tableau associated to $x$.  However, since
$\bolda_0$ is a multiset, the fibre $X(\bolda_0)$ need not
be reduced, and so there may be more than one such lifting.
We define
\begin{equation}
\label{eqn:setoftableaux}
  \calT_x := \{T_{x_1} \mid x_t \in X(\bolda_t),\ x_0 =x\}\,,
\end{equation}
the set of tableaux coming all possible liftings of
the path $\bolda_t$.

The sets $\calT_x$ that arise in this way are characterized in 
terms of Haiman's dual equivalence relation \cite{Hai}.
We will not review all the relevant definitions here but refer the reader
to~\cite[Section 2]{Pur-shifted}, where they may be found along with
the proof of Theorem~\ref{thm:dualequiv}, below.
Let $\boldlambda$ denote the sequence of 
partitions $(\lambda_1, \dots, \lambda_s)$.
For $T \in \SYT(\Rect)$, $k =1, \dots, s$, 
let $T_{\boldlambda}[k]$ be the subtableau of $T$ consisting of entries
\[
|\lambda_1| + \dots +|\lambda_{k-1}| +1\,,\,%
|\lambda_1| + \dots +|\lambda_{k-1}| +2\,,\, \dots\,,\,
|\lambda_1| + \dots +|\lambda_k|\,.
\]
Write $T \sim^*_{\boldlambda} T'$ if $T_{\boldlambda}[k]$ is dual equivalent
to $T'_{\boldlambda}[k]$ for all $k = 1, \dots, s$.  This is an 
equivalence relation on $\SYT(\Rect)$.  
We will say that $T$ has \defn{type} $\boldlambda$ 
if $T_{\boldlambda}[k]$ has rectification shape $\lambda_k$ for all $k$.  
Note that 
if $T$ has type $\boldlambda$ and 
$T \sim^*_{\boldlambda} T'$, then $T'$ also has type $\boldlambda$.

\begin{theorem}
\label{thm:dualequiv}
For every point $x \in X(\bolda)$,
the associated set of tableaux $\calT_x$
is an equivalence classes of the relation $\sim^*_{\boldlambda}$.
The point $x$ is in the intersection \eqref{eqn:schubertintersection}
if and only if $T$ has type $\boldlambda$ for
some (equivalently for every) $T \in \calT_x$.
\end{theorem}

For $T \in \SYT(\Rect)$, \defn{switching} the
subtableaux $T_{\boldlambda}[k]$ and $T_{\boldlambda}[k{+}1]$
is the following procedure (see e.g.~\cite{BSS}).
First add $|\lambda_k|$ to each of
the entries of $T_{\boldlambda}[k]$, and subtract $|\lambda_k|$ from
each of the entries of $T_{\boldlambda}[k{+}1]$.  
Then, slide each of the boxes
of $T_{\boldlambda}[k]$, from largest entry to smallest, through 
$T_{\boldlambda}[k{+}1]$.  
The result is
a new tableau in $\SYT(\Rect)$.
Let $\switch(T, \bolda)$ be the tableau obtained by switching
$T_{\boldlambda}[k]$ and $T_{\boldlambda}[k{+}1]$ for each $k$
such that $a_k = -a_{k+1}$.

\begin{theorem}
\label{thm:tworeflectionsgeneral}
Let $x$ be a point in the intersection \eqref{eqn:schubertintersection}.
\begin{packedenum}
\item[(i)] The point $\altreflect x$ lies in the intersection
\begin{equation}
\label{eqn:altreflectintersection}
\Omega^{\lambda_1}(-a_1) \cap \dots \cap \Omega^{\lambda_s}(-a_s)\,,
\end{equation}
and the associated set of tableaux
is $\{\switch(T,\bolda) \mid T \in \calT_x\}$.
\item[(ii)] The point $\stdreflect x$ lies in the intersection
\begin{equation}
\label{eqn:stdreflectintersection}
\Omega^{\lambda_1}(a_1^{-1}) \cap \dots \cap \Omega^{\lambda_s}(a_s^{-1})\,,
\end{equation}
and the associated set of tableaux
is $\{\switch(T,\bolda)^\vee \mid T \in \calT_x\}$.  
\end{packedenum}
Here $T \mapsto T^\vee$ is the involution on $\SYT(\Rect)$ which rotates 
a tableau by $180^\circ$ and replaces each entry $i$ by $N{+}1{-}i$.
\end{theorem}

\begin{proof}
We will prove (i), leaving (ii), which is similar, to the reader.
Let $\hat \bolda = \altreflect \bolda$
and let $\hat x = \altreflect x$.  
Call a path $\bolda_t = \{(a_1)_t, \dots, (a_N)_t\}$ \defn{short} 
if $(a_k)_t \in \RP^1$ is in some sufficiently small neighbourhood 
of $(a_k)_0$, for all $t \in [0,1]$, $k = 1, \dots, N$.

The definition \eqref{eqn:setoftableaux} of $\calT_x$
requires us to consider a ``suitable'' path
$\bolda_t$, i.e. a path 
$\bolda_t = \{(a_1)_t \preceq \dots \preceq (a_N)_t\}$, $t \in [0,1]$, 
where $\bolda_0 = \bolda$ and $\bolda_1$ is a set. 
It is possible to choose this to be a short path, in which 
$(a_k)_1 < (a_k)_0$ for all $k$,
and we will assume that this is the case.  
This ensures that for any lifting $x_t \in X(\bolda_t)$,
the tableau associated to $\altreflect x_1$ is just $T_{x_1}$
(by Lemma~\ref{lem:tworeflections}(i)).  
Thus $\calT_x$ is also obtained from all possible liftings of the
path $\altreflect \bolda_t$, starting at $\hat x$.

Similarly, to obtain $\calT_{\hat x}$ using \eqref{eqn:setoftableaux}
we need a suitable 
path $\hat \bolda_t = \{(\hat a_1)_t , \dots ,(\hat a_N)_t\}$, $t \in [0,1]$. 
Again, we assume this is a short path, and that
$(\hat a_k)_1 < (\hat a_k)_0$ for all $k$.
Note that we have
\begin{equation}
\label{eqn:directionofperturbation}
(\hat a_k)_1 < (\hat a_k)_0 = -(a_k)_0 < - (a_k)_1
\qquad \text{for $k=1, \dots, N$}\,.
\end{equation}
  
Now, any two short paths from $\hat\bolda$ to $\hat \bolda_1$
are homotopy equivalent to each other.
In particular, the path $\hat \bolda_t$ is homotopy 
equivalent to the concatenation of $\altreflect \bolda_t$
with any short path of sets from $\altreflect \bolda_1$ to $\hat \bolda_1$;
hence we can compute $\calT_{\hat x}$ by lifting this concatenation of
paths, starting at $\hat x$.  As noted above, the first part 
(lifting $\altreflect \bolda_t$ to a path starting at $\hat x$) 
gives us $\calT_x$.
It follows that the relationship between $\calT_x$ and 
$\calT_{\hat x}$ is described by lifting a short path from
$\altreflect \bolda_1$ to $\hat \bolda_1$.
From \eqref{eqn:directionofperturbation}, we see that
along such a path the
points close to $a_k$ must change places (in the $\preceq$-order)
with points close to $a_{k+1}$, whenever $a_{k+1} = -a_k$.
By Theorem~\ref{thm:sliding}, the effect of this on a tableau 
$T \in \calT_x$ is to switch the
subtableaux $T_{\boldlambda}[k]$ and $T_{\boldlambda}[k{+}1]$, and
the result follows.
\end{proof}

Theorem~\ref{thm:tworeflectionsgeneral}
gives us two ways of computing the number of fixed points of
an involution $\phi \in \PGL_2(\CC)$ in the 
intersection \eqref{eqn:schubertintersection}.
First note, that for there to any fixed points, the intersection
itself must be $\phi$-invariant, which is to say we must have
\[
  \big\{(\lambda_1, a_1), \dots, (\lambda_s, a_s)\big\} = 
  \big\{(\lambda_1, \phi(a_1)), \dots, (\lambda_s, \phi(a_s))\big\}\,.
\]
If we take $\phi = \altreflect$, the number of fixed points is
the number of equivalence classes of $\sim^*_{\boldlambda}$
of type $\boldlambda$ that are invariant under 
$T \mapsto \switch(T,\bolda)$.
If we take $\phi = \stdreflect$, the number of fixed points is
the number of equivalence classes of $\sim^*_{\boldlambda}$
of type $\boldlambda$ that are invariant under 
$T \mapsto \switch(T,\bolda)^\vee$.

We note that the equivalence classes of $\sim^*_{\boldlambda}$ are
in bijection with sequences of Littlewood-Richardson tableaux,
which makes it quite manageable to compute moderate sized 
examples by hand, by either method.

\begin{example}
Let $d=3$, $n=7$, and $\lambda = 21$.
Suppose $w_1,w_2$ are distinct non-real complex numbers.
We consider the problem of counting the number of real points in the
intersection
\begin{equation}
\label{eqn:realschubertex}
   \Omega^\lambda(w_1) \cap \Omega^\lambda(\bar w_1)
   \cap \Omega^\lambda(w_2) \cap \Omega^\lambda(\bar w_2)\,.
\end{equation}
The points $w_1, \bar w_1, w_2, \bar w_2$ lie on a circle in $\CP^1$,
so by Proposition~\ref{prop:onsomecircle}, the real points in
\eqref{eqn:realschubertex} are the fixed points of an involution 
$\phi \in \PGL_2(\CC)$.  Hence this problem is equivalent to counting
the $\altreflect$-fixed points in the intersection
\[
   \Omega^\lambda(a_1) \cap \Omega^\lambda(-a_1)
   \cap \Omega^\lambda(a_2) \cap \Omega^\lambda(-a_2)\,,
\]
where $0 < a_1 < a_2 < \infty$.  
By Theorem~\ref{thm:tworeflectionsgeneral}(i), the number of these points
is the number of equivalence classes
of $\sim^*_{\boldlambda}$ of type 
$\boldlambda = (\lambda, \lambda, \lambda, \lambda)$ that are invariant
under $T \mapsto \switch(T, \bolda)$.

There are eight equivalence classes of $\sim^*_{\boldlambda}$
of type $\boldlambda$; representatives of these classes are shown
in Figure~\ref{fig:eighttableaux}.
To compute $T \mapsto \switch(T,\bolda)$, we switch 
$T_{\boldlambda}[1]$ with $T_{\boldlambda}[2]$,
and switch $T_{\boldlambda}[3]$ 
with $T_{\boldlambda}[4]$.
The four tableaux in the top row are invariant under
$T \mapsto \switch(T, \bolda)$; therefore so are the
equivalence classes they represent. 
For each tableaux $T$ in the bottom 
row, one can check that $T \nsim^*_{\boldlambda} \switch(T,\bolda)$.
Thus, exactly four of our equivalence classes are fixed by
$T \mapsto \switch(T, \bolda)$, and hence there four real points in the 
intersection \eqref{eqn:realschubertex}.  This answer agrees with
the experimental calculation in~\cite{HHMS}.

We can also obtain this answer by counting 
$\stdreflect$-fixed points in
the intersection
\[
   \Omega^\lambda(a_1) \cap \Omega^\lambda(a_2)
   \cap \Omega^\lambda(a_2^{-1}) \cap \Omega^\lambda(a_1^{-1})\,,
\]
where $|a_1| < |a_2| < 1$.  
By Theorem~\ref{thm:tworeflectionsgeneral}(ii), these are counted
by equivalence classes of $\sim^*_{\boldlambda}$ of type 
$\boldlambda$ that are invariant under
$T \mapsto \switch(T^\vee, \bolda) = T^\vee$.  
The four tableaux on the left
side of Figure~\ref{fig:eighttableaux} are representatives of these
invariant classes.
\end{example}
\begin{figure}[tb]
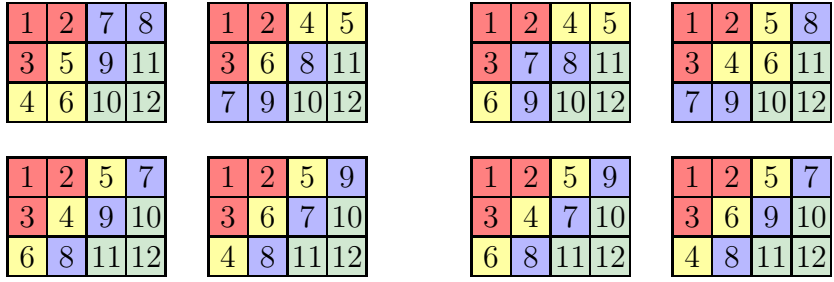

\centering
\begin{young}
!1 & !2 & ?7 & ?8   \\
!3 & ??5 & ?9 & !!11  \\
??4 & ??6 & !!10  & !!12
\end{young}
\quad
\begin{young}
!1 & !2 & ??4 & ??5   \\
!3 & ??6 & ?8 & !!11  \\
?7 & ?9 & !!10  & !!12
\end{young}
\qquad\quad
\begin{young}
!1 & !2 & ??4 & ??5   \\
!3 & ?7 & ?8 & !!11  \\
??6 & ?9 & !!10  & !!12
\end{young}
\quad
\begin{young}
!1 & !2 & ??5 & ?8   \\
!3 & ??4 & ??6 & !!11  \\
?7 & ?9 & !!10  & !!12
\end{young}\bigskip%
\\
\begin{young}
!1 & !2 & ??5 & ?7   \\
!3 & ??4 & ?9 & !!10  \\
??6 & ?8 & !!11  & !!12
\end{young}
\quad
\begin{young}
!1 & !2 & ??5 & ?9   \\
!3 & ??6 & ?7 & !!10  \\
??4 & ?8 & !!11  & !!12
\end{young}
\qquad\quad
\begin{young}
!1 & !2 & ??5 & ?9   \\
!3 & ??4 & ?7 & !!10  \\
??6 & ?8 & !!11  & !!12
\end{young}
\quad
\begin{young}
!1 & !2 & ??5 & ?7   \\
!3 & ??6 & ?9 & !!10  \\
??4 & ?8 & !!11  & !!12
\end{young}
\caption{Eight tableaux that are representatives of 
the eight $\sim^*_{\boldlambda}$ equivalence
classes of type $\boldlambda = (21,21,21,21)$, when $d=3$, $n=7$.
The subtableaux $T_{\boldlambda}[1]$, $T_{\boldlambda}[2]$, 
$T_{\boldlambda}[3]$ and $T_{\boldlambda}[4]$ 
consist of entries $(1,2,3)$, $(4,5,6)$, $(7,8,9)$ and $(10,11,12)$ 
respectively.}
\label{fig:eighttableaux}
\end{figure}
%


\section{$C_r$-fixed points of the Grassmannian}
\label{sec:fixedpoints}

\subsection{Pl\"ucker coordinates}
\label{sec:plucker}

Virtually all of our calculations in the remaining sections
of this paper will be done using the Pl\"ucker coordinates
for the Grassmannian.  Both
the Richardson variety $X_{\lambda/\mu}$ and the fixed point
set $X^r$ are characterized by the fact that certain Pl\"ucker 
coordinates are zero.  In describing of the latter, we begin
to see how $r$-ribbons and $r$-ribbon tableaux fit into the picture.

There are several different ways to index Pl\"ucker coordinates:
we will use $\Lambda$ as our indexing set.
For $\lambda \in \Lambda$, set
\[
  J(\lambda) := \{i-1+\lambda^{d+1-i} \mid 1 \leq i \leq d\}\,.
\]
Suppose $x \in X$ is the subspace spanned by 
polynomials $f_1(z), \dots, f_d(z)$.
Consider the $d \times n$ matrix $A_{ij} := [z^j]f_i(z)$, whose
entries are the coefficients of the polynomials $f_i(z)$.  Our convention
will be that
the rows of $A_{ij}$ are indexed by $i = 1, \dots, d$, while 
the columns are indexed $j = 0, \dots, n-1$.
The \defn{Pl\"ucker coordinates} of a point $x \in X$ are 
$[p_\lambda(x)]_{\lambda \in \Lambda}$, where
$p_\lambda(x) := A_{J(\lambda)}$ is the maximal minor of $A$ with
column set $J(\lambda)$.  These are homogeneous coordinates:
up to a scalar multiple, $[p_\lambda(x)]_{\lambda \in \Lambda}$ does 
not depend on the choice of basis.

We also define positive integer constants
\[
  q_\lambda := 
   \prod_{1 \leq i<j \leq d} (j-i+\lambda^{d+1-j}-\lambda^{d+1-i})\,,
\]
for $\lambda \in \Lambda$. 
The Wronski map can be written explicitly in terms of the Pl\"ucker
coordinates and the constants $q_\lambda$.  
\begin{proposition}[\protect{See \cite[Proposition 2.3]{Pur-Gr}}]
\label{prop:wrplucker}
For any point $x \in X$, 
the Wronskian $\Wr(x;z)$ is 
given in terms of the Pl\"ucker coordinates of $x$ by
\begin{equation}
\label{eqn:pluckerwronskian}
\Wr(x;z) = \sum_{\lambda \in \Lambda}  q_\lambda p_\lambda(x) z^{|\lambda|}\,.
\end{equation}
\end{proposition}

Both sides of \eqref{eqn:pluckerwronskian} are only defined up to a 
scalar multiple, hence this should rightly be interpreted as
an equation in $\PpolN$.

We note a few useful results about the relationship between
Pl\"ucker coordinates and Richardson varieties.  The next two
results are standard facts.

\begin{proposition}
\label{prop:richardsonplucker1}
Let $x \in X$.  There is a unique maximal $\lambda \in \Lambda$ 
such that $p_\lambda(x) \neq 0$, and a unique minimal $\mu \in \Lambda$,
such that $p_\mu(x) \neq 0$.  For these partitions we have
$x \in X_{\lambda/\mu}$. 
\end{proposition}

Let $\Lambda_{\lambda/\mu} := \{\nu \in \Lambda \mid 
\mu \subseteq \nu \subseteq \lambda\}$ denote the interval in the
poset $\Lambda$ between $\mu$ and $\lambda$.
The following theorem can be found in~\cite{HP}.

\begin{theorem}
\label{thm:richardsonplucker2}
As a subscheme of $X$, $X_{\lambda/\mu}$ is defined 
by the the equations $\{p_\nu = 0 \mid \nu \notin \Lambda_{\lambda/\mu}\}$.
In particular, for any point $x \in X_{\lambda/\mu}$, if
$p_\nu(x) \neq 0$ then $\nu  \in \Lambda_{\lambda/\mu}$.
\end{theorem}

\begin{proposition}
\label{prop:richardsonplucker3}
If $h(z) \in \pol{N}$, and $X_{\lambda/\mu}$ is
a compatible Richardson variety, then for any point in
$x \in X_{\lambda/\mu}(h(z)) := X_{\lambda/\mu} \cap X(h(z))$
we have $p_\mu(x) \neq 0$ and $p_\lambda(x) \neq 0$.
\end{proposition}

\begin{proof}
Since $\Wr(x;z) = h(z)$, taking coefficients of $z^{|\mu|}$ and
$z^{|\lambda|}$ in \eqref{eqn:pluckerwronskian}, we have
\[ 
  \sum_{|\nu| = |\mu|} q_\nu p_\nu(x) = [z^{|\mu|}]h(z)
  \qquad\text{and}\qquad
  \sum_{|\nu| = |\lambda|} q_\nu p_\nu(x) = [z^{|\lambda|}]h(z)\,.
\]
Since $X_{\lambda/\mu}$ is $h(z)$-compatible, the right hand sides 
are non-zero.
Finally, for $p_\nu(x)$ to be non-zero we must have 
$\nu \in \Lambda_{\lambda/\mu}$;
hence the first sum is just $q_\mu p_\mu(x)$, and
the second is $q_\lambda p_\lambda(x)$.  The result follows.
\end{proof}

From these facts, we deduce Proposition~\ref{prop:richardson}.

\begin{proof}[Proof of Proposition~\ref{prop:richardson}]
Let $x \in X(h(z))$.  Let $\mu$ be the unique minimal partition
such that $p_\mu(x) \neq 0$, and let $\lambda$ be the unique maximal
partition such that $p_\lambda(x) \neq 0$.  Then $x \in X_{\lambda/\mu}$.
Moreover from \eqref{eqn:pluckerwronskian}, 
we have $[x^k]h(z) = 0$ for $k > |\lambda|$ and $[x^{|\lambda|}]h(z) 
= q_\lambda p_\lambda(x) \neq 0$.  Thus $\deg h(z) = |\lambda|$,
and similarly $\mindeg h(z) = |\mu|$.  Thus $x$ lies in some 
$h(z)$-compatible Richardson variety.

If $x$ were in two $h(z)$-compatible Richardson varieties,
$x$ would lie in their intersection.  This is a proper Richardson
subvariety $X_{\lambda'/\mu'}$, where $\lambda' \subsetneq \lambda$,
or $\mu' \supsetneq \mu$.  But this contradicts
Theorem~\ref{thm:richardsonplucker2}, since $p_\mu(x) \neq 0$,
and $p_\lambda(x) \neq 0$.
\end{proof}

\subsection{Components of $X^r$}

The key to adapting the methods of~\cite{Pur-Gr} to the proof of
Theorem~\ref{thm:ribbon} is an understanding of the components 
of $X^r$.  We first show that these components are naturally
indexed by two sets.  The component of a point $x \in X^r$
can be indexed by an integer vector, determined by the
eigenvalues of the $C_r$-action; it can also be indexed by
an \emph{$r$-core} in $\Lambda$, determined by the non-vanishing 
Pl\"ucker coordinates.

Let $\zeta = e^{\pi i/r}$, and
let $c = 
\big(\begin{smallmatrix} 
\zeta & 0 \\ 0 & \zeta^{-1}
\end{smallmatrix} \big)$.  Viewed as an element of $\PGL_2(\CC)$,
$c$ is a generator for the cyclic group $C_r$.  However, viewed as an 
element of $\SL_2(\CC)$, it acts on the vector space $\poln$
as a semisimple operator.  The monomials are a basis of eigenvectors:
the action of $c$ on the polynomial $z^k \in \poln$ is given by
\[ 
   c z^k = \zeta^{n-1-2k} z^k\,.
\]
For $k=0, \dots, r-1$,
let $M_k$ denote $\zeta^{n-1-2k}$-eigenspace for the action of $c$
on $\poln$.  Thus $M_k$ has a basis of monomials 
$\{z^k, z^{k+r}, \dots, z^{k + (m_k-1)r}\}$, where 
$m_k := \dim M_k = \lceil\frac{n-k}{r}\rceil$.

If $x \in X^r$, then the $c$ action on $\poln$ restricts to the
invariant subspace $x$.
Let $\rspec(x) = (\rspec_0(x), \dots, \rspec_{r-1}(x))$, where 
$\rspec_k(x)$ is
the multiplicity of the eigenvalue $\zeta^{n-1-2k}$ for the
action of $c$ on $x$.  Note that we have
$\rspec_k(x) \leq m_k$, and as the action is semisimple,
$\rspec_0(x) + \dots +\rspec_{r-1}(x) = d$.
For any non-negative integer vector $\bolds = (s_0, \dots s_{r-1})$, 
with $s_0+ \dots + s_{r-1} = d$ and $s_k \leq m_k$ for $k=0, \dots, r-1$,
let 
\[
  X^\bolds := \{x \in X^r \mid \rspec(x) = \bolds\}\,.
\]

\begin{lemma}
\label{lem:grproduct}
Let $\bolds$ be as above.
Then  $X^\bolds$ is naturally isomorphic to the product of
Grassmannians $\prod_{k=0}^{r-1} \Gr(s_k,M_k)$.
\end{lemma}

In particular $X^\bolds$ is irreducible, which shows that 
the eigenvalues of the action of $c$ distinguish the components
of $X^r$.

\begin{proof}
Each $x \in X^r$ decomposes as 
\[ 
  x  = \bigoplus_{k=0}^{r-1}\ (x \cap M_k)\,.
\]
Now $x \in X^\bolds$ iff $ \dim (x \cap M_k) = s_k$, $k=0, \dots, r-1$,
hence we must have $s_k \leq m_k$ for all $k$.
In this case the $r$-tuple $(x \cap M_0, \dots, x \cap M_{r-1})$
is a point in $\Gr(s_0, M_0) \times \dots \times \Gr(s_{r-1}, M_{r-1})$.
Conversely any point $(\Tx_0, \dots, \Tx_{r-1})$ 
in this product of 
Grassmannians corresponds to the point 
$\Tx_0 \oplus \Tx_1 \oplus 
\dots\oplus \Tx_{r-1} \in X^\bolds$.
\end{proof}

A partition is called an \defn{$r$-core} if its diagram does not have
any hook lengths equal to $r$.  Equivalently $\kappa$ is an $r$-core
if there does not exist an $r$-ribbon tableau of shape $\kappa/\mu$
for any $\mu \subsetneq \kappa$.  It is well known that if $\lambda$
is any partition, there is a unique $r$-core, 
$\rcore(\lambda) \subseteq \lambda$ (called the $r$-core of $\lambda$), 
such that there exists an $r$-ribbon tableau of shape 
$\lambda/\rcore(\lambda)$.  
It can be obtained by successively deleting $r$-ribbons from the shape
$\lambda$ until it is no longer possible to do so.  From this
description it is not entirely obvious that $\rcore(\lambda)$ is well 
defined;
this will be more evident from the alternate descriptions
of $r$-cores that are given below.

Our next goal is to show that the components of $X^r$ are in bijection 
with the $r$-cores in $\Lambda$.

\begin{example}
Suppose $n=11$, $d=5$, and $r=3$, and consider the component $X^{(0,3,2)}$
of $X^r$. 
The points of this component are subspaces of $\pol{10}$ spanned 
by $3$ polynomials
of the form $\alpha z + \beta z^4 + \gamma z^7 + \delta z^{10}$, 
and $2$ polynomials of the form $\alpha z^2 + \beta z^5 + \gamma z^8$.
The matrix $A_{ij} = [z^j]f_i(z)$, as in Section~\ref{sec:plucker},
is of the form
\[
\begin{pmatrix}
0 & * & 0 & 0 & * & 0 & 0 & * & 0 & 0 & * \\
0 & * & 0 & 0 & * & 0 & 0 & * & 0 & 0 & * \\
0 & * & 0 & 0 & * & 0 & 0 & * & 0 & 0 & * \\
0 & 0 & * & 0 & 0 & * & 0 & 0 & * & 0 & 0 \\
0 & 0 & * & 0 & 0 & * & 0 & 0 & * & 0 & 0 
\end{pmatrix}\,.
\]
From this matrix, one can see that at most $12$ of the $462$
Pl\"ucker coordinates of a point $x \in X^{(0,3,2)}$ 
can be non-zero.  The partitions indexing these Pl\"ucker
coordinates all have $3$-core equal to $32211$.
\end{example}

\begin{lemma}
\label{lem:corecomponents}
Let $X^\bolds$ be a component of $X^r$.  There is a unique $r$-core
in $\kappa \in \Lambda$ with the following property: for every 
$x \in X^\bolds$ and every
partition $\lambda \in \Lambda$ whose $r$-core is not equal to $\kappa$,
the Pl\"ucker coordinate $p_\lambda(x)$ is zero.
Conversely, for every $r$-core $\kappa \in \Lambda$, there is a unique
component $X^\bolds \subset X^r$ whose associated $r$-core is $\kappa$.
\end{lemma}

Thus the partitions indexing non-zero Pl\"ucker coordinates of a point 
$x \in X^r$ all have
the same $r$-core, and this $r$-core determines the component of $x$.
We will write $X^{r,\kappa} := X^\bolds$, if $\kappa$ is the $r$-core 
corresponding to the component $X^\bolds$.

To prove Lemma~\ref{lem:corecomponents}, we use the correspondence
between partitions and abacus diagrams.
An \defn{$r$-abacus} is a
an arrangement of $d$ beads inside an array with $r$ columns. 
The columns
of this array are called \defn{runners}, and are indexed $0, \dots, r-1$.  
The possible positions of the beads are indexed by 
non-negative integers,
where  $0,1, \dots, r-1$ are the positions in the first row (from left
to right), $r, r+1, \dots, 2r-1$ are the positions in the second row, and
so on.   Thus the possible positions for beads on the $k$\nth runner 
are congruent to $k$ modulo $r$.
The $r$-abacus associated to the partition $\lambda$ has beads
in positions $J(\lambda)$.  An example is given in Figure~\ref{fig:abacus}.

\begin{figure}[tb]
\centering
\begin{young}[bb]
 & & & & _{\tiny 9} & _{\tiny 10} & _{\tiny 11} & _{\tiny 12} 
                    & ,^_{\tiny 13}\hfill \\
 & & & & ,^_{\tiny 8} \hfill \\
 & & & & ,^_{\tiny 7}\hfill  \\
 & & _{\tiny 4}& _{\tiny 5} & ,^_{\tiny 6}\hfill   \\
 & _{\tiny 2} & ,^_{\tiny 3}\hfill  \\
_{\tiny 0} & ,^_{\tiny 1}\hfill
\end{young}
\qquad\qquad
\setlength\unitlength{1ex}
\newcommand{\bead}[1]{\begin{picture}(0,0)(0,0)
\put(0,0){\color{LightBlue} \circle*{2}}
\put(0,0){\circle{2}}
\put(-2,-.4){\makebox[4ex][c]{\tiny #1}}
\end{picture}}
\newcommand{\nobead}{\circle*{.5}}
\begin{picture}(16,18)(0,0)
{\linethickness{.6pt}
\put(1,16){\line(0,1){2}}
\put(15,16){\line(0,1){2}}
\put(1,16){\line(1,0){14}}
\put(1,18){\line(1,0){14}}}
\put(2,0){\line(0,1){16}}
\put(6,0){\line(0,1){16}}
\put(10,0){\line(0,1){16}}
\put(14,0){\line(0,1){16}}
\put(2,14){\nobead}
\put(6,14){\bead{1}}
\put(10,14){\nobead}
\put(14,14){\bead{3}}
\put(2,11){\nobead}
\put(6,11){\nobead}
\put(10,11){\bead{6}}
\put(14,11){\bead{7}}
\put(2,8){\bead{8}}
\put(6,8){\nobead}
\put(10,8){\nobead}
\put(14,8){\nobead}
\put(2,5){\nobead}
\put(6,5){\bead{13}}
\put(10,5){\nobead}
\put(14,5){\nobead}
\put(2,2){\nobead}
\put(6,2){\nobead}
\put(10,2){\nobead}
\put(14,2){\nobead}
\end{picture}
\caption{The partition $844421$, and the corresponding $4$-abacus. 
Here $d=6$.
The positions of the beads correspond to the up-steps along the boundary
of the partition.  The small dots in the abacus indicate positions 
without beads.}
\label{fig:abacus}
\end{figure}
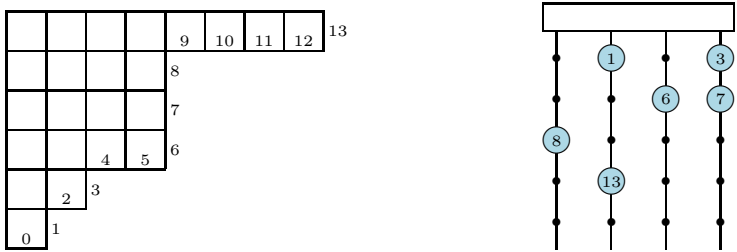

It is not hard to show that deleting an $r$-ribbon from a partition
corresponds to sliding a bead up one position in the corresponding 
$r$-abacus (see \cite{JK}).  From this it follows that 
the $r$-core of a partition is determined by the number 
of beads
on each runner of the $r$-abacus.  A partition is an $r$-core if the
corresponding 
$r$-abacus has its beads upwardly justified.

\begin{proof}
Let $x \in X^r$, and let $\lambda \in \Lambda$.
Let $s_k$ be the number of beads on the $k$\nth runner of the $r$-abacus 
$J(\lambda)$, $k=0, \dots, r-1$.  
We claim that if $x \notin X^\bolds$, then $p_\lambda(x) = 0$.

Suppose $x \in X^{\bolds'}$, and $\bolds'\neq \bolds$.  
Then for some $k$, $s_k < s'_k$.
Let $f_1(z), \dots, f_d(z)$ be is a basis of $c$-eigenvectors for $x$,
where $f_1(z), \dots, f_{s'_k}(z) \in M_k$ are $\zeta^{n-1-2k}$-eigenvectors.
Then $p_\lambda(x)$ is given as the determinant of the minor of
the matrix $A_{ij}$ as in Section~\ref{sec:plucker} with
columns $J(\lambda)$.  Consider the first $s'_k$ rows of this
minor.  There are at most $s_k$ non-zero entries in each row ---
they are in columns corresponding to the beads on the $k$\nth runner of
the abacus.  But since $s_k < s'_k$ these rows cannot be linearly
independent, and so the Pl\"ucker coordinate is $p_\lambda(x)$ zero.

It follows that if $x \in X^\bolds$, then the $p_\lambda(x) = 0$
for every $\lambda \in \Lambda$ that does not have $s_k$ beads
on the $k$\nth runner of the associated $r$-abacus.  Let $\kappa$ be the
$r$-core corresponding to $r$-abacus with $s_k$ upwardly justified
beads on runner $k$.  If $\rcore(\lambda) \neq \kappa$, then
the $r$-abacus does not the correct number of beads on each runner,
and so $p_\lambda(x) = 0$, as required.  

Finally, let $\kappa \in \Lambda$ be an $r$-core.
The argument above shows that 
there there is at most one component $X^{r,\kappa}$ of $X^r$ 
associated to the $r$-core $\kappa$.  To see that such
a component is always non-empty, we note that there 
exists a point in $x \in X^r$
whose Pl\"ucker coordinates are: $p_\kappa(x) = 1$, $p_\lambda(x) = 0$
for $\lambda \neq \kappa$.  This $x$ is necessarily in $X^{r,\kappa}$.
\end{proof}

\begin{theorem}
\label{thm:fixedpointplucker}
Let $x \in X$ and let
$P_x := \{\lambda \in \Lambda \mid p_\lambda(x) \neq 0\}$.
Then $x \in X^{r, \kappa}$ if and only if all partitions in $P_x$
have $r$-core equal to $\kappa$.
\end{theorem}

\begin{proof}
We have already shown one direction in Lemma~\ref{lem:corecomponents}.
It remains to show that if every partition in $P_x$ has $r$-core $\kappa$,
then $x \in X^{r,\kappa}$. 
The action of $c$ on $X$ is given in Plucker coordinates by
\[
  p_\lambda(cx) = \zeta^{-2|\lambda|}p_\lambda(x)
  \qquad \text{for }\lambda \in \Lambda\,.
\]
If all partitions in $P_x$ have the same $r$-core $\kappa$, then
for all $\lambda \in \Lambda$ 
we have either $\zeta^{2|\lambda|} = \zeta^{2|\kappa|}$ 
(if $\rcore(\lambda) = \kappa$), or $p_\lambda(x) = 0$ (otherwise).
Thus $[p_\lambda(cx)]_{\lambda \in \Lambda}$ is a scalar multiple
of $[p_\lambda(x)]_{\lambda \in \Lambda}$, i.e. $cx = x$.  Finally,
$x \in X^{r, \kappa}$, since $p_\lambda(x) \neq 0$ for some partition
with $\rcore(\lambda) = \kappa$.
\end{proof}

Without any further work, we can already deduce a few facts about
fixed points in the fibres of the Wronski map.

\begin{corollary}
Let $h(z)$ be a $C_r$-fixed polynomial, and let
$X_{\lambda/\mu}$ be a compatible Richardson variety.
Every point in $X^r_{\lambda/\mu}(h(z))$ lies in the same component
of $X^r$.
\end{corollary}

\begin{proof}
Let $x \in X^r_{\lambda/\mu}(h(z))$.  By 
Proposition~\ref{prop:richardsonplucker3},
$p_\mu(x) \neq 0$, so $x$ lies in the component of $X^r$
corresponding to the $r$-core of $\mu$.
\end{proof}

The next corollary is a special case of Theorem~\ref{thm:ribbon}.

\begin{corollary}
\label{cor:notsamecore}
Let $h(z)$ be a $C_r$-fixed polynomial, and let
$X_{\lambda/\mu}$ be a compatible Richardson variety.
If $\lambda$ and $\mu$ do not have the same $r$-core, then the
$X^r_{\lambda/\mu}(h(z))$ is empty.  In this case there are no
standard $r$-ribbon tableaux of shape $\lambda/\mu$.
\end{corollary}

\begin{proof}
If there exists a point $x \in X^r_{\lambda/\mu}(h(z))$ then 
$p_\mu(x) \neq 0$ and 
$p_\lambda(x) \neq 0$, so $\mu$ and $\lambda$ must have the same
$r$-core.  Since the $r$-core is obtained by deleting 
$r$-ribbons, if there exists an $r$-ribbon tableaux of shape
$\lambda/\mu$ then $\mu$ and $\lambda$ have the same $r$-core.
\end{proof}

\subsection{Coordinate rings of components and their initial ideals}
\label{sec:ideals}

Let $\Lambda^\bolds$ denote the set of $r$-tuples of partitions
$(\Tlambda_0, \dots, \Tlambda_{r-1})$, 
where $\Tlambda_k :
\Tlambda_k^0 \geq \dots \geq \Tlambda_k^{s_k}$
has at most $s_k$ parts and largest part at most $m_k -s_k$.  

The $r$-abacus $J(\lambda)$ can be encoded by a 
sequence of $r$ partitions in $\Lambda^\bolds$, where $s_k$
is the number of beads on the $k$\nth runner \cite{FS,JK,SW}.  
We do this by encoding the
sequence of beads on each runner into a partition in the standard way,
where the positions of beads correspond to up-steps along walking
along the boundary of the partition.
It is easier to state the reverse map, which sends
\[
  (\Tlambda_0, \dots, \Tlambda_{r-1}) \in \Lambda^\bolds
\]
to the abacus
\[
  J(\lambda) = 
  \big\{k+r\big(\Tlambda_k^{s_k+1-i}+i-1\big) \bigmid %
     i = 1, \dots, s_k,\ k=0, \dots, r-1\big\} \,.
\]
The $r$-tuple 
$(\Tlambda_0, \dots, \Tlambda_{r-1})$ is
called the \defn{$r$-quotient} of $\lambda$.

Let $\kappa$ be the unique $r$-core whose abacus has $s_k$ beads
on runner $k$.  Let $\Lambda^{r, \kappa} \subset \Lambda$ be the set of 
all partitions with $r$-core $\kappa$.  This is also the set of all
partitions whose $r$-abacus has $s_k$ beads on runner $k$.
hence the 
The $r$-quotient construction gives a bijection between $\Lambda^{r, \kappa}$
and $\Lambda^\bolds$.  Moreover, sliding a bead up the $k$\nth runner of
the abacus corresponds to deleting a box from the partition 
$\Tlambda_k$ and to deleting a $r$-ribbon from $\lambda$.  
It follows that 
$\lambda \mapsto (\Tlambda_0, \dots, \Tlambda_{r-1})$
is an isomorphism of partially ordered sets $\Lambda^{r, \kappa} \simeq
\Lambda^\bolds$.  Since the latter is a distributive lattice,
the former is too.  (Note, however, it is not a sublattice of $\Lambda$: 
the meet and join operations on $\Lambda^{r,\kappa}$ do not coincide 
with those of $\Lambda$).

The elements of $\Lambda^{r,\kappa}$ index the non-zero Pl\"ucker
coordinates on $X^{r, \kappa} = X^\bolds$, which by 
Lemma~\ref{lem:grproduct} is naturally identified with the product of 
$\prod_{k=0}^{r-1} \Gr(s_k, M_k)$.
The $r$-quotient construction
tells us how to identify the Pl\"ucker coordinates $p_\lambda$ with
coordinates of the Segre embedding
of this product of Grassmannians.

Let $\CC[\boldp]$ denote the polynomial ring generated by indeterminates
$\{p_\lambda\}_{\lambda \in \Lambda}$.  The homogeneous coordinate
ring of $X$ is $\CC[X] = \CC[\boldp]/I$, where $I$ is the Pl\"ucker ideal.
$I$ is generated by quadratics, and describes all relations among 
the Pl\"ucker coordinates.
Similarly, we write 
$\CC[\Gr(s_k,M_k)] = \CC[\boldp_{k}]/I_k$ for the homogeneous coordinate
ring of $\Gr(s_k,M_k)$.  The tensor product
$\CC[\Gr(s_0,M_0)] \otimes  \dots \otimes \CC[\Gr(s_{r-1},M_{r-1})]$ is
therefore the multihomogeneous ($\ZZ^r$-graded) coordinate ring 
for $\prod_{k=0}^{r-1} \Gr(s_k, M_k)$.  Abusing notation slightly,
we will write the multihomogeneous ideal for this product as 
\[
  I_0 + \dots + I_{r-1} := 
  \big(I_0 \otimes \CC[\boldp_1] \otimes \dots \otimes \CC[\boldp_{r-1}]\big)
   + \dots + 
  \big(\CC[\boldp_0] \otimes \dots \otimes \CC[\boldp_{r-2}] \otimes I_{r-1}\big)\,.
\]

There is a sign $\epsilon_\lambda \in \{\pm 1\}$ 
associated to each partition $\lambda \in \Lambda$, which is
easily obtained from the $r$-abacus.
First let $\pi_1(\lambda), \dots, \pi_d(\lambda)$, be the sequence 
obtained by listing
the positions of the beads on runner $0$ (in increasing
order), followed by those on runner $1$, and so on.  This list gives 
an ordering of the elements of $\widetilde J_\lambda$. 
We define $\epsilon_\lambda$ to be the sign of the permutation
that puts them in increasing order.  

\begin{lemma}
\label{lem:segre}
Consider the $\CC$-algebra homomorphism
$\Phi : \CC[\boldp] \to 
   \CC[\boldp_0] \otimes  \dots \otimes \CC[\boldp_{r-1}]$
defined by
\begin{equation}
\label{eqn:grsegre}
  \Phi(p_\lambda) =
  \begin{cases}
  \epsilon_\lambda\, p_{\Tlambda_0} \otimes 
  \dots \otimes p_{\Tlambda_{r-1}}
  & \quad\text{if }\lambda \in \Lambda^{r, \kappa}\text{ with $r$-quotient } 
  (\Tlambda_0, \dots, \Tlambda_{r-1}) \\
  0 & \quad\text{if $\lambda \notin \Lambda^{r, \kappa}$.}
  \end{cases}
\end{equation}
The homogeneous ideal of $X^{r,\kappa}$ is 
$I^{r,\kappa} := \Phi^{-1}(I_0 + \dots + I_{r-1}) \subset \CC[\boldp]$.
\end{lemma}

Thus $\Phi$ descends to a homomorphism
\[
\hat \Phi: \CC[X^{r, \kappa}] \to
   \CC[\Gr(s_0,M_0)] \otimes  \dots \otimes \CC[\Gr(s_{r-1},M_{r-1})]
\]
from the homogeneous coordinate ring of $X^{r,\kappa}$ to the 
multihomogeneous coordinate ring of $\prod_{k=0}^{r-1} \Gr(s_k, M_k)$.
Hence \eqref{eqn:grsegre}
identifies the coordinates of $X^{r, \kappa}$ with the 
coordinates of Segre embedding of
$\prod_{k=0}^{r-1} \Gr(s_k, M_k)$.

\begin{proof}
Recall from the proof of Lemma~\ref{lem:grproduct} that the isomorphism
$\prod_{k=0}^{r-1} \Gr(s_k, M_k)  \simeq X^{r,\kappa}$, is given
by $(\Tx_0, \dots, \Tx_{r-1}) \mapsto x = 
\Tx_0 \oplus \dots \oplus \Tx_{r-1}$,
where $\Tx_k = x \cap M_k$.
We must show that the Pl\"ucker coordinates of $x$ are (up to the
specified sign) products of Pl\"ucker coordinates of 
$\Tx_0, \dots, \Tx_{r-1}$.

Let $x \in X^{r, \kappa} = X^\bolds$. 
$f_{k1}(z), \dots, f_{ks_k}(z)$ be a basis for $\Tx_k = x \cap M_k$.
The Pl\"ucker coordinates of $x \cap M_k$ are given indexed by
partitions $\Tlambda_k$ with at most $s_k$ parts,
and largest part $\leq m_k-s_k$.  Let
$\widetilde{A}^k_{ij}$ be the $s_k \times n$ matrix with entries 
$\widetilde{A}^k_{ij} = [z^j]f_{ki}(z)$.  Then 
$p_{\Tlambda_k}(\Tx_k)$ is the minor of 
$\widetilde{A}^k_{ij} = [z^j]f_{ki}(z)$ with columns specified
by the positions of the beads on runner $k$ of the $r$-abacus 
$J(\lambda)$.

On the other hand $p_\lambda(x)$ is computed as the minor of the
matrix
\[
    A = \begin{pmatrix}
           \widetilde A^0 \\ \vdots \\ \widetilde A^{r-1}
        \end{pmatrix}
\]
with column set $J(\lambda)$.  Suppose $\lambda \in \Lambda^{r, \kappa}$.
If we take the columns
of this minor in the order $\pi_1(\lambda), \dots, \pi_d(\lambda)$,
we see that this is a block diagonal matrix with determinant
equal to 
\[
p_{\Tlambda_0}(\Tx_0) 
p_{\Tlambda_1}(\Tx_1)
\dotsb p_{\Tlambda_{r-1}}(\Tx_{r-1})\,.  
\]
But reordering the columns in this way changes the sign of the minor
by $\epsilon_\lambda$, and thus we have
\[
  \raisebox{1.8ex}{$\displaystyle
  p_\lambda(x) = \begin{cases}
  \epsilon_\lambda\,
  p_{\Tlambda_0}(\Tx_0) 
  p_{\Tlambda_1}(\Tx_1) 
  \dotsb p_{\Tlambda_{r-1}}(\Tx_{r-1})
  &\quad\text{if }\lambda \in \Lambda^{r,\kappa}  \\
  0 &\quad\text{otherwise.}
  \end{cases}$}
  \qedhere
\]
\end{proof}

Consider a Richardson variety $X_{\lambda/\mu}$, where $\lambda$
and $\mu$ both have $r$-core equal to $\kappa$.  Then the intersection
$X^{r,\kappa}_{\lambda/\mu} := X^{r,\kappa} \cap X_{\lambda/\mu}$
is non-empty.  
Write $I^{r,\kappa}_{\lambda/\mu} \subset \CC[\boldp]$
for the homogeneous ideal of $X^{r,\kappa}_{\lambda/\mu}$.
Also put $\Lambda^{r, \kappa}_{\lambda/\mu} 
:= \Lambda^{r,\kappa} \cap \Lambda_{\lambda/\mu}$.

\begin{lemma}
\label{lem:richardsonquotient}
Suppose $\lambda$ and $\mu$ have $r$-quotients 
$(\Tlambda_0, \dots \Tlambda_{r-1})$,
and 
$(\Tmu_0, \dots \Tmu_{r-1})$ respectively.
Under the isomorphism
$X^{r, \kappa} \cong \prod_{k=0}^{r-1} \Gr(s_k, M_k)$,
$X^{r, \kappa}_{\lambda/\mu}$ is identified with
$\prod_{k=0}^{r-1} \richvar{k}$,
where $\richvar{k}$ denotes the
Richardson variety in $\Gr(s_k, M_k)$ corresponding to the skew
shape $\Tlambda_k/\Tmu_k$.
\end{lemma}

\begin{proof}
For $k=0, \dots, r-1$, 
let $\richid{k} \subset \CC[\boldp_k]$ 
denote the homogeneous ideal of the Richardson variety 
$\richvar{k}$.  By Theorem~\ref{thm:richardsonplucker2}, this is generated
by $I_k$, together with 
\begin{equation}
\label{eqn:richardsonquotientgens}
\{p_\tau \mid \tau \nsupseteq \Tmu_k
\ \text{or}\ \tau \nsubseteq \Tlambda_k\}\,.
\end{equation}
The ideal $I^{r,\kappa}_{\lambda/\mu}$ is generated by
$I^{r,\kappa}$ and $\{p_\nu \mid \nu \notin \Lambda_{\lambda/\mu}\}$.

We must show that 
\begin{equation}
\label{eqn:richardsonfixedideal}
I^{r, \kappa}_{\lambda/\mu} = \Phi^{-1}(\richid{0} + \dots + \richid{r-1})\,.
\end{equation}
The right hand side of \eqref{eqn:richardsonfixedideal} is
generated by $\Phi^{-1}(I_0 + \dots + I_{r-1}) = I^{r,\kappa}$
together with all 
\[
  \{p_\nu \mid \Phi(p_\nu) \in \richid{0} + \dots + \richid{r-1} \}\,.
\]
By \eqref{eqn:grsegre}, $p_\nu$ is in this set if and only if
$\nu \notin \Lambda^{r,\kappa}$ or
$p_{\Tnu_k}$ is in the 
set \eqref{eqn:richardsonquotientgens} for some $k$.  The latter occurs
when $\nu \in \Lambda^{r, \kappa} \setminus \Lambda_{\lambda/\mu}$.
Thus the generators for the right hand side 
of \eqref{eqn:richardsonfixedideal} are the same as the generators for
the left hand side, as required.
\end{proof}

We will need some facts about the initial ideal of 
$I^{r,\kappa}$ --- and more generally the initial ideal of
$I^{r,\kappa}_{\lambda/\mu}$ --- for 
certain choices of weights on the Pl\"ucker coordinates.

Let $\boldw = (w_\lambda)_{\lambda \in \Lambda} \in \QQ^{\Lambda}$ 
be a vector of rational numbers.  The \defn{weight} of a monomial
$m(\boldp) = c \prod_{\lambda \in \Lambda} p_\lambda^{k_\lambda} 
\in \CC[\boldp]$ with respect to $\boldw$ is 
$$\weight_\boldw(m) 
:= \sum_{\lambda \in \Lambda} w_\lambda k_\lambda\,.$$
If $g(\boldp) \in \CC[\boldp]$, then the 
\defn{initial form} of $g$ with respect to $\boldw$,
denoted $\initial_\boldw(g)$,
is the sum of all monomial
terms in $g$ for which the weight of the term is minimized.
The \defn{initial ideal} of an ideal $H \subset \CC[\boldp]$ 
with respect to $\boldw$ is the ideal
$$\initial_\boldw(H) 
:= \{\initial_\boldw(g) \mid g \in H\}\,.$$
In this context, the vector $\boldw$ is called a \defn{weight vector}.

The weight vectors of interest to us will be constructed in the 
following way.  We begin with an $r$-ribbon tableau $T$ of shape
$\lambda/\mu$ and a sequence of rational numbers
$b_1, \dots, b_\ell$, where $r\ell = |\lambda/\mu|$.

Let $\smallrect_k := (m_k{-}s_k)^{s_k}$, so that
$(\smallrect_0, \dots, \smallrect_{r-1})$ is the maximal element
in the lattice $\Lambda^\bolds$.  For each $i=1, \dots, \ell$,
the entries labelled $i$ in $T$ form a ribbon.  This ribbon,
under the $r$-quotient correspondence, is identified with a single
box in some $\smallrect_k$.  Put the rational number $b_i$
in this box.  
We require our input data to be such that the following condition
holds.

\begin{condition}
\label{cond:square}
In the filling of $\smallrect_k$,
the entries are weakly decreasing along
rows and down columns, and no
$2 \times 2$ square has four equal entries.
\end{condition}

In every case we will consider, $b_1, \dots, b_\ell$ will be a weakly
decreasing sequence in which at most two of the $b_i$ are
equal; this guarantees that Condition~\ref{cond:square} will always 
satisfied.

If not all the boxes of all $\smallrect_k$ are filled,
we fill the remaining boxes with rational numbers that
are either less than $b_1$ or greater than $b_\ell$, so
that Condition~\ref{cond:square} is still satisfied.
For each partition
$\Tnu_k \subset \smallrect_k$,  let 
$\Tw_{\widetilde\nu_k}$ be the sum of all entries in the filling
of $\smallrect_k$ that are \emph{outside} the partition 
$\Tnu_k$.  
The weight vector we consider will be
\[
w_\nu =
\begin{cases}
\sum_{k=0}^{r-1}
(\Tw_{\widetilde \nu_k} - \Tw_{\widetilde \lambda_k})
&\quad\text{if $\nu \in \Lambda^{r,\kappa}$ with $r$-quotient 
$(\Tnu_0, \dots, \Tnu_{r-1})$}
\\
0 &\quad\text{if }\nu \notin \Lambda^{r,\kappa}\,.
\end{cases}
\]

\begin{example}
Suppose $r=2$, $d=3$, $n=8$.  Let us construct a weight vector $\boldw$
for the domino tableau
\[
   T\ =\ \begin{young}[c]
   , & ??2 & !3 & !!4 & !!4 \\
   ?1 & ??2 & !3 &  !!!6  \\
   ?1 & ???5 & ???5 & !!!6
   \end{young}\ .
\]
and the decreasing sequence $9,8,7,6,5,4$.  We start by placing these 
numbers in the boxes of $\smallrect_0$ and $\smallrect_1$ corresponding
to the ribbons in $T$:
\[
   \smallrect_0\ :\ \begin{young}[c] 
   9 & 8 &  
   \end{young}
   \qquad\quad
   \smallrect_1\ :\ \begin{young}[c] 
   7 & 6  \\
   5 & 4  
   \end{young}\ .
\]
For the empty box we have a choice --- we are allowed to fill it with any 
number less than $4$ --- here we will choose to fill it with a $1$.
To determine any
particular $w_{\nu}$, say for $\nu=322$, we shade the boxes 
corresponding to $(\Tnu_0, \Tnu_1) = (1, 11)$,
\[
   \smallrect_0\ :\ \begin{young}[c] 
   ?9 & 8 & 1 
   \end{young}
   \qquad\quad
   \smallrect_1\ :\ \begin{young}[c] 
   ?7 & 6  \\
   ?5 & 4  
   \end{young}\ ,
\]
and sum the numbers outside of the shaded boxes: $8+1+6+4=19$.  
We do the same procedure for the partition $\lambda = 544$: in
this case the sum is $1$.
Then $w_{322}$ is the difference of these two sums 
$w_{322} = 19-1 = 18$.
Note that $w_\lambda$ is always $0$, and the choice of numbers that
fill the empty boxes does
affect $w_\nu$ if $\nu \in \Lambda_{\lambda/\mu}$.
\end{example}

For such a weight vector, the initial ideals of $I^{r, \kappa}$ and
$I^{r,\kappa}_{\lambda/\mu}$ are
described by the lattice structure of $\Lambda^{r,\kappa}$.

\begin{lemma}
\label{lem:initialideal}
For every weight vector $\boldw$ obtained from the process above,
the initial ideal 
$\initial_\boldw(I^{r,\kappa})$  is generated by
\[
 \{(\epsilon_{\nu \vee \sigma} p_{\nu \vee \sigma})
(\epsilon_{\nu \wedge\sigma} p_{\nu \wedge \sigma})
    - (\epsilon_\nu  p_\nu) (\epsilon_\sigma p_\sigma)
  \mid \nu, \sigma \in \Lambda^{r,\kappa} \}
  \ \cup\ %
  \{ p_\nu  \mid \nu \notin \Lambda^{r,\kappa}\}\,,
\]
and the initial ideal $\initial_\boldw(I^{r,\kappa}_{\lambda/\mu})$ 
is generated by
\begin{equation}
\label{eqn:initgenerators}
 \{(\epsilon_{\nu \vee \sigma} p_{\nu \vee \sigma})
(\epsilon_{\nu \wedge\sigma} p_{\nu \wedge \sigma})
    - (\epsilon_\nu  p_\nu) (\epsilon_\sigma p_\sigma)
  \mid \nu, \sigma \in \Lambda^{r,\kappa} \}
  \ \cup\ %
  \{ p_\nu  \mid \nu \notin \Lambda^{r,\kappa}_{\lambda/\mu} \}\,.
\end{equation}
Here $\wedge$ and $\vee$ denote the meet and join operations on
the lattice $\Lambda^{r,\kappa}$.
\end{lemma}

\begin{proof}
It is enough to prove the second statement, since 
$I^{r, \kappa} = I^{r,\kappa}_{\lambda/\mu}$ in the case where
$\mu = \kappa$ and $\lambda$ is the largest partition 
in $\Lambda^{r,\kappa}$.

As in the proof of Lemma~\ref{lem:richardsonquotient}, 
let $\richid{k} \subset \CC[\boldp_k]$ denote the ideal of the 
Richardson variety $\richvar{k}$.  
Each of the vectors 
$\Tboldw_k := \big(\Tw_{\widetilde\lambda_k}\big)
_{\Tlambda_k \subset \smallrect_k}$, $k=0, \dots, r-1$
defines a weight vector for $\CC[\boldp_k]$.
It is well known (see \cite[Theorem 14.16]{MS}) that for these weight
vectors, if Condition~\ref{cond:square} holds, the initial ideal
$\initial_{\Tboldw_k}(I_k)$
is generated by the binomials
\[
 \{
 p_{\tau \vee \upsilon} p_{\tau \wedge \upsilon}
 - p_\tau p_\upsilon \mid \tau, \upsilon \subseteq \smallrect_k \}\,.
\]
More generally, one can show (e.g. using \cite[Theorem 7.1]{KLS})
that
$\initial_{\Tboldw_k}(\richid{k})$ is generated by 
\[
 \{
 p_{\tau \vee \upsilon} p_{\tau \wedge \upsilon}
 - p_\tau p_\upsilon \mid \tau, \upsilon \subset \smallrect_k \}
 \cup 
 \{p_\tau \mid \tau \nsupseteq \Tmu_k
\ \text{or}\ \tau \nsubseteq \Tlambda_k\}
\]
The initial ideal of 
$\richid{0} + \dots + \richid{r-1}
\subset
\CC[\boldp_0] \otimes  \dots \otimes \CC[\boldp_{r-1}]$
with respect to the weight vector
$\Tboldw := (\Tboldw_0, \dots, \Tboldw_{r-1})$ 
is then the sum of these:
$\initial_{\Tboldw}(\richid{0}+ \dots + \richid{r-1}) = 
\initial_{\Tboldw_0}(\richid{0}) + \dots 
+ \initial_{\Tboldw_{r-1}}(\richid{r-1})$.  
This is a basic fact about sums of ideals in disjoint sets of variables,
which follows readily from Buchburger's criterion.

Now, for any $g(\boldp) \in \CC[\boldp]$, 
$\Phi(\initial_\boldw(g))$ is either $0$ or equal to
$\initial_{\Tboldw}(\Phi(g))$.  Moreover, any 
$\Tboldw$-initial form 
in the image of $\Phi$ is the image of some $\boldw$-initial form 
in $\CC[\boldp]$.  It follows that
\begin{align*}
  \initial_\boldw(I^{r,\kappa}_{\lambda/\mu}) 
  &= 
  \initial_\boldw(\Phi^{-1}(\richid{0}+ \dots + \richid{r-1})) \\ 
  &= 
  \Phi^{-1}(\initial_{\Tboldw}(\richid{0} + \dots + \richid{r-1}))\\
  &= 
  \Phi^{-1}(\initial_{\Tboldw_0}(\richid{0}) + \dots + 
          \initial_{\Tboldw_{r-1}}(\richid{r-1}))
  \,.
\end{align*}

Since each of the initial ideals 
$\initial_{\Tboldw_k}(\richid{k})$ is generated
in degree $1$ and $2$,
$\initial_\boldw(I^{r, \kappa}_{\lambda/\mu})$ is generated 
by the $\Phi$-preimage
of the multihomogeneous elements of degree 
$(1,1,\dots 1)$ and $(2,2, \dots, 2)$ in
$\initial_{\Tboldw_0}(\richid{0}) + \dots + 
\initial_{\Tboldw_{r-1}}(\richid{r-1})$.
Thus
$\initial_\boldw(I^{r, \kappa}_{\lambda/\mu})$ is generated by 
$\{ p_\lambda  \mid \lambda \notin \Lambda^{r,\kappa}_{\lambda/\mu}\}$
together with binomials of the form
\begin{equation}
\label{eqn:initialbinomial}
  (\epsilon_{\nu'} p_{\nu'}) (\epsilon_\sigma p_{\sigma'})
  -
  (\epsilon_\nu p_\nu) (\epsilon_\sigma p_\sigma)
\end{equation}
where the $r$-quotients 
of $\nu$, $\sigma$, $\nu'$ and $\sigma'$ satisfy 
\begin{equation}
\label{eqn:quotientbinomial}
  p_{\widetilde {\nu'}_k} p_{\widetilde {\sigma'}_k}
  - p_{\Tnu_k} p_{\Tsigma_k} \in 
  \initial_{\Tboldw_k}(I_k)
\qquad\text{for }k=0, \dots, r-1\,.
\end{equation}
Among these are the binomials from \eqref{eqn:initgenerators}, which
arise when 
$\widetilde {\nu'}_k = \Tnu_k \vee \Tsigma_k$
and
$\widetilde {\sigma'}_k = \Tnu_k \wedge \Tsigma_k$
for all $k$.
On the other hand, each binomial \eqref{eqn:initialbinomial}, is in the 
ideal
generated by \eqref{eqn:initgenerators}.  Indeed, 
\eqref{eqn:quotientbinomial} implies that
$\nu' \vee \sigma' = \nu \vee \sigma$
and $\nu' \wedge \sigma' = \nu \wedge \sigma$, 
and so \eqref{eqn:initialbinomial} can be written as
\[
  \big(
  (\epsilon_{\nu \vee \sigma} p_{\nu \vee \sigma})
  (\epsilon_{\nu \wedge\sigma} p_{\nu \wedge \sigma})
      - (\epsilon_\nu  p_\nu) (\epsilon_\sigma p_\sigma)
  \big)
  -
  \big(
  (\epsilon_{\nu' \vee \sigma'} p_{\nu' \vee \sigma'})
  (\epsilon_{\nu' \wedge\sigma'} p_{\nu' \wedge \sigma'})
      - (\epsilon_{\nu'}  p_{\nu'}) (\epsilon_{\sigma'} p_{\sigma'})
  \big)
\]
Thus $\initial_\boldw(I^{r,\kappa}_{\lambda/\mu})$ is generated 
by \eqref{eqn:initgenerators}, as required.
\end{proof}

\section{Fixed points of the fibre and ribbon tableaux}
\label{sec:ribbon}


\subsection{Fibres of the Wronski map over a field of Puiseux series}
\label{sec:fibretableau}

Let $\FF$ be an algebraically closed field of characteristic zero
containing $\CC$ as a subfield.
Let $\psK := \Fpuiseux{u} = \bigcup_{n \geq 1} \FF(\!(u^{\frac{1}{n}})\!)$ 
be the field of Puiseux series over $\FF$.  This is also algebraically
closed field of characteristic zero, which contains $\CC$ as
a subfield.  As noted in the introduction, it is enough to prove
Theorem~\ref{thm:ribbon} working over any algebraically closed
field, and we will do so working over the field $\psK$.
For most of our arguments it will suffice to consider $\FF= \CC$,
however, for others we will take $\FF = \puiseux{v}$, the field
of puiseux series over $\CC$.

From this point forward, all algebraic objects ($X$, $X_{\lambda/\mu}$, 
$\Wr$, etc.) will be considered over the field $\psK$, rather than over
$\CC$.  For example, $X := \Gr(d, \Kpoln)$ will be the Grassmannian
of $d$-planes in the vector space $\Kpoln$ over $\psK$.

The advantage of working over $\psK$ is that it is a complete valuation 
ring.
If $g = c_\ell u^\ell + \sum_{k>\ell} c_k u^k \in \psK^\times$, the 
\defn{valuation} of $g$ is defined to be
$$\val(g) := \ell\,.$$
The \defn{leading term} $\leadterm(g)$ and 
\defn{leading coefficient} $\leadcoeff(g)$
are 
\[
  \leadterm(g) := c_\ell u^\ell \qquad \qquad
  \leadcoeff(g) := [u^\ell]g = c_\ell\,.
\]
Additionally, we set $\val(0) := +\infty$, $\val(\infty) := -\infty$
and $\leadterm(0) := 0$.
Let $\psK_+ = \{ g \in \psK \mid \val(g) \geq 0\}$.

Let $X_{\lambda/\mu}$ be a Richardson variety in $X$.
Let $a_1, a_2, \dots, a_{|\lambda/\mu|} \in \psK^\times$ be non-zero
puiseux series, and put
\begin{equation}
\label{eqn:hfactored}
  h(z) = z^{|\mu|}\prod_{i=1}^{|\lambda/\mu|}(z+a_i)\,.
\end{equation}
It will be convenient to adopt the convention that 
$a_0 = 0$, and $a_{|\lambda/\mu|+1}= \infty$.
We will assume in what follows that the Pl\"ucker coordinates
of a point $x \in X_{\lambda/\mu}(h(z))$ are normalized so
that $p_\lambda(x) = 1$.

In the case where $a_1, a_2, \dots, a_{|\lambda/\mu|}$ have distinct
valuations, the 
points in the fibre of $X_{\lambda/\mu}(h(z))$ can be identified with 
standard Young tableaux of shape $\lambda/\mu$.

\begin{lemma}
\label{lem:valtableau}
Suppose that 
\[
  \val(a_1) > \val(a_2) > \dots > \val(a_{|\lambda/\mu|})\,.
\]
For every point $x \in X_{\lambda/\mu}(h(z))$,
there is a unique chain of partitions
 $\mu = \lambda_0 \subsetneq \lambda_1 \subsetneq \dots \subsetneq
    \lambda_{|\lambda/\mu|} = \lambda$
with the property that 
\begin{equation}
\label{eqn:valplucker}
  \val(p_{\lambda_k}(x)) = \sum_{i = k+1}^{|\lambda/\mu|} \val (a_i)
   \qquad \text{for }k=0, \dots, |\lambda/\mu|\,.
\end{equation}
Moreover, for this chain we have
\begin{equation}
\label{eqn:leadtermdescent}
  \leadterm(p_{\lambda_k}(x)) =  \frac{q_{\lambda}}{q_{\lambda_k}}
   \prod_{i=k+1}^{|\lambda/\mu|} \leadterm(a_i)\
   \qquad \text{for }k=0, \dots, |\lambda/\mu|\,.
\end{equation}
\end{lemma}

The chain of partitions $\lambda_0 \subsetneq \lambda_1 \subsetneq
\dots \subsetneq \lambda$ can be identified with
a standard Young tableaux in $T_x \in \SYT(\lambda/\mu)$, obtained by 
placing entry $k$ in the box $\lambda_k/\lambda_{k-1}$.
$T_x$ is the tableau corresponding to the point 
$x \in X_{\lambda/\mu}(h(z))$.

For any shape $\nu \in \Lambda_{\lambda/\mu}$, let $T_x - \nu$
denote the set of entries in $T_x$ that are outside the shape $\nu$.
Thus $T_x - \lambda_k = \{k{+}1, \dots, |\lambda/\mu|\}$.
The system \eqref{eqn:valplucker} can be therefore rewritten as
$\val(p_{\lambda_k}(x)) = \sum_{i \in (T_x -\lambda_k)} \val(a_i)$.
In fact, when written this way, these equations hold for all Pl\"ucker 
coordinates of $X_{\lambda/\mu}$, not just those in the chain.

\begin{lemma}
\label{lem:valothers}
Suppose that 
\[
  \val(a_1) > \val(a_2) > \dots > \val(a_{|\lambda/\mu|})\,.
\]
If $x \in X_{\lambda/\mu}(h(z))$, let
$T = T_x$ be the corresponding tableau, and let
$\lambda_0 \subsetneq \lambda_1 \subsetneq \dots \subsetneq \lambda$
be the associated chain of partitions.
We have
\begin{equation}
\label{eqn:valothers}
  \val(p_{\nu}(x)) = 
\sum_{i \in (T -\nu)} \val(a_i)
\qquad \text{for }\nu  \in \Lambda_{\lambda/\mu}\,.
\end{equation}
In particular, 
$p_{\lambda_k}(x)$ has the unique minimal valuation among all
Pl\"ucker coordinates of $x$ indexed by partitions of size $|\lambda_k|$.  
That is,
if $\nu \in \Lambda$ is any partition of size $|\lambda_k|$ but 
not equal to $\lambda_k$, we have
  $\val(p_\nu(x)) > \val(p_{\lambda_k}(x))$.
\end{lemma}

\begin{proof}[Proof of Lemmas~\ref{lem:valtableau} and~\ref{lem:valothers}]
This proof is little more than a trivial modification of the proof of
\cite[Corollary 4.4]{Pur-Gr}.  We sketch the main idea here, omitting
several details that are the same.

Since the number of points in $X_{\lambda/\mu}(h(z))$ is exactly
$|\SYT(\lambda/\mu)|$, it is enough to show that for
each tableau $T \in \SYT(\lambda/\mu)$ encoding a chain of partitions
$\mu = \lambda_0 \subsetneq \lambda_1 \subsetneq \dots \subsetneq \lambda$,
there exists a point 
$x \in X_{\lambda/\mu}(h(z))$ such that \eqref{eqn:valplucker},
\eqref{eqn:leadtermdescent} and \eqref{eqn:valothers} hold. 

To do this, we use \cite[Theorem 4.2]{Pur-Gr}, which tells us how to obtain
such a point from a solution to a certain system of equations.  
In our case, the system of equations we need to solve is
\begin{equation}
\label{eqn:basicltsystem}
    q_{\lambda_i} \omega_{i+1} \dotsb \omega_{|\lambda/\mu|}
    = q_{\lambda} \leadcoeff(a_{i+1}) \dotsb \leadcoeff(a_{|\lambda/\mu|})
    \qquad i=1, \dots, |\lambda/\mu|\,,
\end{equation}
in complex variables $\omega_1, \dots, \omega_{|\lambda/\mu|}$.
This has the unique solution
\begin{equation}
\label{eqn:basicltsolution}
     \omega_i = \frac{q_{\lambda_i}\leadcoeff(a_i)}{q_{\lambda_{i-1}}} \,.
\end{equation}
According to \cite[Theorem 4.2]{Pur-Gr}, every solution
$(\omega_1, \dots, \omega_{|\lambda/\mu|})$ of \eqref{eqn:basicltsystem}
gives rise to a point $x \in X_{\lambda/\mu}(h(z))$ satisfying
\begin{equation}
\label{eqn:basicltplucker}
     \leadterm(p_\nu(x)) = 
     \left(\prod_{i \in (T-\nu)} \omega_i\right) 
     u^{\left(\sum_{i \in (T{-}\nu)} \val(a_i)\right)}
     \qquad\text{for } \nu \in \Lambda_{\lambda/\mu}\,.
\end{equation}
Taking valuations of both sides of \eqref{eqn:basicltplucker}, 
we see that $x$ satisfies \eqref{eqn:valothers}.
In the case where $\nu = \lambda_k$, we substitute the
solution \eqref{eqn:basicltsolution} into \eqref{eqn:basicltplucker}, 
giving
\begin{align*}
     \leadterm(p_\nu(x)) 
     &= 
     \left(\prod_{i \in (T- \lambda_k)} \omega_i\right) 
     u^{\left(\sum_{i \in (T{-}\lambda_k)}\val(a_i)\right)} \\
     &= 
     \left(\prod_{i=k+1}^{|\lambda/\mu|} 
     \frac{q_{\lambda_i}\leadcoeff(a_i)}{q_{\lambda_{i-1}}}\right)
     u^{\left(\sum_{i=k+1}^{|\lambda/\mu|} \val(a_i)\right)} \\
     &= 
     \frac{q_{\lambda}}{q_{\lambda_k}}
     \prod_{i=k+1}^{|\lambda/\mu|} \leadcoeff(a_i)u^{\val(a_i)}
\end{align*}
which is equivalent to \eqref{eqn:leadtermdescent}.  
This in turn implies \eqref{eqn:valplucker}.

Finally, we turn to the last statement of Lemma~\ref{lem:valothers}.
Let $\nu \in \Lambda$ be any partition of size $|\lambda_k|$. 
If $\nu  \in \Lambda_{\lambda/\mu}$, then
since $\val(a_1) > \dots > \val(a_{|\lambda/\mu|})$,
\[
  \val(p_\nu(x)) = \sum_{i \in (T-\nu)} \val(a_i) \geq
   \sum_{i = k+1}^{|\lambda/\mu|} \val(a_i) = \val(p_{\lambda_k}(x))\,.
\]
Equality occurs if and only if $T - \nu = \{k{+}1, \dots, N\}$, i.e.
$\nu = \lambda_k$. 
If $\nu \notin \Lambda_{\lambda/\mu}$
then by Theorem \ref{thm:richardsonplucker2}
$\val(p_\nu(x)) = +\infty > \val(p_{\lambda_k}(x))$.
\end{proof}

\subsection{Beyond the generic case}

For a polynomial $h(z)$ of the 
form \eqref{eqn:fixedpoly} with coefficients in $\psK$, the roots 
of $h(z)$ do not have distinct valuations.  
Thus Lemmas~\ref{lem:valtableau} and~\ref{lem:valothers} cannot be
used directly in the case of interest to us.  Nevertheless, some 
parts of these
lemmas still hold when the valuations of the roots are non-distinct.

\begin{lemma}
\label{lem:nongenericval}
Suppose that
\[
  \val(a_1) \geq \val(a_2) \geq \dots \geq \val(a_{|\lambda/\mu|})\,.
\]
There exists a chain of partitions 
  $\mu = \lambda_0 \subsetneq \lambda_1 \subsetneq \dots \subsetneq
    \lambda_{|\lambda/\mu|} = \lambda$
(not necessarily unique)
encoded by a tableau $T \in \SYT(\lambda/\mu)$ with the following
properties:  
\begin{packedenum}
\item[(1)] For all $\nu \in \Lambda_{\lambda/\mu}$
we have
\begin{equation}
\label{eqn:valotherineq}
  \val(p_{\nu}(x)) \geq \sum_{i \in (T-\nu)} \val(a_i)\,.
\end{equation}
Equality holds if the leading coefficients of 
$a_1, \dots, a_{|\lambda/\mu|}$ are generic.
\item[(2)] If $\val(a_k) > \val(a_{k+1})$, then 
\eqref{eqn:valplucker} and \eqref{eqn:leadtermdescent} hold.
\item[(3)] If $\val(a_k) > \val(a_{k+1})$, then
$p_{\lambda_k}(x)$ has the unique minimal valuation among
all Pl\"ucker coordinates of $x$ indexed by partitions of size $|\lambda_k|$.
\end{packedenum}
\end{lemma}

\begin{proof}
Fix $\bolda$, and consider the family of finite $\FF$-schemes over 
$(\FF^\times)^{|\lambda/\mu|}$, whose fibre at a point 
$\boldb = (b_1, \dots , b_{|\lambda/\mu|})$ is
$X_{\lambda/\mu}(h_\boldb(z))$, where
\[
  h_\boldb(z) :=  z^{|\mu|}\prod_{i=1}^{|\lambda/\mu}(z + a_ib_i)\,.
\]
For each component $C$ of the total
space, there is a unique vector $(m_\lambda)_{\lambda \in \Lambda}$,
such that $\val(p_\lambda(x)) \geq m_\lambda$ for all $x \in C$, 
with equality on Zariski-dense subset of $C$.

We claim that there is a Zariski-dense subset
$\calU \subset (\FF^\times)^{|\lambda/\mu|}$
such that if $\boldb \in \calU$, 
then for any $x^* \in X_{\lambda/\mu}(h_\boldb(z))$,
\eqref{eqn:valothers} holds for the point $x^*$ and some
tableau $T \in \SYT(\lambda/\mu)$.
This shows that the vector $(m_\lambda)_{\lambda \in \Lambda}$
above is given by the right hand side of \eqref{eqn:valothers},
which implies (1).  This claim will also be used in the proof of (2).

To prove the claim, we work over $\FF = \puiseux{v}$, 
though still assuming $a_1, \dots, a_{|\lambda/\mu|} \in \puiseux{u}$.
Consider the set 
\[
   \calU = \{(v^{\gamma_1}, \dots, v^{\gamma_{|\lambda/\mu|}})
   \mid \gamma_k \in \ZZ,\ \gamma_1 > \dots > \gamma_{|\lambda/\mu|}\}
\]
which is Zariski-dense in $\FF^{|\lambda/\mu|}$.
Fix $\boldb = (v^{\gamma_1}, \dots, v^{\gamma_{|\lambda/\mu|}}) \in \calU$,
and let $\varepsilon > 0$ be a small positive rational number.
Put
\[
  \hat h(z) :=  
  z^{|\mu|}\prod_{i=1}^{|\lambda/\mu}(z + \hat a_i)\,.
\]
where 
$\hat a_k = a_kv^{\gamma_k}u^{\varepsilon \gamma_k}$.
Since $\varepsilon$ is assumed to be small, we have 
$\val(\hat a_1) > \dots > \val(\hat a_{|\lambda/\mu|})$, so for
any point 
$\hat x \in X_{\lambda/\mu}(\hat h(z))$, \eqref{eqn:valothers}
holds: in particular, there is a tableau $T$ such that
\begin{align*}
  \val(p_{\nu}(\hat x)) &= 
\sum_{i \in (T -\nu)} \val(\hat a_i) \\
&= \sum_{i \in (T -\nu)} \val(a_i) + O(\varepsilon)
\end{align*}
for $\nu \in \Lambda_{\lambda/\mu}$.

Now consider the substitution $\Sigma : v \mapsto vu^{-\varepsilon}$.
Although this does not give a well defined transformation from $\psK$
to itself we do have 
$\Sigma(p_\nu(\hat x)) \in \psK$, because the valuations of
the Pl\"ucker coordinates of $\hat x$ are uniformly 
bounded below for all positive $\varepsilon$ near $0$.  Thus we
obtain a point 
\[
  x^* := \Sigma(\hat x) 
  \in X_{\lambda/\mu}(\Sigma(\hat h(z))) 
  = X_{\lambda/\mu}(h_\boldb(z))\,,
\]
and every point in $X(h_\boldb(z))$ is of this form.
The Pl\"ucker coordinates of $x^*$ satisfy
\begin{align*}
\val(p_\nu(x^*)) &= \val(p_\nu(\hat x)) + O(\varepsilon) \\
 &= \sum_{i \in (T -\nu)} \val(a_i) + O(\varepsilon)
\end{align*}
for $\nu \in \Lambda_{\lambda/\mu}$.
Since $\varepsilon$ can be made arbitrarily small, we see 
that \eqref{eqn:valothers} holds for the point $x^*$.

Next we prove (2).  Since the right hand size 
of \eqref{eqn:leadtermdescent}
is a continuous nowhere-vanishing function of 
$\leadterm(a_1), \dots, \leadterm(a_{|\lambda/\mu|})$,
it is enough to prove this under the assumption that
the leading coefficients 
$\leadcoeff(a_1), \dots \leadcoeff(a_{|\lambda/\mu|})$ are generic. 
Our claim shows that in this case \eqref{eqn:valothers} holds
for $x$ and some tableau $T \in \SYT(\lambda/\mu)$.  Suppose $T$ 
encodes the chain of partitions 
$\mu = \lambda_0 \subsetneq \lambda_1 \subsetneq \dots \subsetneq
    \lambda_{|\lambda/\mu|} = \lambda$.
As in the proof of Lemma~\ref{lem:valothers}, 
it follows that when $\val(a_k) > \val(a_{k+1})$, 
$p_{\lambda_k}(x)$ has the unique minimal valuation among
all Pl\"ucker coordinates of $x$ indexed by partitions of 
size $|\lambda_k|$.  

Now $h(z) = \Wr(x;z)$, which we can rewrite 
using \eqref{eqn:pluckerwronskian} as
\[
  q_\lambda z^{|\mu|} \prod_{k=1}^{|\lambda/\mu|} (z+a_k)
  = \sum_{\nu \in \Lambda}  q_\nu p_\nu(x) z^{|\nu|}\,.
\]
(Since the Pl\"ucker coordinates are assumed to be normalized so that 
$p_\lambda(x) = 1$, we need the factor $q_\lambda$ on the left hand side
to ensure that the coefficients of $z^{|\lambda|}$ agree on both 
sides.) Taking the coefficient
of $z^{|\mu|+k}$ on both sides, we see that
\[
  q_\lambda e_{|\lambda/\mu|-k}(a_1, \dots, a_{|\lambda/\mu|}) 
  = \sum_{\nu \vdash |(\mu|+k)}  q_\nu p_\nu(x)\,,
		\]
where $e_i$ denotes the $i$\nth elementary symmetric function.
Now take the leading term of both sides.  In the case where 
$\val(a_k) > \val(a_{k+1})$, 
the leading term of the right hand side 
is
\begin{equation}
\label{eqn:ltrightside}
   q_{\lambda_k} \leadterm(p_{\lambda_k})\,,
\end{equation}
as all other monomials in the sum have strictly larger valuation.
By similar reasoning, the leading term of the left hand side is
\begin{equation}
\label{eqn:ltleftside}
   q_\lambda \leadterm(a_{k+1}) \dotsb \leadterm(a_{|\lambda/\mu|})\,.
\end{equation}
Equating \eqref{eqn:ltrightside} and \eqref{eqn:ltleftside}, we 
obtain \eqref{eqn:leadtermdescent}.  Taking valuations 
of both sides gives \eqref{eqn:valplucker}.

Finally, we use the last argument in the proof of Lemma~\ref{lem:valothers}
once more to deduce (3) from (1) and (2).
\end{proof}

The chain of partitions in Lemma~\ref{lem:nongenericval} 
is not necessarily uniquely determined by the point $x$. 
However if $0 = k_0 < k_1 < \dots < k_{\ell-1} < k_\ell = |\lambda/\mu|$ 
are the indices for which $\val({a_{k_i}}) > \val(a_{k_i+1})$, 
then the subchain
\[
\mu = \lambda_0 \subsetneq \lambda_{k_1} \subsetneq \dots \subsetneq
\lambda_{k_\ell} = \lambda
\]
is characterized by the property (3) in 
Lemma~\ref{lem:nongenericval}.  Hence this particular subchain of partitions
is uniquely determined by $x$.  We can encode this subchain into a tableau,
again denoted $T_x$, by placing entry $i$ in each of the boxes of 
$\lambda_{k_i}/\lambda_{k_i-1}$.

In this way we can associate a tableau to every point $x \in X$.
This is not always an interesting tableau: for example if 
$h(z) \in \FF[z]$, then all roots have the same valuation and so $T_x$ 
is the tableau in which every entry is $1$.  In particular, the
reader should note that this
definition of $T_x$ is not does \emph{not} agree with the definition 
given in Section~\ref{sec:circle}.
We now show that in the case of interest to us, when $x$
is a generic $C^r$-fixed point,
this tableau is a standard $r$-ribbon tableau.

\begin{lemma}
\label{lem:alwaysribbon}
Suppose that $h(z)$ is $C_r$-fixed polynomial of the 
form \eqref{eqn:fixedpoly}, where
\[
  \val(h_1) > \val(h_2) > \dots > \val(h_\ell)\,.
\]
Let $X_{\lambda/\mu}$ be an
$h(z)$-compatible Richardson variety.
If $x \in X^r_{\lambda/\mu}(h(z))$, then the corresponding tableau $T_x$ 
is a standard $r$-ribbon tableau of shape $\lambda/\mu$.
\end{lemma}

\begin{proof}
Rewrite $h(z)$ in the form \eqref{eqn:hfactored}. 
Since each binomial $(z^r+h_i)$
has $r$ roots, each with valuation $\frac{1}{r}\val(h_i)$, we may assume 
that
\begin{align*}
  \val(a_1) = \dots = \val(a_r) &> \val(a_{r+1}) = \dots = \val(a_{2r}) \\
     &> \val(a_{2r+1}) = \dots = \val(a_{3r}) \\
     &\qquad\quad\vdots \\
     &> \val(a_{r\ell-r+1}) = \dots = \val(a_{r\ell})\,.
\end{align*}
Thus the tableau $T_x$ has entries $1,2, \dots, \ell$, each appearing
$r$ times.  

Let $\mu = \lambda_0 \subsetneq \lambda_1 \subsetneq \dots 
\subsetneq \lambda_\ell = \lambda$ be the chain of partitions associated 
to $T_x$.  By Lemma~\ref{lem:nongenericval}(2), each of Pl\"ucker
coordinates $p_{\lambda_k}(x)$ satisfies \eqref{eqn:leadtermdescent}
and hence is non-zero, for $k = 0, \dots, \ell$.
By Lemma~\ref{lem:corecomponents}, this implies that the each of
the partitions in the chain have the same $r$-core.  In particular,
$\lambda_{k+1}/\lambda_k$ is a partition of size $r$ where
$\lambda_{k+1}$ and $\lambda_k$ have the same $r$-core.  Any
such partition must be an $r$-ribbon, and the result follows.
\end{proof}

\subsection{Existence and uniqueness}

Suppose that $h(z)$ is $C_r$-fixed polynomial of the 
form \eqref{eqn:fixedpoly}, where $\val(h_1) > \dots > \val(h_\ell)$.
We have shown that every point $x \in X_{\lambda/\mu}(h(z))$ has
some associated tableau $T_x$, and that for a point in $X^r$, 
this tableau is a standard $r$-ribbon tableau.

Now, fix $T \in \ribbon(\lambda/\mu)$, and let 
$\mu = \lambda_0 \subsetneq \lambda_1 \subsetneq \dots \subsetneq
\lambda_\ell = \lambda$
be the corresponding chain of partitions.
To complete the proof of Theorem~\ref{thm:ribbon},
we need to prove that there exists a unique point 
$x \in X^r_{\lambda/\mu}(h(z))$ such that $T_x = T$, and moreover
that $x$ is a reduced point of $X^r_{\lambda/\mu}(h(z))$.

Suppose that such a point $x$ exists.  Then the following must be true:
\begin{packedenum}
\item[(1)] $x \in X^r_{\lambda/\mu}$; equivalently by
Theorem~\ref{thm:richardsonplucker2} 
and Theorem~\ref{thm:fixedpointplucker}, $p_\nu(x) = 0$ 
for $\nu \notin \Lambda^{r, \kappa}_{\lambda/\mu}$;
\item[(2)] by Lemma~\ref{lem:nongenericval}(2),
\begin{equation}
\label{eqn:ribbonltdescent}
  \leadterm(p_{\lambda_k}(x)) =  \frac{q_{\lambda}}{q_{\lambda_k}}
   \prod_{i=k+1}^{\ell} \leadterm(h_i)\
   \qquad \text{for }k=0, \dots, \ell\,.
\end{equation}
\end{packedenum}
We will first show that there exists a point in $X^r_{\lambda/\mu}$ 
satisfying these two conditions.  Then we will show that the 
additional condition $\Wr(x;z) = h(z)$ makes this point unique and
reduced in $X^r_{\lambda/\mu}(h(z))$.

Let $\boldw$ be a weight vector constructed 
(as in Section~\ref{sec:ideals}) from the tableau $T$ and the
decreasing sequence of rational numbers $\val(h_1), \dots, \val(h_\ell)$.
Note that $w_{\lambda_k} = \sum_{i=k+1}^\ell \val(h_i)$.
Let
\begin{equation}
\label{eqn:defY}
   Y := 
  \{x \in X^r_{\lambda/\mu} \mid 
  \val(p_\nu(x)) = w_\nu
  \text{ for all }\nu \in \Lambda^{r,\kappa}_{\lambda,\mu}\}\,.
\end{equation}
As we will see, the point $x \in  X^r_{\lambda/\mu}(h(z))$ corresponding
to $T$ will be in $Y$.
Consider a vector 
$\boldalpha = (\alpha_\nu)_{\nu \in \Lambda} \in \FF^\Lambda$ 
satisfying the following conditions:
\begin{alignat}{2}
   \label{eqn:alpha1}
   \alpha_\nu &= 0
    &&\qquad \text{if and only if }\nu \notin \Lambda^{r,\kappa}_{\lambda/\mu} \,, \\
   \label{eqn:alpha2}
   (\epsilon_\nu \alpha_\nu)(\epsilon_\sigma \alpha_\sigma)
   &= 
   (\epsilon_{\nu \vee \sigma}\alpha_{\nu \vee \sigma})
   (\epsilon_{\nu \wedge\sigma} \alpha_{\nu \wedge \sigma})
    &&\qquad \text{for }\nu, \sigma \in \Lambda^{r,\kappa}_{\lambda/\mu} \,.
\end{alignat}
Here $\epsilon_\nu = \pm 1$ is the sign defined in Section~\ref{sec:ideals}.
For any such $\boldalpha$, let
\begin{equation}
\label{eqn:defYalpha}
  Y_{\boldalpha} := 
  \{x \in Y \mid 
  \leadterm(p_\nu(x)) = \alpha_\nu u^{w_\nu} 
  \text{ for all }\nu \in \lambda\}\,.
\end{equation}

\begin{lemma}
\label{lem:alphachain}
There for any sequence $\beta_0, \dots, \beta_\ell \in \FF^\times$,
there is a unique vector $\boldalpha \in \FF^\Lambda$ 
satisfying \eqref{eqn:alpha1}, \eqref{eqn:alpha2} and
$\alpha_{\lambda_k} = \beta_k$ for $k= 0, \dots, \ell$.
\end{lemma}

\begin{proof}
This follows from the fact that $\Lambda^{r,\kappa}_{\lambda/\mu}$
is a distributive lattice.  Condition \eqref{eqn:alpha1} asserts
that $\nu \mapsto \epsilon_\nu \alpha_\nu$ is a modular function, and so this
function is uniquely by its values on the maximal chain
$\lambda_0 \subsetneq \dots \subsetneq \lambda_\ell$.
\end{proof}

\begin{lemma}
\label{lem:inYcondition}
If $x \in X^r_{\lambda/\mu}$ is a point such that 
\begin{equation}
\label{eqn:inYcondition}
  \leadterm(p_{\lambda_k}(x)) =  \alpha_{\lambda_k} u^{w_{\lambda_k}}
   \qquad \text{for }k=0, \dots, \ell\,,
\end{equation}
then $x \in Y_\boldalpha$.
\end{lemma}

\begin{proof}
Let $\alpha'_\nu = [u^{w_\nu}] p_\nu(x)$.  
Since $p_\nu(x) = 0$ for $\nu \notin \Lambda^{r,\kappa}_{\lambda/\mu}$, 
it is enough to show that
$\leadterm(p_\nu(x)) = 
   \alpha_\nu u^{w_\nu}$
for
$\nu \in \Lambda^{r,\kappa}_{\lambda/\mu}$. 
Equivalently, we must show that $\val(p_\lambda(x)) \geq w_\nu$ and
$\alpha'_\nu = \alpha_\nu \neq 0$
for $\nu \in \Lambda^{r,\kappa}_{\lambda/\mu}$.

Note that
$\alpha'_\nu = 0$ for $\nu \notin \Lambda^{r,\kappa}_{\lambda/\mu}$
and
$\alpha'_{\lambda_k} = \alpha_{\lambda_k}$ for $k= 0, \dots, \ell$.
By Lemma~\ref{lem:nongenericval}(1), \eqref{eqn:valotherineq} holds,
which can be re-expressed as
\[
  \val(p_\nu(x)) \geq w_\nu \qquad \text{for all }\nu \in \Lambda\,.
\]
Thus $\alpha'_\nu = p_\nu(x)u^{-w_\nu}|_{u = 0}$.
This implies that
$(\alpha'_\nu)_{\nu \in \Lambda}$ is in the variety of the initial ideal
$\initial_\boldw(I_{\lambda/\mu}^{r, \kappa})$.  
By Lemma~\ref{lem:initialideal} we see that
$\boldalpha' = (\alpha'_\nu)_{\nu \in \Lambda}$ is a vector 
satisfying \eqref{eqn:alpha1}, \eqref{eqn:alpha2} and
$\alpha'_{\lambda_k} = \alpha_{\lambda_k}$ for $k= 0, \dots, \ell$.
By Lemma~\ref{lem:alphachain}, we deduce that 
and we deduce that $\boldalpha' = \boldalpha$, as required.
\end{proof}

\begin{lemma}
\label{lem:Ylocalcoords}
There exist local coordinates $y_1, \dots, y_\ell$ on an open
subset of $X$ containing $Y$
with the following properties:
\begin{packedenum}
\item[(1)] 
     For each $\nu \in \Lambda^{r,\kappa}_{\lambda/\mu}$,
     there exists a polynomial
     $P_\nu(y_1, \dots, y_\ell) \in \psK_+[y_1, \dots, y_\ell]$
     such that $p_\nu(x) = u^{w_\nu}P_\nu(y_1(x), \dots, y_\ell(x))$
     for all $x \in Y$.
\item[(2)] 
     $\leadterm(p_{\lambda_k}(x))  
     = \epsilon_{\lambda_k}
       \leadterm(y_{k+1}(x) \dots y_\ell(x)) u^{w_{\lambda_k}}$
     for all $x \in Y$.
\item[(3)] The map $x \mapsto (y_1(x), \dots, y_\ell(x))$ gives
 a bijection
  \[
     Y_\boldalpha \leftrightarrow 
       \Big\{(y_1, \dots, y_\ell) \in (\psK_+)^\ell 
       \Bigmid \leadterm(y_k) 
     = \tfrac{\epsilon_{\lambda_k}\alpha_{\lambda_k}}
       {\epsilon_{\lambda_{k-1}}\alpha_{\lambda_{k-1}}}
     \text{ for }k=1, \dots, \ell\Big\}\,.
  \]
  In particular $Y_\boldalpha$ is non-empty, and 
  $\val(y_1(x)) = \dots = \val(y_\ell(x)) = 0$ 
      for all $x \in Y$.
\end{packedenum}
\end{lemma}

\begin{example}
\label{ex:Ylocalcoords}
We illustrate the construction of the local coordinates $y_1, \dots, y_\ell$,
in the case where $r=1$.  Suppose $d=3$, $n=7$, and 
\[
   T\ =\ %
   \begin{young}[c]
    , & 2 & 4 & 7 \\
    1 & 5 & 6 & 9 \\
    3 & 8 
   \end{young}\ .
\]
For any point $x \in Y$, we can choose a basis for $x$ so that
the matrix $A_{ij}$ of Section~\ref{sec:plucker} has the form
\[
  A = \begin{pmatrix}
   0 & 0 & 0 & \Hy_2\Hy_4\Hy_7 & \Hy_4\Hy_7 & \Hy_7 & 1 \\
   0 & \Hy_1\Hy_5\Hy_6\Hy_9 & \Hy_5\Hy_6\Hy_9 & \Hy_6\Hy_9 & \Hy_9 & 1 & 0 \\
  \Hy_3\Hy_8 & \Hy_8 & 1 & 0 & 0 & 0 & 0 \\
  \end{pmatrix}\,,
\]
in which the entries are $0$ or products of indeterminates:
\[
   A_{ij} = \prod_{m=j+i-d+1}^{\lambda^i+i-d+1} \Hy_{T_{im}}\,,
\]
where
$T_{im}$ denotes the entry of $T$ in row $i$, column $m$.
The coordinates $y_1, \dots, y_\ell$ are defined to be
\[
     y_k(x) := u^{-\val(h_k)}\Hy_k \qquad\text{for }k=1, \dots, \ell\,.
\]
The reader
can check that for any $h(z)$, with $\val(h_1) > \dots > \val(h_\ell)$,
properties (1)--(3) of Lemma~\ref{lem:Ylocalcoords} are
satisfied.
\end{example}

\begin{proof}
When $r=1$, the construction of $y_1, \dots, y_\ell$ is illustrated 
in Example~\ref{ex:Ylocalcoords}.  The 
proof that properties (1)--(3) hold for this construction is routine.
When $r>1$, we use the identification 
\[
   X^{r,\kappa}_{\lambda/\mu} \cong \prod_{k=0}^{r-1} \richvar{k}
\]
of Lemma~\ref{lem:richardsonquotient},
and define $y_1, \dots, y_\ell$ to be the local coordinates on
appropriate subsets of $\richvar{0}, \dots, \richvar{r-1}$.
Specifically, if $i_1, \dots, i_m$ are the entries of $T$ that 
correspond to
boxes of $\smallrect_k$ under the $r$-quotient construction,
then $y_{i_1}, \dots, y_{i_m}$ are the local coordinates
(defined as in the $r=1$ case) on $\richvar{k}$.
We leave it to the reader to verify that properties (1)--(3) can
be deduced from the $r=1$ case.
\end{proof}

We note that Lemmas~\ref{lem:alphachain},~\ref{lem:inYcondition}
and~\ref{lem:Ylocalcoords} (and their proofs) are still valid under 
the weaker condition
$\val(h_1) \geq \dots \geq \val(h_\ell)$, provided that  
Condition~\ref{cond:square} is satisfied in the construction of
the weight vector $\boldw$.  We will need this fact
in Section~\ref{sec:dihedral}.

\begin{lemma}
\label{lem:Ytableau}
If $x \in Y$, then $T_x = T$.
\end{lemma}

\begin{proof}
For $x \in Y$ we have 
\[
   \val(p_{\lambda_k}(x)) = w_{\lambda_k} 
     = \sum_{i=k+1}^\ell \val(h_i) < w_\nu = \val(p_\nu(x))
\]
for every partition $\nu \in \Lambda^{r,\kappa}_{\lambda/\mu}$
such that $|\nu| = |\lambda_k|$, $\nu \neq \lambda_k$.  The inequality
here
follows from the definition of $\boldw$.  This shows that $T_x = T$.
\end{proof}

\begin{lemma}
\label{lem:existunique}
With $h(z)$ as above, let $X_{\lambda/\mu}$ be a compatible Richardson
variety.
For every $T \in \ribbon(\lambda/\mu)$, there
exists a unique point $x \in X^r_{\lambda/\mu}(h(z))$ whose corresponding
tableau is $T_x = T$.  Moreover $X^r_{\lambda/\mu}(h(z))$ is reduced.
\end{lemma}

\begin{proof}
By Lemma~\ref{lem:alphachain} there is a unique $\boldalpha \in \FF^\Lambda$ 
satisfying \eqref{eqn:alpha1}, \eqref{eqn:alpha2} and
\begin{equation}
\label{eqn:alpha3}
   \alpha_{\lambda_k} = \frac{q_\lambda}{q_{\lambda_k}}\prod_{i=k+1}^\ell \leadcoeff(h_i) 
   \qquad \text{for }k=0,\dots, \ell\,.  
\end{equation}
For this $\boldalpha$, \eqref{eqn:ribbonltdescent} is equivalent
to \eqref{eqn:inYcondition}.  Thus if $x \in X^r_{\lambda/\mu}(h(z))$
is a point such that $T_x = T$, then by Lemma~\ref{lem:inYcondition}
we must have $x \in Y_\boldalpha$.  Conversely, by 
Lemma~\ref{lem:Ytableau} for every point $x \in Y_\boldalpha$ we have
$T_x = T$.  It suffices therefore to show that there exists a unique,
reduced point in $Y_\boldalpha \cap X^r_{\lambda/\mu}(h(z))$.

Let 
\[
   H := \big\{g(z) \in \Kpol{N} \bigmid %
         \leadterm\big([z^j]g(z)\big) = \leadterm\big([z^j]h(z)\big) 
             \text{ for }j=0, \dots, N\big\}\,.
\]
We now show that $\Wr(Y_\boldalpha) \subset H$.
For any $x \in Y_\boldalpha$ we have 
\begin{align*}
\leadterm\big([z^j]\Wr(x;z)\big) &= 
\tfrac{1}{q_\lambda}\,
\leadterm\Big([z^j] \sum_{\nu \vdash j} q_\nu p_\nu(x)\Big) \\
&= 
{\begin{cases}
\frac{q_{\lambda_k}}{q_\lambda}\, \leadterm (p_{\lambda_k}(x)) & 
\quad\text {if $j = |\lambda_k|$ for some $k$} \\
0 & \quad \text{otherwise}\,,
\end{cases}}
\end{align*}
since, in the first case, $T_x = T$, so $p_{\lambda_k}(x)$ has unique 
minimal valuation among all Pl\"ucker coordinates indexed by partitions 
of size $j = |\lambda_k|$.  On the other hand, if $j = |\lambda_k|$,
\begin{align} 
\nonumber
\leadterm\big([z^j]h(z) \big)
&=
\prod_{i=k+1}^\ell \leadterm(h_i) 
\\
\label{eqn:hcoeffalpha}
&=
\tfrac{q_{\lambda_k}}{q_\lambda}\, \alpha_{\lambda_k} u^{w_{\lambda_k}}  \\
\nonumber
&=
\tfrac{qi{\lambda_k}}{q_\lambda}\, \leadterm(p_{\lambda_k}(x))\,,
\end{align}
and $[z^j]h(z) = 0$ otherwise.
Thus $\Wr(x;z) \in H$, for all $x \in Y_\boldalpha$.  To finish, we show
that the map $\Wr : Y_\boldalpha \to H$ is invertible.  For this we
write $\Wr(x;z) = h(z)$ as a system of equations in the local coordinates
$y_1, \dots, y_\ell$ of Lemma~\ref{lem:Ylocalcoords}, and use Hensel's
Lemma (see e.g. \cite[Exer. 7.25]{Eis}),
which is the analogue of the inverse function theorem for valuation rings.

Comparing non-zero coefficients of both sides, $\Wr(x;z) = h(z)$
can be written as
\begin{equation}
\label{eqn:inverseequations}
     \sum_{\nu \vdash (|\mu|+rk)} q_\nu p_\nu(x)
     = [z^{|\mu|+rk}] h(z)
     \qquad\text{for }k=0, \dots, \ell\,.
\end{equation}
By Lemma~\ref{lem:Ylocalcoords}(2), the leading term of the left hand 
side of \eqref{eqn:inverseequations} is
\begin{equation}
\label{eqn:ltinverselhs}
     q_{\lambda_k} \leadterm(p_{\lambda_k}(x)) 
     = q_{\lambda_k} 
     \epsilon_{\lambda_k}\leadterm(y_{k+1}(x) \dotsb y_\ell(x)) 
     u^{w_{\lambda_k}}
    \,.
\end{equation}
From \eqref{eqn:hcoeffalpha}, the leading term of the right hand side 
of \eqref{eqn:inverseequations} is
\begin{equation}
\label{eqn:ltinverserhs}
   q_{\lambda_k} \alpha_{\lambda_k} u^{w_{\lambda_k}}\,.
\end{equation}
Dividing both sides by $\epsilon_{\lambda_k}q_{\lambda_k}u^{w_{\lambda_k}}$,  the system of 
equations \eqref{eqn:inverseequations} 
can be written as
\begin{equation}
\label{eqn:firstorderinverse}
   y_{k+1} \dotsb y_\ell = \epsilon_{\lambda_k} \alpha_{\lambda_k} 
    + u^{\delta_k}G_k(y_1, \dots, y_\ell)
   \qquad\text{for }k=0, \dots \ell\,,
\end{equation}
where $\delta_k > 0$, and by Lemma~\ref{lem:Ylocalcoords}(1)
$G_k(y_0, \dots, y_\ell) \in \psK_+[y_1, \dots, y_\ell]$.
To apply Hensel's Lemma to the system \eqref{eqn:firstorderinverse}, 
we need to check two things.
\begin{packedenum}
\item[(1)]
Every $x \in Y_\boldalpha$ is a first order solution
to \eqref{eqn:firstorderinverse};
that is, $\leadterm(y_{k+1}(x) \dotsb y_\ell(x)) 
= \epsilon_{\lambda_k}\alpha_{\lambda_k}$
for $x \in Y_\boldalpha$.
This follows from Lemma~\ref{lem:Ylocalcoords}(3).

\item[(2)]
The Jacobian matrix 
$\big(\tfrac{\partial}{\partial y_j} y_{i+1} \dotsb y_\ell \big)_{i,j=1, \dots, \ell}$ is invertible over $\psK_+$,
for $x \in Y_\boldalpha$.  This is true because the Jacobian matrix
is upper-triangular, and the diagonal entries are invertible 
by Lemma~\ref{lem:Ylocalcoords}(3).
\end{packedenum}
We deduce
that the system \eqref{eqn:firstorderinverse} has a unique, reduced
solution $(y_1, \dots, y_\ell) \in (\psK_+)^{\ell+1}$. 
 This solution gives the local coordinates 
of the point $x \in Y_\boldalpha \cap X^r_{\lambda/\mu}(h(z))$,
as required.
\end{proof}

Lemma~\ref{lem:existunique} is the final ingredient in the proof of
Theorem~\ref{thm:ribbon}.

\begin{proof}[Proof of Theorem~\ref{thm:ribbon}]
Since $h(z)$ is assumed to be a generic polynomial of the 
form \eqref{eqn:fixedpoly}, we may assume that
$\val(h_1) > \dots > \val(h_\ell)$.  
Lemmas~\ref{lem:alwaysribbon} and~\ref{lem:existunique} show
that in this case, $x \mapsto T_x$
gives a bijection between the reduced finite scheme 
$X^r_{\lambda/\mu}(h(z))$ and $\ribbon(\lambda/\mu)$.
\end{proof}

\begin{remark} 
Although $X^r_{\lambda/\mu}(h(z))$ is reduced, it is not always
true that $X_{\lambda/\mu}(h(z))$ is reduced when $h(z)$ is a generic
polynomial of the form~\eqref{eqn:fixedpoly}.  One can see this 
from the combinatorics.  For example, suppose 
$r=3$, $h(z)=z^3+1$ and $\lambda/\mu = 21$.  Then
$|\SYT(\lambda/\mu)| = 2$, and $|{\ribbonplain^3}(\lambda/\mu)| = 1$,
which means that $X_{\lambda/\mu}(h(z))$ is a finite scheme of length $2$,
and $C_3$ acts on it with exactly $1$ fixed point.  Since $C_3$
cannot act on a set of size $2$ with only one fixed point, this
shows that $X_{\lambda/\mu}(h(z))$ is supported on a single point,
hence non-reduced.  More generally, if $r$ is prime and
$|\SYT(\lambda/\mu)| \not \equiv |\ribbon(\lambda/\mu)| \pmod r$, then for
every $C_r$-fixed $h(z)$ compatible with $X_{\lambda/\mu}$,
we have that
$X_{\lambda/\mu}(h(z))$ is non-reduced.
\end{remark}


\section{Dihedral group actions}
\label{sec:dihedral}

\subsection{$D_r$-fixed points}

Our final goal in this paper is to prove Theorem~\ref{thm:D-ribbon}.
We begin with some simple observations about $D_r$-fixed points
in $X$ and $D_r$-fixed polynomials in $\KPpolN$.

For any point $x \in X$, let $x^\vee := \stdreflect x$.

\begin{proposition}
The Pl\"ucker coordinates of $x^\vee$ are
\begin{equation}
\label{eqn:reflectX}
  p_\lambda(x^\vee) = p_{\lambda^\vee}(x)\,,
  \qquad \lambda \in \Lambda\,.
\end{equation}
\end{proposition}

We note once again, that the Pl\"ucker coordinates are homogeneous
coordinates, and hence \eqref{eqn:reflectX} must be viewed as a system
of equations that holds up to a scalar multiple.

\begin{proof}
The action of $\stdreflect$ on a polynomial $\poln$ is given by
\[
   \stdreflect(\alpha_{n-1}z^{n-1} + \alpha_{n-2} z^{n-2} + \dots + \alpha_0)
    =  \alpha_0 z^{n-1} + \dots + \alpha_{n-2}z + \alpha_{n-1}\,.
\]
In terms of the matrix
$A_{ij}$ used to defined the Pl\"ucker coordinates (see
Section~\ref{sec:plucker}), the action of $\stdreflect$ simply
reverses the order of the columns; hence the minor $A_{J(\lambda^\vee)}$ 
for $x$ becomes $(-1)^{d(d-1)/2}A_{J(\lambda)}$ 
for $x^\vee$.  These are the Pl\"ucker coordinates $p_{\lambda^\vee}(x)$
and $p_\lambda(x^\vee)$ respectively.  The sign is a constant scalar
multiple, so it can be discarded.
\end{proof}

\begin{proposition}
\label{prop:D-compatible}
If $x \in X$ is a $D^r$ fixed point, and $X_{\lambda/\mu}$
is the minimal Richardson variety containing $x$.  
Then $\mu = \lambda^\vee$.
\end{proposition}

\begin{proof}
Since $x = x^\vee = \stdreflect x$, $p_{\nu^\vee}(x) = p_\nu(x)$
for all $\nu \in \Lambda$.
In particular, if 
$\lambda$ is the maximal partition such that $p_\lambda(x) \neq 0$,
then $\lambda^\vee$ is the minimal partition such that
$p_{\lambda^\vee}(x) \neq 0$.  By Proposition~\ref{prop:richardsonplucker1},
we must have $\mu = \lambda^\vee$.
\end{proof}

This establishes one easy special case of Theorem~\ref{thm:D-ribbon}: if
$\mu \neq \lambda^\vee$, then there are no $D_r$-fixed points
in $X_{\lambda/\mu}(h(z))$, and the combinatorial sets
$\rotribbon(\lambda/\mu)$ and $\longribbon(\lambda/\mu)$ are
both empty.

Let $h(z) \in \KPpolN$ be a $D_r$-fixed polynomial, and let
$X_{\lambda/\mu}$ be a compatible Richardson variety.  
Henceforth,
we will assume that $\mu = \lambda^\vee$.
Working over the field $\psK$, we will write $h(z)$ in the 
form \eqref{eqn:fixedpoly}, where
where $m = \mu$, $r\ell = |\lambda/\mu|$, and 
\[
  \val(h_1) \geq \val(h_2) \geq \dots \geq \val(h_\ell)\,.
\]
The space of all $h(z)$ that are compatible with $X_{\lambda/\mu}$ 
always has two components.  If $\ell$
is odd, the analysis will be the same for both components, and
we treat them as a single case.  If $\ell$ is even, the analysis
is different for the two components.
Hence will need to consider three different cases, corresponding
to the three different generic types of $D_r$-fixed polynomials
described in the introduction.

In type (1), $\ell$ is odd, and we can assume 
the $h_1,\dots,h_\ell$ have distinct valuations.  
Since $\{h_1, \dots, h_\ell\}$ is invariant
under $z \mapsto \frac{1}{z}$, this implies that 
$h_{(\ell+1)/2} = \pm 1$, and $ h_{\ell+1-k} =(h_k)^{-1}$
for $k=1, \dots, \frac{\ell-1}{2}$.
Similarly in type (2), where $\ell$ is even, 
we may assume that $h_1,\dots,h_\ell$ have distinct valuations;
hence $h_{\ell+1-k} = (h_k)^{-1} $ for $k=1, \dots, \frac{\ell}{2}$,
and none are equal to $\pm 1$. 
In type (3), $\ell$ is even, and we are required to have 
$h_{\ell/2} = 1$ and $h_{(\ell/2)+1} = -1$,
which both have valuation $0$.
The rest of the $h_i$ can be assumed to have 
distinct valuations, and hence
$h_{\ell+1-k} =(h_k)^{-1} $, 
for $k=1, \dots, \frac{\ell}{2}-1$.  A summary is given in 
Table~\ref{tab:generictypes}.

\begin{table}[tb]
\centering
\begin{tabular}{|c|c|c|}
\hline 
\rule{0pt}{2.5ex}
Type (1) & Type (2) & Type (3) \\[.5ex] \hline
\rule{0pt}{3.5ex}
$\ell$ is odd & $\ell$ is even & $\ell$ is even \\[1.5ex]
$\val(h_1) > \dots > \val(h_\ell)$ &
$\val(h_1) > \dots > \val(h_\ell)$ &
$\val(h_1) > \dots > \val(h_{\ell/2})$ \\[1.5ex]
$h_{(\ell+1)/2} = \pm 1$ & $h_{\ell/2},\,h_{(\ell/2)+1} \neq \pm 1$ &
$h_{\ell/2} = 1$ and
$h_{(\ell/2)+1} = -1$ \\[1.75ex]
\parbox{1.5in}{\centering
  $h_{\ell+1-k} = (h_k)^{-1}$ \\for $k=1, \dots, \frac{\ell-1}{2}$}
  &
\parbox{1.5in}{\centering
  $h_{\ell+1-k}= (h_k)^{-1}$ \\for $k=1, \dots, \frac{\ell}{2}$} &
\parbox{1.5in}{\centering
  $h_{\ell+1-k} =(h_k)^{-1} $ \\for $k=1, \dots, \frac{\ell}{2}-1$}
\\[2.25ex] \hline
\end{tabular}
\caption{The three generic types of $D_r$-fixed polynomials over $\psK$.}
\label{tab:generictypes}
\end{table}

\begin{proposition}
The $D_r$-fixed polynomials of types (1), (2) or (3) form
a Zariski-dense subset of all $D_r$-fixed polynomials in $\KPpolN$
that are compatible with $X_{\lambda/\mu}$.
\end{proposition}

\begin{proof}
Let $h(z) \in \KPpolN$ be polynomial compatible with $X_{\lambda/\mu}$.
For $a \in \psK$, let $M(a)$ denote the largest integer for 
which $(z^r + a)^{M(a)}$ divides $h(z)$.  If $a \neq \pm 1$ then
$M(a) = M(a^{-1})$.  It follows that $\ell - M(1) - M(-1)$ is even.
There are four cases.

\begin{packedenum}
\item[(1)]
If $\ell$ is odd, $M(1)$ is odd, $M(-1)$ is even, then $h(z)$ is in
the closure of the polynomials of type (1) for which $h_{(\ell+1)/2} = 1$.

\item[(2)]
If $\ell$ is odd, $M(1)$ is even, $M(-1)$ is odd, then $h(z)$ is in
the closure of the polynomials of type (1) for which $h_{(\ell+1)/2} = -1$.

\item[(3)]
If $\ell$ is even, $M(1)$ is even, $M(-1)$ is even, then $h(z)$ is in
the closure of the polynomials of type (2).

\item[(4)]
If $\ell$ is even, $M(1)$ is odd, $M(-1)$ is odd, then $h(z)$ is in
the closure of the polynomials of type (3).  \qedhere
\end{packedenum}
\end{proof}

\subsection{Reflection on the fibre models rotation of tableaux}

We continue to assume that $h(z)$ is a $D_r$-fixed polynomial,
and $X_{\lambda/\mu}$ is a 
compatible Richardson variety with $\mu = \lambda^\vee$.
For any tableau $T$ of shape $\lambda/\mu$, let $T^\vee$ denote the 
tableau obtained by rotating $T$ by $180^\circ$, 
and replacing each entry $k$ by $m{+}1{-}k$, where $m$ is 
the largest entry.

\begin{lemma}
\label{lem:reflectrotate}
Let $x \in X(h(z))$ be any point, and $T_x$ be the associated tableau,
as in Section~\ref{sec:fibretableau}.
Then the tableau associated to $x^\vee$ is $(T_x)^\vee$.
\end{lemma}

Thus the action of $\stdreflect$ on the fibre $X(h(z))$
models the action of $180^\circ$ rotation on tableaux.  We
note that this result does not require $h(z)$ to be generic, of
one of the three types, though for our applications it will be.

\begin{proof}
Let 
$\lambda_0 \subsetneq \lambda_1 \subsetneq \dots \subsetneq \lambda_\ell$
be the chain of partitions associated to $T_x$.  Then
$p_{\lambda_k}(x)$ has unique minimal valuation among all Pl\"ucker
coordinates of $x$ indexed by partitions of size $|\lambda_k|$.
By \eqref{eqn:reflectX}, $p_{\lambda_k^\vee}(x^\vee)$ has unique minimal 
valuation among all Pl\"ucker coordinates of $x^\vee$ indexed by
partitions of size $|\lambda_k^\vee|$.  Thus
$\lambda_\ell^\vee \subsetneq \lambda_{\ell-1}^\vee 
\subsetneq \dots \subsetneq \lambda_0^\vee$ encodes $T_{x^\vee}$,
i.e. $(T_x)^\vee = T_{x^\vee}$.
\end{proof}

We immediately obtain the first case of Theorem~\ref{thm:D-ribbon}.  

\begin{proof}[Proof of Theorem~\ref{thm:D-ribbon}(i)]
If $h(z)$ is of type (1) or (2), 
then $h_1, \dots, h_\ell$ have distinct valuations, and hence
by Lemma~\ref{lem:existunique}
we have a bijective correspondence between the reduced finite
scheme $X^r_{\lambda/\mu}(h(z))$
and $\ribbon(\lambda/\mu)$.  For any $x \in X^r_{\lambda/\mu}(h(z))$,
 we have $x = x^\vee$ iff
$T_x = T_{x^\vee} = (T_x)^\vee$, where the last equality is by 
Lemma~\ref{lem:reflectrotate}.
Hence $\stdreflect$-fixed points in $X^r_{\lambda/\mu}(h(z))$ are
in bijection with the set $\rotribbon(\lambda/\mu)$ of
rotationally-invariant $r$-ribbon tableaux.
\end{proof}

If $h(z)$ is of type (3), then a bit more work is required,
since $h_{\ell/2}$ and $h_{(\ell/2)+1}$ have the same valuation.
In this case, the corresponding tableau is not a standard $r$-ribbon 
tableau,
but something close.
Let $T$ be a tableau of shape $\lambda/\mu$, where $|\lambda/\mu| = r \ell$
and $\ell$ is even.  We say $T$ is 
an \defn{almost-standard} $r$-ribbon tableau if the entries of
$T$ not equal to $\frac{\ell}{2}$ form $r$-ribbons, and
the entries equal to $\frac{\ell}{2}$ form a skew shape with
$2r$-boxes that supports an $r$-ribbon.
We will write the chain of partitions corresponding to such a 
tableau $T$ as
\begin{equation}
\label{eqn:chainmissingone}
   \mu = \lambda_0 \subsetneq \dots  \subsetneq
    \lambda_{(\ell/2)-1} \subsetneq
    \lambda_{(\ell/2)+1} \subsetneq \dots \subsetneq \lambda_\ell = \lambda\,.
\end{equation}

\begin{lemma}
\label{lem:almostribbon}
Let $x \in X^r_{\lambda/\mu}(h(z))$ where $h(z)$ is of type (3).
Then the corresponding $T_x$ is an almost-standard $r$-ribbon tableau.
\end{lemma}

\begin{proof}
The argument is essentially the same as the proof of 
Lemma~\ref{lem:alwaysribbon}.
\end{proof}

One possibility, if $T$ is an almost-standard $r$-ribbon tableau of
shape $\lambda/\mu$, is that the entries equal to $\frac{\ell}{2}$ actually
form a $2r$-ribbon.  In this case we say that $T$ has a 
\defn{long ribbon}.

\begin{lemma}
\label{lem:oneortwopoints}
Let $T$ be  an almost-standard $r$-ribbon tableau, and
suppose $h(z)$ is of type (3).  
\begin{packedenum}
\item[(i)]
If $T$ has a long ribbon, then there is exactly one point
$x \in X^r_{\lambda/\mu}(h(z))$ such that $T_x = T$; moreover, $x$
is a reduced point of $X^r_{\lambda/\mu}(h(z))$.
\item[(ii)]
If $T$ does not have a long ribbon, then
there are at most two points in 
$X^r_{\lambda/\mu}(h(z))$ corresponding to $T$.  
\end{packedenum}
\end{lemma}

\begin{proof}
Consider a polynomial $h_\boldb(z)$ of the form
\[
    h_\boldb(z) = z^{|\mu|}\prod_{i=1}^\ell (z^r + b_ih_i)
\]
where $\boldb = (b_1, \dots, b_\ell) \in (\FF^\times)^\ell$.
We claim
that if $b_1, \dots, b_\ell$ are generic,
then statements (i) and (ii) are true with $h(z)$ replaced by $h_\boldb(z)$.

To prove the claim, we work over $\FF=\puiseux{v}$.  Let
\[
   \calU = \{(v^{\gamma_1}, \dots, v^{\gamma_\ell})
   \mid \gamma_k \in \ZZ,\ \gamma_1 > \dots > \gamma_\ell\}\,.
\]
Fix $\boldb = (v^{\gamma_1}, \dots, v^{\gamma_\ell}) \in \calU$ 
and a small positive rational number $\varepsilon > 0$.  Put
\[
   \hat h(z) := 
    z^{|\mu|}\prod_{i=1}^\ell(z^r + v^{\gamma_i}u^{\varepsilon\gamma_i}h_i)
     \,.
\]
Arguing as in the proof
of Lemma~\ref{lem:nongenericval}, we see that for any point
$\hat x \in X_{\lambda/\mu}(\hat h(z))$ such that
\begin{equation}
\label{eqn:valalmostribbon}
   \val(p_{\lambda_k}(\hat x)) 
   = \sum_{i=k+1}^\ell \val(h_i) + O(\varepsilon)
   \qquad\text{for }k = 0, \dots, \tfrac{\ell}{2}{-}1, 
                       \tfrac{\ell}{2}{+}1, \dots, \ell\,.
\end{equation}
there is a corresponding point $x^* := \Sigma(\hat x)$ such 
that $T_{x^*} = T$.  Conversely, if $\hat x$ does not satisfy
\eqref{eqn:valalmostribbon}, then $T_{x^*} \neq T$.  Points
$\hat x \in X_{\lambda/\mu}(\hat h(z))$ 
satisfying \eqref{eqn:valalmostribbon} are in bijection with
tableaux $T \in \ribbon(\lambda/\mu)$ whose associated chain
of partitions is of the form
\[
   \mu = \lambda_0 \subsetneq \dots 
    \lambda_{(\ell/2)-1} \subsetneq \lambda_{\ell/2} \subsetneq
    \lambda_{(\ell/2)+1} \subsetneq \dots \subsetneq 
      \lambda_\ell = \lambda\,,
\]
for some $\lambda_{\ell/2}$;
in other words, these points correspond to
extensions of the chain \eqref{eqn:chainmissingone}
to a maximal chain in $\Lambda^{r,\kappa}_{\lambda/\mu}$.  If $T$
has a long ribbon, there is a unique way to extend this chain, 
hence one point
in $X^r_{\lambda/\mu}(h_\boldb(z))$ corresponding to $T$.  
If $T$ does not have a long ribbon, there are two ways to extend the 
chain, hence two corresponding points.  This proves the claim.

Now, if $x^* \in X^r_{\lambda/\mu}(h_\boldb(z))$ is a point such 
that $T_{x^*} = T$, then by Lemma~\ref{lem:nongenericval}(2), we have
\begin{equation}
\label{eqn:ltalmostribbon}
   \leadterm(p_{\lambda_k}(x^*)) 
   = \frac{q_\lambda}{q_{\lambda_k}} \prod_{i=k+1}^\ell \leadterm(b_ih_i)
   \qquad\text{for }k = 0, \dots, \tfrac{\ell}{2}{-}1, 
                       \tfrac{\ell}{2}{+}1, \dots, \ell\,.
\end{equation}
The right hand side of \eqref{eqn:ltalmostribbon} is a 
nowhere-vanishing
function of $b_1, \dots b_\ell$.  This
implies that property \eqref{eqn:ltalmostribbon} is preserved when we 
take the limit as $\boldb \to (1, \dots, 1)$.
Thus in case (i) the limiting fibre $X^r_{\lambda/\mu}(h(z))$ contains 
exactly one reduced point $x$ such that $T_x = T$,
which proves (i).
In case (ii) the limiting fibre $X^r_{\lambda/\mu}(h(z))$ contains
either:
\begin{packedenum}
\item[(1)] two distinct reduced points $x, x'$ such 
that $T_x = T_{x'} = T$;\quad or
\item[(2)] a single point $x$ of multiplicity two such that $T_x = T$,
\end{packedenum}
which proves (ii).
\end{proof}

\begin{lemma}
\label{lem:twopoints}
Let $T$ be an almost-standard $r$-ribbon tableau, and
suppose $h(z)$ is of type (3).  
If $T$ does not have a long ribbon, then
there are exactly two points
$x, x' \in X^r_{\lambda/\mu}(h(z))$ such that $T_x = T_{x'} = T$,
and both are reduced points of $X^r_{\lambda/\mu}(h(z))$.
\end{lemma}

\begin{proof}
The proof uses essentially the same idea as the proof of 
Lemma~\ref{lem:existunique}.  First, we identify a 
vector $\boldalpha \in \FF^\Lambda$, such that one of the
points will lie in $Y_\boldalpha$.  Then we apply Hensel's
Lemma, to show that the point exists and is unique.

Let $\lambda_{\ell/2}$, and $\lambda'_{\ell/2}$ denote the two 
partitions of size $\tfrac{N}{2}$ that extend the
chain \eqref{eqn:chainmissingone} to a maximal chain.
Let $T^+ \in \ribbon(\lambda/\mu)$ be the tableau corresponding to the
extended chain of partitions
   $\lambda_0 \subsetneq \dots 
    \subsetneq \lambda_{\ell/2} \subsetneq \dots \subsetneq 
      \lambda_\ell$,
and let $\boldw$ be a weight vector obtained from $T^+$ and
the weakly decreasing sequence $\val(h_1), \dots, \val(h_\ell)$.
Then with $Y$ defined as in \eqref{eqn:defY}, we have $T_x = T$ for
every $x \in Y$: the proof
is analogous to that of Lemma~\ref{lem:Ytableau}.

By Lemma~\ref{lem:alphachain} there is a unique $\boldalpha \in \FF^\Lambda$ 
satisfying \eqref{eqn:alpha1}, \eqref{eqn:alpha2} and
\begin{align}
\label{eqn:alpha4}
   \alpha_{\lambda_k} 
   &= \frac{q_\lambda}{q_{\lambda_k}}\prod_{i=k+1}^\ell \leadcoeff(h_i)
   \qquad\text{for }k = 0, \dots, \tfrac{\ell}{2}{-}1, 
                       \tfrac{\ell}{2}{+}1, \dots, \ell  \\
\label{eqn:alpha5}
   \alpha_{\lambda_{\ell/2}} 
   &= 
   \epsilon_\circ\,q_\lambda
   \left(\frac{
   q_{\lambda'_{\ell/2}}}{
   q_{\lambda_{(\ell/2)-1}}
   q_{\lambda_{\ell/2}}
   q_{\lambda_{(\ell/2)+1}}
   }\right)^{\frac{1}{2}}
   \prod_{i=\frac{\ell}{2}+2}^\ell  \leadcoeff(h_i)\,,
\end{align}
where $\epsilon_\circ = (\epsilon_{\lambda_{(\ell/2)-1}}
   \epsilon_{\lambda_{\ell/2}}
   \epsilon_{\lambda_{(\ell/2)+1}}
   \epsilon_{\lambda'_{\ell/2}})^{1/2}$; we may choose either
square root.

Note that 
\[
   \lambda_{\ell/2} \wedge \lambda'_{\ell/2}
= \lambda_{(\ell/2)-1} \qquad\text{and} \qquad
\lambda_{\ell/2} \vee \lambda'_{\ell/2} = \lambda_{(\ell/2)+1}\,,
\]
so from \eqref{eqn:alpha2} we have
\begin{equation}
\label{eqn:twomiddlealphas}
\begin{aligned}
\alpha_{\lambda'_{\ell/2}}
&= (\epsilon_\circ)^2\, 
   \frac{ \alpha_{\lambda_{(\ell/2)-1}}
  \alpha_{\lambda_{(\ell/2)+1}} }
  {\alpha_{\lambda_{\ell/2}}} \\
&= 
   \epsilon_\circ\,q_\lambda
   \left(\frac{
   q_{\lambda_{\ell/2}}}{
   q_{\lambda_{(\ell/2)-1}}
   q_{\lambda'_{\ell/2}}
   q_{\lambda_{(\ell/2)+1}}
   }\right)^{\frac{1}{2}}
   \prod_{i=\frac{\ell}{2}}^\ell  \leadcoeff(h_i)  
   \,=\, - \frac{q_{\lambda_{\ell/2}}}{q_{\lambda'_{\ell/2}}} 
         \alpha_{\lambda_{\ell/2}}
 \,,
\end{aligned}
\end{equation}
where the last step uses $h_{\ell/2}=1$, $h_{(\ell/2)+1}=-1$.

For $j=0, \dots, N$, and $g \in \Kpol{N}$, let
\[
   L_j(g(z)) := \begin{cases}
          \leadterm\big([z^j]g(z)\big)
          & \quad \text{if }j=|\lambda_k|,\ %
          \val\big([z^j]g(z)\big) \leq \sum_{i=k+1}^\ell \val(h_i) \\
          0 & \quad \text{otherwise}\,.
   \end{cases}
\]
Note that 
$\val\big([z^{|\lambda_k|}]h(z)\big)\geq\sum_{i=k+1}^\ell \val(h_i)$,
with equality when $k \neq \frac{\ell}{2}$.  On the other hand, there
are two terms in the expansion of $[z^{N/2}]h(z)$ that have this
minimal valuation: 
\begin{multline}
\label{eqn:middletermh}
   \big[u^{\sum_{i=(\ell/2)+1}^\ell \val(h_i)}z^{N/2}\big]h(z) \\
   =
   \big[u^{\sum_{i=(\ell/2)+1}^\ell \val(h_i)}]
   \Big( \big(h_{\ell/2} h_{(\ell/2)+2} \dotsb h_\ell\big)\,+\,
   \big(h_{(\ell/2)+1} h_{(\ell/2)+2} \dotsb h_\ell\big) \Big)
   = 0\,,
\end{multline}
and so we have
\[
   L_j(h(z)) = \begin{cases} 
   \leadterm\big([z^j]h(z)\big)
   &\quad \text{if }j \neq \tfrac{N}{2} \\
   0 & \quad\text{if }j=\tfrac{N}{2}\,.
   \end{cases}
\]
Define
\[
   H := \big\{g(z) \in \Kpol{N} \bigmid %
         L_j(g(z)) = L_j(h(z)) 
             \text{ for }j=0, \dots, N\big\}\,.
\]
We now show that the $\Wr(Y_\boldalpha) \subset H$.  
Let
$x \in Y_\boldalpha$.  For $j \neq N/2$,
we have
\[
   L_j(\Wr(x;z)) = \leadterm\big([z^j]\Wr(x;z)\big)
   = \leadterm\big([z^j]h(z)\big) = L_j(h(z))\,;
\]
this is the same calculation as in the proof of Lemma~\ref{lem:existunique}.
For $j=N/2$, we have 
\[
  \val (p_{\lambda_{\ell/2}}(x)) = \val (p_{\lambda'_{\ell/2}}(x))
   = w_{\lambda_{\ell/2}}  
   = w_{\lambda'_{\ell/2}} 
   = \sum_{i=\frac{\ell}{2}+1}^\ell \val(h_i)\,,
\] 
and all other Pl\"ucker coordinates of $x$
indexed by partitions of size $N/2$ have strictly larger valuation.
Thus we have
\begin{alignat*}{2}
 L_{N/2}(\Wr(x;z))
  &= \sum_{\nu \vdash \frac{N}{2}} q_\nu 
     \big[u^{\sum_{i=(\ell/2)+1}^\ell\val(h_i)}\big] p_\nu(x) &\qquad&\\
  & = 
   q_{\lambda_{\ell/2}} 
   \big[u^{w_{\lambda_{\ell/2}}}\big]p_{\lambda_{\ell/2}}(x)
   +
   q_{\lambda'_{\ell/2}} 
   \big[u^{w_{\lambda'_{\ell/2}}}\big]p_{\lambda'_{\ell/2}}(x) \\
  & = 
   q_{\lambda_{\ell/2}} 
   \alpha_{\lambda_{\ell/2}}(x)
   +
   q_{\lambda'_{\ell/2}} 
   \alpha_{\lambda'_{\ell/2}}(x)  
    &&\text{(by \eqref{eqn:defYalpha})}\\
  &= 0  &&\text{(by \eqref{eqn:twomiddlealphas})}\,,
\end{alignat*}
and hence $\Wr(x;z) \in H$.

To complete the proof, we use Hensel's Lemma to show 
that $\Wr: Y_\boldalpha \to H$ is invertible.  First, we write the 
system of equations \eqref{eqn:inverseequations}
in terms of the local coordinates $y_1, \dots, y_\ell$ of 
Lemma~\ref{lem:Ylocalcoords}. 

For $k \neq \frac{\ell}{2}$, leading term of the left hand side 
and right hand side of \eqref{eqn:inverseequations} are given
by \eqref{eqn:ltinverselhs} and \eqref{eqn:ltinverserhs} 
respectively.
When $k=\frac{\ell}{2}$, the leading term of the left hand side
is
\begin{multline}
\label{eqn:middletermy}
   q_{\lambda_{\ell/2}} 
   \leadterm(p_{\lambda_{\ell/2}}(x))
   +
   q_{\lambda'_{\ell/2}} 
   \leadterm(p_{\lambda'_{\ell/2}}(x)) \\
   =\ %
   q_{\lambda_{\ell/2}} 
   \epsilon_{\lambda_{\ell/2}}
   \leadterm(y_{(\ell/2)+1}(x) y_{(\ell/2)+2}(x) \dotsb y_\ell(x))
   u^{w_{\lambda_{\ell/2}}} \\
   +
   q_{\lambda'_{\ell/2}} 
   \epsilon_{\lambda'_{\ell/2}}
   \leadterm(y_{(\ell/2)}(x) y_{(\ell/2)+2}(x) \dotsb y_\ell(x))
   u^{w_{\lambda'_{\ell/2}}}
   \,.
\end{multline}
(Here we have used $\leadterm(p_{\lambda'_{\ell/2}}(x)) 
   = \epsilon_{\lambda'_{\ell/2}}
\leadterm(y_{(\ell/2)}(x) y_{(\ell/2)+2}(x) \dotsb y_\ell(x))
   u^{w_{\lambda'_{\ell/2}}}$, 
which
comes from applying Lemma~\ref{lem:Ylocalcoords} with the chain of
partitions
\[
   \mu = \lambda_0 \subsetneq \dots \subsetneq
    \lambda_{(\ell/2)-1} \subsetneq \lambda'_{\ell/2} \subsetneq
    \lambda_{(\ell/2)+1} \subsetneq \dots \subsetneq 
      \lambda_\ell = \lambda\,.
\]
The local coordinates for this chain are the same as those for
the chain $\lambda_0 \subsetneq \dots \subsetneq \lambda_{\ell/2}
\subsetneq \dots \subsetneq \lambda_\ell$, but the roles of $y_\ell$
and $y_{\ell/2}$ are reversed.)
By \eqref{eqn:middletermh}, 
the right hand side of \eqref{eqn:inverseequations}
has valuation
strictly greater than $w_{\lambda_{\ell/2}}$ when $k= \frac{\ell}{2}$.
Thus we can write the system \eqref{eqn:inverseequations} as:
\begin{equation}
\label{eqn:firstorderinverse2}
\begin{gathered}
   y_{k+1} \dotsb y_\ell = \epsilon_{\lambda_k} \alpha_{\lambda_k} 
    + u^{\delta_k}G_k(y_1, \dots, y_\ell)
   \qquad\text{for }k \neq \tfrac{\ell}{2} \\
  \big(
   \epsilon_{\lambda'_{\ell/2}}
    q_{\lambda'_{\ell/2}} 
   y_{\ell/2}
  + 
   \epsilon_{\lambda_{\ell/2}}
   q_{\lambda_{\ell/2}} 
 y_{(\ell/2)+1} \big) 
  y_{(\ell/2)+2} \dotsb y_\ell 
   = 
     u^{\delta_{\ell/2}}G_{\ell/2}(y_1, \dots, y_\ell)
\end{gathered}
\end{equation}
where $\delta_k > 0$, and 
$G_k(y_0, \dots, y_\ell) \in \psK_+[y_1, \dots, y_\ell]$, for
$k=0, \dots, \ell$.
Let
\[
   F_k(y_1, \dots, y_\ell)
   :=
   \begin{cases}
   y_{k+1} \dotsb y_\ell &\quad\text{if }k \neq \frac{\ell}{2} \\
  \big(
   \epsilon_{\lambda'_{\ell/2}}
    q_{\lambda'_{\ell/2}} 
   y_{\ell/2}
  + 
   \epsilon_{\lambda_{\ell/2}}
   q_{\lambda_{\ell/2}} 
 y_{(\ell/2)+1} \big) 
  y_{(\ell/2)+2} \dotsb y_\ell 
   &\quad\text{if }k =\frac{\ell}{2}\,.
   \end{cases}
\]
We need to check two things.
\begin{packedenum}
\item[(1)]
Every $x \in Y_\boldalpha$ is a first order solution
to \eqref{eqn:firstorderinverse};
that is, for $x \in Y_\boldalpha$,
\[
  [u^0]F_k(y_1(x), \dotsb ,y_\ell(x)) = 
 \begin{cases}
   \alpha_{\lambda_k} & \quad\text{if }k \neq \frac{\ell}{2} \\
   0 & \quad\text{if }k = \frac{\ell}{2} \,.
\end{cases}
\]
For $k\neq \frac{\ell}{2}$, this follows from 
Lemma~\ref{lem:Ylocalcoords}(3).  For $k=\frac{\ell}{2}$, we calculate
\begin{alignat*}{2}
  \leadterm\big(F_k(y_1(x), &\dots ,y_\ell(x))\big) &\qquad&\\
   &= 
    \leadterm\big(q_{\lambda_{\ell/2}}p_{\lambda_{\ell/2}}(x) 
    u^{-w_{\lambda_{\ell/2}}} +
    q_{\lambda'_{\ell/2}}p_{\lambda'_{\ell/2}}(x)
    u^{-w_{\lambda_{\ell/2}}} 
    \big) 
    &&\text{(by \eqref{eqn:middletermy})} \\
   &= 
    \leadterm\big(q_{\lambda_{\ell/2}}\alpha_{\lambda_{\ell/2}}(x)  +
    q_{\lambda'_{\ell/2}}\alpha_{\lambda'_{\ell/2}}(x)
    \big)
    &&\text{(by \eqref{eqn:defYalpha})} \\
   &=  0
    &&\text{(by \eqref{eqn:twomiddlealphas})}\,. \\
\end{alignat*}
\item[(2)]
The Jacobian matrix 
$\big(\tfrac{\partial}{\partial y_j} F_i(y_1, \dots, y_\ell) \big)_{i,j=1, \dots, \ell}$ is invertible over $\psK_+$,
for $x \in Y_\boldalpha$.  This is a block upper triangular matrix
in which
most of the diagonal blocks are $1\times 1$ blocks and are 
invertible by Lemma~\ref{lem:Ylocalcoords}(3); however, there is one
$2 \times 2$ diagonal block (in rows $\frac{\ell}{2}$, $\frac{\ell}{2}{+}1$),
which is
\[
   \begin{pmatrix}
   \epsilon_{\lambda'_{\ell/2}}
    q_{\lambda'_{\ell/2}} 
  y_{(\ell/2)+2} \dotsb y_\ell 
  & 
   \epsilon_{\lambda_{\ell/2}}
   q_{\lambda_{\ell/2}} 
  y_{(\ell/2)+2} \dotsb y_\ell \\
  y_{(\ell/2)+1} 
  y_{(\ell/2)+2} \dotsb y_\ell 
  &
  y_{\ell/2} 
  y_{(\ell/2)+2} \dotsb y_\ell 
   \end{pmatrix}\,.
\]
This $2 \times 2$ matrix is invertible over $\psK_+$ if and only if
\[
   \epsilon_{\lambda'_{\ell/2}}
    q_{\lambda'_{\ell/2}} \leadcoeff(y_{\ell/2})
    \neq 
   \epsilon_{\lambda_{\ell/2}}
   q_{\lambda_{\ell/2}} 
   \leadcoeff(y_{(\ell/2)+1)})\,.
\]
But by Lemma~\ref{lem:Ylocalcoords}(3), this is equivalent to 
\[
    q_{\lambda'_{\ell/2}} 
    \alpha_{\lambda'_{\ell/2}} 
    \neq 
    q_{\lambda_{\ell/2}} 
    \alpha_{\lambda_{\ell/2}} \,,
\]
which follows from \eqref{eqn:twomiddlealphas}, as
 $\alpha_{\lambda_{\ell/2}} \neq 0$.  
Thus all diagonal blocks of the
Jacobian matrix are invertible over $\psK_+$, as required.
\end{packedenum}
Therefore, by Hensel's Lemma the system \eqref{eqn:firstorderinverse2} has a 
unique, reduced solution $(y_1, \dots, y_\ell) \in (\psK_+)^{\ell+1}$.
This solution gives the local coordinates of a point
$x \in Y_\boldalpha \cap X^r_{\lambda/\mu}(h(z))$; since $x \in Y$, we
have $T_x = T$.  We obtain the second point $x' \in X^r_{\lambda/\mu}(h(z))$,
by replacing $\epsilon_\circ$ by $-\epsilon_\circ$ in the definition 
\eqref{eqn:alpha5} of $\alpha_{\lambda_{\ell/2}}$.  This shows
that there are at least two reduced points in $X^r_{\lambda/\mu}(h(z))$
corresponding to $T$.  By Lemma~\ref{lem:oneortwopoints}(ii), these 
are the only such points.
\end{proof}

\begin{lemma}
\label{lem:longribbon}
Let $h(z)$ be a polynomial of type (3), and let 
$x \in X^r_{\lambda/\mu}(h(z))$.
If $x = x^\vee$ then $T_x$ 
has a long ribbon.
\end{lemma}

\begin{proof}
Suppose to the contrary that $x = x^\vee$ but the tableau $T = T_x$ 
does not have a long ribbon.
Let $\lambda_0, \dots, \lambda_\ell$ and $\boldalpha$ be as in the
proof of Lemma~\ref{lem:twopoints}.
By Lemma~\ref{lem:reflectrotate} we must have $T= T^\vee$, so
$(\lambda_k)^\vee = \lambda_{\ell-k}$, for $k \neq \frac{\ell}{2}$.
It is not hard to see that we must also have
$(\lambda_{\ell/2})^\vee = \lambda_{\ell/2}$.
By \eqref{eqn:reflectX}, $x = x^\vee$ implies that
\[
    \frac{p_{\lambda_k}(x)}
    {p_{(\lambda_k)^\vee}(x)}
\]
is constant for all $k$.  In particular we must have
\[
    \frac{p_{\lambda_{(\ell/2)-1}}(x)}
    {p_{(\lambda_{(\ell/2)+1}}(x)}
    =
    \frac{p_{\lambda_{\ell/2}}(x)}
    {p_{\lambda_{\ell/2}}(x)}
    = 1\,.
\]
But from \eqref{eqn:defYalpha} and \eqref{eqn:alpha4}, we have
\[
    \leadcoeff\left(\frac{p_{\lambda_{(\ell/2)-1}}(x)}
    {p_{(\lambda_{(\ell/2)+1}}(x)}\right)
    =
    \frac{\alpha_{\lambda_{(\ell/2)-1}}}
    {\alpha_{(\lambda_{(\ell/2)+1}}} 
    = \frac{q_{\lambda_{(\ell/2)+1}}
       \leadcoeff(h_{\ell/2})\leadcoeff(h_{(\ell/2)+1})}
         {q_{\lambda_{(\ell/2)-1}}} < 0\,,
\]
which gives a contradiction.
\end{proof}

We now obtain the second case of Theorem~\ref{thm:D-ribbon}.  

\begin{proof}[Proof of Theorem~\ref{thm:D-ribbon}(ii)]
If $h(z)$ is of type (3), then Lemmas~\ref{lem:oneortwopoints}
and~\ref{lem:twopoints} show that $X^r_{\lambda/\mu}(h(z))$ is reduced.
Suppose $x = x^\vee$.  Then $T_x$ must be an almost-standard
ribbon tableau (by Lemma~\ref{lem:almostribbon}) with
a long ribbon (by Lemma~\ref{lem:longribbon}), and be 
rotationally-invariant (by Lemma~\ref{lem:reflectrotate}).  
In other words, $T_x \in \longribbon(\lambda/\mu)$. 
Conversely, if $T_x \in \longribbon(\lambda/\mu)$, then 
$T_x = (T_x)^\vee = T_{x^\vee}$.  But by the uniqueness in
Lemma~\ref{lem:oneortwopoints}(i)
we have $x = x^\vee$.  Thus $\stdreflect$-fixed points in
$X^r_{\lambda/\mu}(h(z))$ are in bijection with 
$\longribbon(\lambda/\mu)$.
\end{proof}



\end{document}